\documentclass[12pt]{amsart} 
\usepackage{amssymb}
\usepackage{amsmath,amsfonts}
\usepackage{amsthm}
\usepackage{enumerate}
\usepackage{fullpage}
\usepackage{graphicx}
\usepackage{xcolor}
\usepackage{color}
\usepackage[toc]{appendix}
\usepackage[pdftex]{hyperref}
\usepackage{cleveref}

\newtheorem{theorem}{Theorem}[section]
\newtheorem{lemma}[theorem]{Lemma}
\newtheorem{proposition}[theorem]{Proposition}
\newtheorem{corollary}[theorem]{Corollary}
\newtheorem{remark}[theorem]{Remark}
\newtheorem{definition}[theorem]{Definition}

\numberwithin{equation}{section}

\newcommand{\HH}{H}

\newcommand{\definedas}{\mathrel{\raise.095ex\hbox{\rm :}\mkern-5.2mu=}}
\newcommand{\asdefined}{\mathrel{=\mkern-5.2mu\raise.095ex\hbox{\rm :}}}

\newcommand{\Ric}{\operatorname{Ric}}
\newcommand{\Rm}{\operatorname{Rm}}
\newcommand{\scal}{\operatorname{R}}
\newcommand{\R}{\mathbb{R}}
\newcommand{\N}{\mathbb{N}}
\newcommand{\diver}{\operatorname{div}}
\newcommand{\deter}{\operatorname{det}}
\newcommand{\crit}{\operatorname{Crit}f}
\newcommand{\supp}{\operatorname{supp}}

\begin{document}
\title{Uniqueness of static vacuum asymptotically flat black holes and equipotential photon surfaces\\ in $n+1$ dimensions \`a la Robinson}
	
\author{Carla Cederbaum$^{1,5}$, Albachiara Cogo$^{2,5}$, Benedito Leandro$^{3,6}$,\\ and Jo\~ao Paulo dos Santos$^{4,6}$}
\footnotetext[1]{\textsf{cederbaum@math.uni-tuebingen.de}}
\footnotetext[2]{\textsf{albachiara.cogo@math.uni-tuebingen.de}}
\footnotetext[3]{\textsf{bleandrone@mat.unb.br}}
\footnotetext[4]{\textsf{joaopsantos@unb.br}}
\footnotetext[5]{Mathematics Department, University of T\"ubingen, Germany}
\footnotetext[6]{Departamento de Matem\'atica, Universidade de Bras\'ilia, 70910-900, Bras\'ilia-DF, Brazil}

\date{}
	
\begin{abstract}
In this paper, we combine and generalize to higher dimensions the approaches to proving the uniqueness of connected $(3+1)$-dimensional static vacuum asymptotically flat black hole spacetimes by M\"uller zum Hagen--Robinson--Seifert and by Robinson. Applying these techniques, we prove and/or reprove geometric inequalities for connected $(n+1)$-dimensional static vacuum asymptotically flat spacetimes with either black hole or equipotential photon surface or specifically photon sphere inner boundary. In particular, assuming a natural upper bound on the total scalar curvature of the boundary, we recover and extend the well-known uniqueness results for such black hole and equipotential photon surface spacetimes. We also relate our results and proofs to existing results, in particular to those by Agostiniani--Mazzieri and by Nozawa--Shiromizu--Izumi--Yamada.
\end{abstract}

\maketitle{}

\vspace{-1.5ex}
	
	\noindent \emph{Mathematics Subject Classification (2020)}: 83C15, 53C17, 83C57, 53C24\\
	\noindent \emph{Keywords}: Black holes, static metrics, vacuum Einstein equation, equipotential photon surfaces, photon spheres, $D$-tensor, $T$-tensor.

\section{Introduction and results}\label{sec:intro}
	Black holes are among the most intriguing objects in nature and have captured the attention of researchers since Schwarz\-schild provided the first non-trivial solution of Einstein's equation of general relativity. Their properties and shape have since and continue to be thoroughly investigated. In the static case, it is well-established~\cite{israel,robinsonetal,robinson,bunting,miao,agostiniani,raulot} that the black hole solution found by Schwarz\-schild constitutes the only $3+1$-dimensional asymptotically flat static vacuum spacetime with an (a priori possibly disconnected) black hole horizon arising as its inner boundary. This fact is known as  ``static vacuum black hole uniqueness''; it also goes by the pictorial statement that ``static vacuum black holes have no hair''. We refer the interested reader to the reviews~\cite{Heusler,robinsonreview} for more information.
	
	In the higher dimensional case with spacetime dimension $n+1\geq3+1$, the analogous fact has also been asserted \cite{Hwang1,gibbons,ndimunique,ndimuniqueMFO,raulot,agostiniani,Nozawa}; however, all proofs make extra assumptions. The proofs by Hwang~\cite{Hwang1} and by Gibbons, Ida, and Shiromizu~\cite{gibbons} extend the method by Bunting and Masood-ul-Alam~\cite{bunting} allowing to deal with possibly disconnected horizons (see \cite{ndimunique,ndimuniqueMFO} for a more general version of this approach, and see~\cite{HollandsIshibashi} for a review of related results). These proofs rely on the rigidity case of the positive mass theorem and hence currently\footnote{but see~\cite{schoen,lohkamp}} make a spin assumption (using Witten's Dirac operator approach~\cite{Witten}) or impose an upper bound of $n+1\leq7+1$ on the spacetime dimension (using the minimal hypersurface approach by Schoen and Yau~\cite{SchoenYau2,SchoenYau}). Building on ideas by Walter Simon, Raulot \cite{raulot} exploits spinor techniques and thus explicitly makes a spin assumption. Instead, the proof by Agostiniani and Mazzieri~\cite{agostiniani} via potential theory, monotone functions, and a (conformal) splitting theorem assumes connectedness of the horizon as well as an upper bound on the total scalar curvature of the (time-slice) of the horizon, see also \Cref{subsec:AM}. Nozawa, Shiromizu, Izumi, and Yamada~\cite{Nozawa} derive the same statement as \cite{agostiniani} by a combination and generalization to higher dimensions of the divergence theorem based methods by M\"uller zum Hagen, Robinson, and Seifert~\cite{robinsonetal} and by Robinson~\cite{robinson}, see \Cref{subsec:Nozawa} for more details. Related results were recently presented in \cite{harvie2024quasisphericalmetricsstaticminkowski,medvedev2024staticmanifoldsboundarygeometry}.\\
	
\emph{The first main goal of this paper} is to give a rigorous new proof of static vacuum black hole uniqueness under the same geometric assumptions as Agostiniani--Mazzieri \cite[Theorem 2.8]{agostiniani}, but allowing for weaker decay assumptions, see \Cref{teoclassifica1} and \Cref{rem:decay}. Moreover, we reproduce all geometric inequalities for connected horizons proved in~\cite[Theorem 2.8]{agostiniani}, extend them to a wider class of parameters, and  identify a concrete relationship between our method and the approach taken in \cite{agostiniani}, see \Cref{sec:discussion}. We do so by combining, extending, and generalizing to higher dimensions the approaches by M\"uller zum Hagen, Robinson, and Seifert~\cite{robinsonetal} and by Robinson~\cite{robinson}. Our proof is rather similar to the derivation of the same statement by Nozawa, Shiromizu, Izumi, and Yamada~\cite[Section 5]{Nozawa} but allowing for weaker decay assumptions as well as filling in subtle analytic details, closing a gap in the uniqueness argument, and highlighting a connection to the analysis of Ricci solitons, see also \Cref{subsec:Nozawa}.

\begin{theorem}[Black Hole Uniqueness]\label{teoclassifica1}
Let $(M^n,g,f)$ be an asymptotically flat static vacuum system of mass $m\in\R$  and dimension $n\geq3$ with connected static horizon inner boundary $\partial M$. Let 
		\begin{align}\label{def:spartialM}
			s_{\partial M}\definedas\left(\frac{\vert\partial M\vert}{\vert\mathbb{S}^{n-1}\vert}\right)^{\frac{1}{n-1}}
		\end{align}
		denote the \emph{area radius of $\partial M$}, where $|\partial M|$ and $| \mathbb{S}^{n-1}|$ denote the surface area of $(\partial M,g_{\partial M})$ with respect to the induced metric $g_{\partial M}$ on $\partial M$ and of $(\mathbb{S}^{n-1},g_{\mathbb{S}^{n-1}})$, respectively. Then
		\begin{align}\label{maininequality}
			\frac{\left(s_{\partial M}\right)^{n-2}}{2} \sqrt{\frac{ \int_{\partial M} \scal_{\partial M} \,dS}{(n-1)(n-2) \, |\mathbb{S}^{n-1}| \left(s_{\partial M}\right)^{n-3} }} \geq m \geq \frac{\left(s_{\partial M}\right)^{n-2}}{2} ,
		\end{align}
		where $\scal_{\partial M}$ and $dS$ denote the scalar curvature and the hypersurface area element of $\partial M$ with respect to $g_{\partial M}$, respectively. In particular, $\partial M$ satisfies 
		\begin{align}\label{ineqA}
			\int_{\partial M} \scal_{\partial M} dS\geq  (n-1)(n-2) |\mathbb{S}^{n-1}| \left(s_{\partial M}\right)^{n-3}
		\end{align}
and $(M,g,f)$ has positive mass $m>0$. 
		
Moreover, equality holds on either side of \eqref{maininequality} and/or in \eqref{ineqA} if and only if $(M,g)$ is isometric to the Schwarz\-schild manifold $(M^{n}_{m},g_{m})$ of mass~$m$ and $f$ corresponds to the Schwarzschild lapse function $f_{m}$ under this isometry.
\end{theorem}
	
\begin{remark}[Black hole uniqueness follows from \Cref{teoclassifica1}]
The last statement gives the desired black hole uniqueness result subject to the scalar curvature bound condition 
\begin{align}\label{CondR}
\int_{\partial M} \scal_{\partial M} dS\leq  (n-1)(n-2) |\mathbb{S}^{n-1}| \left(s_{\partial M}\right)^{n-3}
\end{align}
 see also \Cref{rem:GB}. \Cref{teoclassifica1} implies several other interesting geometric inequalities such as a static version of the Riemannian Penrose inequality, see \cite{agostiniani,Mizuno,Nozawa} for more information.
\end{remark}
	
	Another recent direction of extending static vacuum black hole uniqueness results is to investigate uniqueness of spacetimes containing ``photon spheres'' (as introduced in~\cite{CVE}) or, more generally, ``photon surfaces'' (as introduced in~\cite{CVE,Perlick}). Here, \emph{photon surfaces} are timelike hypersurfaces of a spacetime which ``capture'' null geodesics; in static spacetimes, a photon surface is called \emph{equipotential} if the lapse function along it ``only depends on time'', and called a \emph{photon sphere} if the lapse function is (fully) constant along it (as introduced in~\cite{Ced}), see \Cref{sec:photo} and the references given there for definitions and more information. Photon surfaces  are relevant in gravitational lensing and in geometric optics, i.e., for trapping phenomena, and related to dynamical stability questions for black holes. 
	
	Photon spheres were first discovered in the $3+1$-dimensional Schwarz\-schild spacetimes of positive mass and persist in their higher dimensional analogs. (Equipotential) photon surfaces also naturally occur in Schwarzschild spacetimes of all dimensions and for all positive and negative masses, see~\cite{CedGalSurface,CJV}. (Asymptotically flat) static vacuum equipotential photon surface uniqueness is fully established in $3+1$ spacetime dimensions~\cite{Ced,carlagregpmt,CedGalSurface,CCF,raulot}. In particular, \cite{carlagregpmt,CedGalSurface,raulot} allow for combinations of black hole horizons and equipotential photon surfaces, assuming that all equipotential photon surface components are ``outward directed'', meaning that they have ``positive quasi-local mass'', see \Cref{rem:Smarrmass}. In contrast, Cederbaum, Cogo, and Fehrenbach~\cite{CCF} restrict to a connected, not necessarily outward directed equipotential photon surface, establishing uniqueness for the first time also in the negative and zero (total) mass cases. They generalize, exploit, and compare different techniques of proof, namely those from \cite{israel,Ced,agostiniani} and in particular Robinson's approach~\cite{robinson}.
	
	In higher dimensions $n+1\geq3+1$, the same result is established by Cederbaum and Galloway~\cite{CedGalSurface}, building on work by Cederbaum~\cite{ndimunique,ndimuniqueMFO} which uses the positive mass theorem; hence the ensuing restrictions discussed above apply. Raulot's spinorial approach~\cite{raulot} also covers higher dimensions, subject to a spin condition.\\[-1ex]
	
 The \emph{second main goal of this paper} is to demonstrate that the generalized divergence theorem based approach we derive can also be used to prove the expected uniqueness claim for connected equipotential photon surfaces, assuming the same upper bound on the total scalar curvature of the boundary as in the black hole case, see \Cref{teoclassifica2}. This generalizes the $3+1$-dimensional extension of Robinson's approach to connected equipotential photon surfaces by \cite{CCF}. Moreover, we prove similar geometric inequalities for connected equipotential photon surfaces as for black holes. Last but not least, we include the negative (total) mass case which has so far only been addressed in \cite{CCF} in $3+1$ dimensions. 
 
We do not address the zero mass case here. For $n=3$, the zero mass case and its connection to the Willmore inequality is established in \cite{CCF}. In higher dimensions, this requires extra considerations and is still work in progress.

\begin{theorem}[Equipotential Photon Surface Uniqueness]\label{teoclassifica2}
		Let $(M^n,g,f)$ be an asymptotically flat static vacuum system of mass $m\in\R$ and dimension $n\geq3$ with connected boundary $\partial M$ arising as a time-slice of an equipotential photon surface. Let $f_{0}>0$ denote the constant value of $f$ on $\partial M$ and assume that $f_{0}\neq1$. If $f_{0}\in(0,1)$ then
		\begin{align}\label{maininequality2}
				\frac{(1-f_0^2)\left(s_{\partial M}\right)^{n-2}}{2}\, \sqrt{ \frac{\left(s_{\partial M}\right)^{2}\left(\scal_{\partial M} - \frac{n-2}{n-1} \HH^2\right)}{(n-1)(n-2)(1-f_{0}^{2})}}\geq m \geq \frac{(1-f_0^2)\left(s_{\partial M}\right)^{n-2}}{2}.
		\end{align}
		
		Here, $\scal_{\partial M}$, $\HH$, and $s_{\partial M}$ denote the scalar curvature, the mean curvature, and the surface area radius \eqref{def:spartialM} of $\partial M$ with respect to the induced metric $g_{\partial M}$ on $\partial M$, respectively. In particular, $\partial M$ satisfies 
		\begin{align}\label{implicationphoto}
			\scal_{\partial M}-\tfrac{n-2}{n-1}\HH^{2} \geq \frac{(n-1)(n-2)(1-f_{0}^{2})}{\left(s_{\partial M}\right)^{2}}
		\end{align}
		and $(M,g,f)$ has positive mass $m>0$. If $f_{0}\in(1,\infty)$ then
		\begin{align}\label{maininequality2b}
			\begin{split}
				\frac{(1-f_0^2)\left(s_{\partial M}\right)^{n-2}}{2}\, \sqrt{ \frac{\left(s_{\partial M}\right)^{2}\left(\scal_{\partial M} - \frac{n-2}{n-1} \HH^2\right)}{(n-1)(n-2)(1-f_{0}^{2})}}\leq m \leq \frac{(1-f_0^2)\left(s_{\partial M}\right)^{n-2}}{2}.
			\end{split}
		\end{align}
		In particular, $\partial M$ satisfies 
		\begin{align}\label{implicationphotob}
			\scal_{\partial M}-\tfrac{n-2}{n-1}\HH^{2} \leq  \frac{(n-1)(n-2)(1-f_{0}^{2})}{\left(s_{\partial M}\right)^{2}}
		\end{align}
		and $(M,g,f)$ has negative mass $m<0$. Moreover, for any $f_{0}\in(0,1)\cup(1,\infty)$, if
		\begin{align}\label{AgoMazzCond}
			\scal_{\partial M} \leq  \frac{(n-1)(n-2)}{\left(s_{\partial M}\right)^{2}}
		\end{align}
		then $(M,g)$ is isometric to the piece $[s_{\partial M},\infty)\times\mathbb{S}^{n-1}$ of the Schwarz\-schild manifold $(M^{n}_{m},g_{m})$ of mass~$m$ and $f$ corresponds to the restriction of the Schwarzschild lapse function $f_{m}$ to $[s_{\partial M},\infty)$ under this isometry.
	\end{theorem}
	The last statement gives the desired equipotential photon surface uniqueness result subject to a scalar curvature bound condition, see also \Cref{rem:GB}. \Cref{teoclassifica2} implies several other interesting geometric inequalities, see \cite{agostiniani} for more information.
	
	\begin{remark}[About Conditions~\eqref{CondR} and \eqref{AgoMazzCond}]\label{rem:GB}
		Note that \eqref{CondR} and \eqref{AgoMazzCond} are equivalent in case $\scal_{\partial M}=\text{const}$ (as it is the case for time-slices of equipotential photon surfaces, see \Cref{prop:photo}). In dimension $n = 3$, conditions~\eqref{CondR} and \eqref{AgoMazzCond} are of course automatically satisfied by the Gau\ss--Bonnet theorem. Hence \Cref{teoclassifica1} gives static vacuum black hole uniqueness in $3$ dimensions without extra assumptions (other than connectedness of the static horizon) and \Cref{teoclassifica2} gives static vacuum equipotential photon surface uniqueness in $3$ dimensions without extra assumptions (other than connectedness of the photon surface), including the negative (total) mass case.
	\end{remark}
	
	To prove \Cref{teoclassifica1,teoclassifica2}, we proceed as follows: First, in \Cref{sec:proofs}, we derive the following higher dimensional version of Robinson's identity \cite[Equation (2.3)]{robinson}, using the so-called \emph{$T$-tensor} instead of the Cotton tensor $C$ used by Robinson~\cite{robinson}. We also introduce an additional parameter $p\in\R$ as a power into the identity, with $p=3$ corresponding to Robinson's identity, and $p=\frac{3}{2}$ corresponding to the approach taken by M\"uller zum Hagen, Robinson, and Seifert~\cite{robinsonetal}. Similar parameters called $p$ and $c$, respectively, were introduced in \cite{agostiniani} and in \cite{Nozawa}; we refer the reader to \Cref{sec:discussion} for a discussion of the relation between the parameter $p$ and its range and the parameters $p$ and $c$ from \cite{agostiniani,Nozawa}.
	
	\begin{theorem}[Generalized Robinson identity]\label{maintheoremp}
		Let $(M^n,g,f)$, $n\geq3$, be a static vacuum system with $0<f<1$ or $f>1$ in $M$. Then, for all $c,d,p\in\R$, the \emph{generalized Robinson identity}
		\begin{align}
			\begin{split}\label{mainformulap}
				&\|\nabla f \|^2\,\diver\left(\frac{F(f)}{f}\Vert\nabla f\Vert^{p-3}\,\nabla\|\nabla f\|^{2}+G(f)\|\nabla f\|^{p-1}\,\nabla f\right)  \\
				&=\Vert\nabla f\Vert^{p-3}\,F(f)\left[\frac{(n-2)^{2}f}{(n-1)^{2}}\|T\|^{2}+\frac{p-p_{n}}{2f}\left\|\nabla\|\nabla f\|^{2}+\frac{4(n-1)}{(n-2)}\frac{f\|\nabla f\|^{2}\,\nabla f}{1-f^{2}}\right\|^{2}\right]
			\end{split}
		\end{align}
		holds on $M\setminus\crit$, with $\crit\definedas\{q\in M\,\vert\,\Vert\nabla f\Vert_{q}=0\}$ denoting the set of critical points of~$f$ and the constant $p_{n}$ is given by 
\begin{equation}\label{eq:thresholdp}
 p_{n}\definedas 2-\frac{1}{n-1}.
\end{equation}
 Here, $F,G\colon[0,1)\cup(1,\infty)\to\R$ are given by
		\begin{align}\label{eq:F}
			F(t)&\definedas\frac{ct^{2}+d}{\vert 1-t^{2}\vert^{\frac{(n-1)(p-1)}{n-2}-1}},\\\label{eq:G}
			G(t)&\definedas\frac{4\left(\frac{(n-1)(p-1)}{n-2}-1\right)F(t)}{(p-1)(1-t^{2})}-\frac{4c}{(p-1)\vert 1-t^{2}\vert^{\frac{(n-1)(p-1)}{n-2}-1}}
		\end{align}
for $t\in[0,1)\cup(1,\infty)$, $\Vert\cdot\Vert$, $\nabla$, and $\diver$ denote the tensor norm, covariant derivative, and covariant divergence with respect to $g$. The tensor $T$ is given by
		\begin{align}
			\begin{split}\label{tensorT}
				T (X, Y, Z) \definedas& \frac{n-1}{n-2} \left(\Ric(X, Z) \nabla_{\!Y}f - \Ric(Y, Z) \nabla_{\!X}f\right)\\
				& - \frac{1}{n-2} \left( \Ric(X, \nabla f) g(Y, Z) - \Ric(Y, \nabla f) g(X, Z)\right),
			\end{split}
		\end{align}
		for $X,Y,Z\in\Gamma(TM)$, where $\Ric$ denotes the Ricci curvature tensor of $(M, g)$.

Moreover, if $p\geq3$, the divergence on the left hand side continuously extends to $\crit$ and \eqref{mainformulap} holds on $M$. Furthermore,  if $p\geq p_{n}$, it follows that
\begin{align}\label{mainformulap0}
\diver\left(\frac{F(f)}{f}\Vert\nabla f\Vert^{p-3}\,\nabla\|\nabla f\|^{2}+G(f)\|\nabla f\|^{p-1}\,\nabla f\right) \geq0
\end{align}
on $M\setminus\crit$ provided that $F(f)\geq0$.
	\end{theorem}

	\Cref{maintheoremp} reproduces Robinson's identity \cite[(2.3)]{robinson} when $n=3$, $p=3$, and $0<f<1$, and its generalization to the negative mass case by Cederbaum, Cogo, and Fehrenbach~\cite{CCF} when $n=3$, $p=3$, and $f>1$. When $n=3$, $p=p_{2}=\frac{3}{2}$, and $0<f<1$, \eqref{mainformulap} is very closely related to the divergence identities derived by M\"uller zum Hagen, Robinson, and Seifert \cite{robinsonetal}. 
	
\begin{remark}[Generalizations]
The divergence identity \eqref{mainformulap} may be of independent interest, allowing to prove geometric inequalities for more general boundary geometries than the level set boundaries we are interested in this work. As it is purely local, it may also be of use to prove related results in different asymptotic scenarios such as ALE spaces.
\end{remark}

The $T$-tensor\label{p:Ricci} introduced in \eqref{tensorT} is specifically adapted to the geometry of static vacuum systems, see \Cref{sec:Ttensor}. As $\scal=0$ in static vacuum systems, it formally coincides\footnote{up to a factor $n-1$, and with a different function $f$} with the $D$-tensor introduced for the analysis and classification of Ricci solitons by Cao and Chen \cite{Cao,cao2013}, inspired by Israel's~\cite{israel} and in particular Robinson's~\cite{robinson} approaches to proving black hole uniqueness. Both the $D$-tensor and the $T$-tensor have seen many applications in classification problems for Ricci solitons and quasi-Einstein manifolds.\\[-1ex]
	
	As the next step in proving \Cref{teoclassifica1,teoclassifica2}, we will exploit \Cref{maintheoremp} to prove some important geometric inequalities on $\partial M$. These inequalities can be stated in a parametric way (\Cref{maintheoremp2}), or, equivalently, as two separate inequalities (\Cref{maintheoremp3}). Both versions of the geometric inequalities and their equivalence will be proven in \Cref{sec:inequalities}. The parametric geometric inequalities in \Cref{maintheoremp} have also been established by Agostiniani and Mazzieri~\cite{agostiniani} for $p\geq3$. To the best knowledge of the authors, they are new for $3>p\geq p_{n}$.

	\begin{theorem}[Parametric geometric inequalities]\label{maintheoremp2}
		Let $(M^n,g,f)$, $n\geq3$, be an asymptotically flat static vacuum system of mass $m\in\R$ with connected boundary $\partial M$. Assume that $f\vert_{\partial M} = f_0$ for a constant $f_{0}\in[0,1)\cup(1,\infty)$ and that the normal derivative $\nu(f)\vert_{\partial M}\asdefined \kappa$ is constant, with unit normal $\nu$ pointing towards the asymptotic end. Let $F$ and $G$ be as in \Cref{maintheoremp} for some $p\geq p_{n}$ and some constants $c,d\in\R$ satisfying $c+d\geq0$ and $cf_{0}^{2}+d\geq0$. Set $F_{0}\definedas F(f_{0})$, $G_{0}\definedas G(f_{0})$. Then 
		\begin{align}\label{mainformulap2}
			F_0\, \kappa^{p-2} \int_{\partial M} \left(\scal_{\partial M} -  \tfrac{n-2}{n-1}\HH^2+\Vert\mathring{h}\Vert^{2} \right)dS  - G_0\,\kappa^p |\partial M |\geq\mathcal{F}^{\,c,d}_{\!p}(m),
		\end{align}
		and $\kappa,m>0$ if $f_{0}\in[0,1)$ and 
		\begin{align}\label{mainformulap2b}
			F_0\, \vert\kappa\vert^{p-2} \int_{\partial M} \left(\scal_{\partial M} -  \tfrac{n-2}{n-1}\HH^2+\Vert\mathring{h}\Vert^{2} \right)dS  - G_0\,\vert\kappa\vert^p |\partial M |\leq-\mathcal{F}^{\,c,d}_{\!p}(m)
		\end{align}
		and $\kappa,m<0$ if $f_{0}\in(1,\infty)$. Here, $\scal_{\partial M}$, $\HH$, $\mathring{h}$, and $dS$ denote the scalar curvature, the mean curvature, the trace-free part of the second fundamental form, and the area element of $\partial M$, and $|\partial M|$ and $| \mathbb{S}^{n-1}|$ denote the area of $(\partial M,g_{\partial M})$ and of $(\mathbb{S}^{n-1},g_{\mathbb{S}^{n-1}})$, respectively. The constant $\mathcal{F}^{c,d}_{p}(m)\in\R$ is given by
		\begin{align}\label{eq:Fmcd}
			\mathcal{F}^{\,c,d}_{\!p}(m)&\definedas\frac{4(n-2)^{p}}{2^{\frac{(n-1)(p-1)}{n-2}}(p-1)}\vert\mathbb{S}^{n-1}\vert (c+d)\vert m\vert^{p-\frac{(n-1)(p-1)}{n-2}}.
		\end{align}
Unless $c=d=0$, equality holds in \eqref{mainformulap2} or in \eqref{mainformulap2b} if and only if $(M,g)$ is isometric to a suitable piece of the Schwarz\-schild manifold $(M^{n}_{m},g_{m})$ of mass~$m$ and $f$ corresponds to the corresponding restriction of the Schwarzschild lapse function $f_{m}$ under this isometry.
	\end{theorem}
		
	This can equivalently be expressed as follows.

	\begin{theorem}[Geometric inequalities]\label{maintheoremp3}
		Let $(M^n,g,f)$, $n\geq3$, be an asymptotically flat static vacuum system of mass $m\in\R$ with connected boundary $\partial M$. Assume that $f\vert_{\partial M} = f_0$ for a constant $f_{0}\in[0,1)\cup(1,\infty)$ and that $\nu(f)\vert_{\partial M}\asdefined \kappa$ is constant. Then 
		\begin{align}\label{maininequality2h}
			\begin{split}
				\frac{(1-f_0^2)\left(s_{\partial M}\right)^{n-2}}{2} & \sqrt{\frac{ \int_{\partial M} \left(\scal_{\partial M} - \frac{n-2}{n-1} \HH^2+\Vert\mathring{h}\Vert^{2}\right) dS}{(n-1)(n-2)(1-f_{0}^{2}) |\mathbb{S}^{n-1}| \left(s_{\partial M}\right)^{n-3}}}\\
				&\quad\geq m \geq  \frac{(1-f_0^2)\left(s_{\partial M}\right)^{n-2}}{2}.
			\end{split}
		\end{align}
		holds if $f_{0}\in[0,1)$ and 
		\begin{align}\label{maininequality2bh}
			\begin{split}
				\frac{(1-f_0^2)\left(s_{\partial M}\right)^{n-2}}{2} & \sqrt{\frac{ \int_{\partial M} \left(\scal_{\partial M} - \frac{n-2}{n-1} \HH^2+\Vert\mathring{h}\Vert^{2}\right) dS}{(n-1)(n-2)(1-f_{0}^{2}) |\mathbb{S}^{n-1}| \left(s_{\partial M}\right)^{n-3}}}\\
				&\quad\leq m \leq \frac{(1-f_0^2)\left(s_{\partial M}\right)^{n-2}}{2}
			\end{split}
		\end{align}
		holds if $f_{0}\in(1,\infty)$. Equality holds on either side in each of \eqref{maininequality2h}, \eqref{maininequality2bh} if and only if $(M,g)$ is isometric to a suitable piece of the Schwarz\-schild manifold $(M^{n}_{m},g_{m})$ of mass~$m$ and $f$ corresponds to the restriction of the Schwarzschild lapse function $f_{m}$ under this isometry.
\end{theorem}
	
The equality case assertions in \Cref{maintheoremp2,maintheoremp3} and thus in \Cref{teoclassifica1,teoclassifica2} rely on the following rather general rigidity theorem which we will prove in \Cref{sec:Ttensor,sec:rigidity}.

	\begin{theorem}[Rigidity theorem]\label{thm:rigidity}
		Let $(M^n,g,f)$, $n\geq3$, be an asymptotically flat static vacuum system of mass $m\in\R$ with connected boundary $\partial M$. Assume that $f\vert_{\partial M} = f_0$ for a constant $f_{0}\in[0,1)\cup(1,\infty)$. Assume that $T=0$ on $M$. Then $(M,g)$ is isometric to the piece $[s_{\partial M},\infty)\times\mathbb{S}^{n-1}$ of the Schwarz\-schild manifold $(M^{n}_{m},g_{m})$ of mass~$m$ and $f$ corresponds to the restriction of $f_{m}$ to $[s_{\partial M},\infty)$ under this isometry. In particular $\partial M$ is totally umbilic, has constant mean curvature, and is isometric to a round sphere.
	\end{theorem}
	
\Cref{teoclassifica1,teoclassifica2} then follow directly from \Cref{maintheoremp2,maintheoremp3,thm:rigidity} as we will show towards the end of \Cref{sec:inequalities}.

\begin{remark}[Independent interest]
\Cref{maintheoremp3,thm:rigidity} may be of independent interest as they assume much less about the properties of $\partial M$ than \Cref{teoclassifica1,teoclassifica2}. Thus, similar geometric inequalities may be derived from \Cref{maintheoremp} under different asymptotic and/or inner boundary conditions.
\end{remark}
	
	Having completed the proofs of \Cref{teoclassifica1,teoclassifica2}, we will then discuss some geometric implications as well as the relation to the existing strategies of proving \Cref{teoclassifica1} implemented by Agostiniani and Mazzieri~\cite{agostiniani} and put forward by Nozawa, Shiromizu, Izumi, and Yamada in \cite{Nozawa} in \Cref{sec:discussion}. In particular, we will define monotone functions $\mathcal{H}^{c,d}_{p}$ along the level sets of the lapse function $f$ in the style of the functions $U_{p}$ introduced in \cite{agostiniani} and relate $\mathcal{H}^{c,d}_{p}$ to $U_{p}$ (see \Cref{subsec:AM}). This will shed light on the relation of the two proofs and extend the monotonicity results of \cite{agostiniani} to $3>p\geq p_n$. In \Cref{subsec:Nozawa}, we will investigate the relationship between the $(0,3)$ $T$-tensor we use in our approach to the $(0,2)$-tensor $H$ used by \cite{Nozawa}; in particular we will show that vanishing of $T$ does \emph{not} locally imply vanishing of $H$ as is claimed in \cite{Nozawa} and as is necessary to conclude for $p=p_n$. Moreover, we will discuss how our analysis completes the strategy of proof put forward in~\cite{Nozawa}.\\[-2ex]
		
\subsection*{This paper is structured as follows:}
In \Cref{sec:prelims}, we introduce our notation and definitions, in particular the precise notion of asymptotic flatness we are using. We also collect some straightforward and/or well-known facts about static horizons and equipotential photon surfaces. In \Cref{sec:Ttensor}, we prove useful facts about the $T$-tensor which could also be of independent interest, while in \Cref{sec:rigidity}, we will show how these facts imply \Cref{thm:rigidity}. In \Cref{sec:proofs}, we will give a proof of \Cref{maintheoremp}. In \Cref{sec:inequalities}, we will prove \Cref{maintheoremp2,maintheoremp3} and show how they imply \Cref{teoclassifica1,teoclassifica2}. The final \Cref{sec:discussion} is dedicated to deducing and discussing consequences of \Cref{teoclassifica1,teoclassifica2}, in particular to constructing monotone functions along the level sets of the lapse function $f$ and comparing those to the monotone functions introduced and exploited in \cite{agostiniani} and to a comparison between our tensor $T$ and the tensor $H$ used by \cite{Nozawa}.\\[-2ex]
	
	\subsection*{Acknowledgements}
	The authors thank Klaus Kr\"oncke for helpful comments and questions and are indebted to Tetsuya Shiromizu for pointing us to the very inspiring reference \cite{Nozawa} after we had shared the first version of this article as a preprint \cite{arxiv}. The work of Carla Cederbaum was funded by the Deutsche Forschungsgemeinschaft (DFG, German Research Foundation) -- 441897040. Albachiara Cogo is thankful to Universidade de Bras\'ilia and Universidade Federal de Go\'ias, where part of this work has been carried out.
	Benedito Leandro was partially supported by CNPq Grant 403349/2021-4 and 303157/2022-4. Jo\~ao Paulo dos Santos was partially supported by CNPq Grant CNPq 315614/2021-8.
	
	\section{Preliminaries}\label{sec:prelims}
	In this section, we will collect all relevant definitions as well as some straightforward and/or well-known facts useful for the proofs of \Cref{teoclassifica1,teoclassifica2} and \Cref{maintheoremp,maintheoremp2,maintheoremp3,thm:rigidity}. Our sign and scaling convention for the mean curvature $\HH$ of a smooth, oriented hypersurface of $(M,g)$ is such that the unit round sphere $\mathbb{S}^{n-1}$ in $(\R^{n},\delta)$ has mean curvature $\HH=n-1$ with respect to the unit normal $\nu$ pointing towards infinity.
	
	\subsection{Static vacuum systems and asymptotic considerations}\label{sec:asymptotics}
	\begin{definition}[Static vacuum systems]\label{def1}
		A smooth, connected Riemannian manifold $(M^n,g)$, $n\geq3$, is called a \emph{static system} if there exists a smooth \emph{lapse function} $f\colon M\rightarrow(0,+\infty)$. A static system is called a \emph{static vacuum system} if it satisfies the \emph{static vacuum equations}
		\begin{align}\label{vacuum equation1}
			\nabla^{2}f&=f\Ric,\\\label{vacuum equation2}
			\Delta f&=0,
		\end{align}
		where $\nabla^{2}$ and $\Delta$ denote the Hessian and Laplacian with respect to $g$, respectively, and $\Ric$ denotes its Ricci curvature tensor. If $M$ has non-empty boundary $\partial M$, it is assumed that $g$ and $f$ extend smoothly to $\partial M$, with $f\geq0$ on $\partial M$.
	\end{definition}
	
	It follows readily from the trace of \eqref{vacuum equation1} and from \eqref{vacuum equation2} that the scalar curvature $\scal$ of a static vacuum system $(M^{n},g,f)$ vanishes,
	\begin{align}\label{scal0}
		\scal=0.
	\end{align}
	
	It can easily be seen that the warped product \emph{static spacetime} $(\R\times M,\overline{g}=-f^{2}dt^{2}+g)$ constructed\footnote{In case  $f=0$ on $\partial M$, one usually assumes that $\overline{g}$ smoothly extends to the boundary of $\R\times M$ (although of course the warped product structure breaks down there).} from a static vacuum system $(M^{n},g,f)$ satisfies the vacuum Einstein equation $\overline{\Ric}=0$, with $\overline{\Ric}$ denoting the Ricci curvature tensor of $\overline{g}$. Conversely, a static spacetime $(\R\times M^{n},\overline{g}=-f^{2}dt^{2}+g)$ solving the vacuum Einstein equation has time-slices $\left\{t=\text{const.}\right\}$ isometric to $(M,g)$ with lapse function $f\colon M\to(0,\infty)$ such that $(M,g,f)$ is a static vacuum system.
	
	The prime example of a static vacuum system is the \emph{$n$-dimensional Schwarz\-schild\footnote{In higher dimensions, the associated static spacetimes are also known as Schwarz\-schild--Tangher\-lini spacetimes.} system $(M^{n}_{m},g_{m},f_{m})$ of mass $m>0$ and dimension $n\geq3$}, given by 
	\begin{align}
		\begin{split}\label{Schw}
			f_{m}(r) &= \sqrt{1-\frac{2m}{r^{n-2}}},\\
			g_{m}&=\frac{dr^{2}}{f_{m}(r)^{2}}+r^{2}g_{\mathbb{S}^{n-1}},
		\end{split}
	\end{align}
	on $M^n_m= ((2m)^{\frac{1}{n-2}},\infty) \times \mathbb{S}^{n-1}$, where $g_{\mathbb{S}^{n-1}}$ denotes the canonical metric on $\mathbb{S}^{n-1}$ and $r \in ((2m)^{\frac{1}{n-2}},\infty)$ is the \emph{radial coordinate}. It is well-known that by a change of coordinates (e.g.\ to ``isotropic coordinates''), one can assert that $g_{m}$ and $f_{m}$ smoothly extend to $\partial M_{m}=\{r=(2m)^{\frac{1}{n-2}}\}\times\mathbb{S}^{n-1}$, with induced metric $g_{\partial M_{m}}=(2m)^{\frac{2}{n-2}}g_{\mathbb{S}^{n-1}}$ and $f_{m}=0$ on $\partial M_{m}$. Moreover, by another change of coordinates (e.g.\ to ``Kruskal--Szekeres coordinates''), one can smoothly extend the associated $(n+1)$-dimensional static Schwarz\-schild spacetime $(\R\times M_{m},\overline{g}_{m}=-f_{m}^{2}dt^{2}+g_{m})$ to include (and indeed extend beyond) the boundary of $\R\times M_{m}$. Similarly, the \emph{$n$-dimensional Schwarzschild system of mass $m\leq0$} is given by \eqref{Schw} on $M^n_m= (0,\infty) \times \mathbb{S}^{n-1}$; the associated spacetime cannot be extended when $m<0$ and isometrically embeds into the Minkowski spacetime when $m=0$.
	
	We will use the following weak notion of asymptotic flatness.
	\begin{definition}[Asymptotic flatness]\label{def2}
		A static  system $(M^{n},g,f)$, $n\geq3$, is said to be \emph{asymptotically flat of mass $m\in\R$ (and decay rate $\tau\geq0$)} if there exist a  \emph{mass (parameter)} $m\in\R$ as well as a compact subset $K \subset M$ and a smooth diffeomorphism $x \colon M \setminus K \longrightarrow \R^n \setminus \overline{B}$ for some open ball $B$ such that, in the coordinates $(x^{i})$ induced by the diffeomorphism $x$,
		\begin{enumerate}[i)]
			\item  the metric components $g_{ij}$ satisfy the decay conditions
			\begin{align}\label{eq:g-asymp}
				(x_{*}g)_{ij} &= \delta_{ij} + o_2(|x|^{-\tau})
			\end{align}
			as $|x|\rightarrow \infty$ for all $i,j \in \left\{1,\dots,n \right\}$, and
			\item the lapse function $f$ can be written as 
			\begin{align}\label{eq:f-asymp}
				f\circ x^{-1} &= f_{m}(\vert x\vert)+ o_2(|x|^{-(n-2)})=1-\frac{m}{|x|^{n-2}} + o_2(|x|^{-(n-2)})
			\end{align}
			as $|x|\rightarrow \infty$.
		\end{enumerate}
		
		Here and throughout the paper, for a given smooth function $\Psi \colon \R^{n} \rightarrow \R$, the notation $\Psi = o_l (|x|^{\alpha})$ for some $l \in \mathbb{N}$, $\alpha \in \R$ means that
		\begin{equation}\label{def: little_o}
			\sum_{|J|\leq l} |x|^{\alpha+|J|}|\partial^J \Psi| = o(1)
		\end{equation}
		as $|x|\rightarrow \infty$. The meaning of the notation $\Psi = O_l (|x|^{\alpha})$ is analogous, substituting $O(1)$ by $o(1)$ in \Cref{def: little_o}. For improved readability, we will from now on mostly suppress the explicit mention of the diffeomorphism $x$ in our formulas.
	\end{definition}
	
	\begin{remark}[Asymptotic assumptions, decay rates]\label{rem:decay}
		\Cref{teoclassifica1,teoclassifica2} and \Cref{maintheoremp2,maintheoremp3,thm:rigidity} apply for all decay rates $\tau\geq0$, in particular for $\tau=0$, which is why we do not explicitly mention the decay rates in their statements. See also \Cref{rem:admissibledecay} for further information on possible decay rates.
		
		In the standard definition of asymptotic flatness for Riemannian manifolds, one usually requires stronger asymptotic conditions, namely $g_{ij}=\delta_{ij}+O_2(|x|^{-\frac{1}{2}-\varepsilon})$ for some $\varepsilon>0$ and integrability of the scalar curvature $\scal$ on $M$ (which is automatic in the static setting). Under these additional assumptions, it can be seen by a standard computation that the mass parameter $m$ from \eqref{eq:f-asymp} coincides with the ADM-mass of $(M,g)$. We do not appeal to any facts or properties of the ADM-mass, so we don't need to impose such restrictions.
		
		Our decay assumptions are also very weak when compared with the other static vacuum uniqueness results discussed in \Cref{sec:intro}. Most of these results require that $(M,g,f)$ is asymptotic to the Schwarzschild system of mass $m$, implying standard asymptotic flatness with $\varepsilon=\frac{1}{2}$ and also faster decay of the error term in \eqref{eq:f-asymp}. In contrast, \cite{agostiniani,CCF} make the same assumption on the decay of $f$ as we make in \eqref{eq:f-asymp}. On the other hand, \cite{agostiniani} assumes $\tau=\frac{n-2}{2}$; but see \cite[Remark 7.1]{CCF} and \Cref{subsec:AM}. Instead, \cite{Nozawa} assumes  Schwarzschildean decay which gives $\tau$ arbitrarily close to $1$ from below. However, all their asymptotic arguments are adapted to $\tau\geq0$ here. Finally, it is conceivable that our asymptotic decay assumptions can be boot-strapped to stronger decay assertions as e.g.\ in~\cite{KM}, using \eqref{vacuum equation1} and~\eqref{vacuum equation2}.
\end{remark}
	
	It is well-known and straightforward to see that the Schwarz\-schild system $(M^{n}_{m},g_{m},f_{m})$ of mass $m$ is asymptotically flat of mass $m$ for any decay rate $\tau\geq0$. To see this, one switches from the spherical polar coordinates $r$ and $\eta\in\mathbb{S}^{n-1}$ to the canonically associated Cartesian coordinates $x=r\eta$ outside a suitably large ball.
	
	The following remark will be useful for our strategy of proof, in particular for \Cref{maintheoremp2}, where we will use it when applying the divergence theorem on $M$, and for \Cref{thm:rigidity}, where we will use it to properly study the level set flow of $f$ and conclude isometry to a Schwarzschild system.
	\begin{remark}[Completeness]\label{rem:complete}
		Asymptotically flat static systems $(M^{n},g,f)$, $n\geq3$, with boundary $\partial M$ are necessarily metrically and geodesically complete (up to the boundary $\partial M$) with at most finitely many boundary components, see e.g.\ \cite[Appendix]{CGM}. Moreover, the connected components of $\partial M$ are necessarily all closed, see e.g.\ \cite[Appendix]{CGM}. Here, to be \emph{geodesically complete up to the boundary} means that any geodesic $\gamma\colon I\to M$  with $I\neq\R$ can be smoothly extended to a geodesic $\widehat{\gamma}\colon J\to M\cup\partial M$ such that either $J=\R$, $J=[a,\infty)$, $J=(-\infty,b]$, or $J=[a,b]$ for some $a,b\in\R$ such that $\widehat{\gamma}(a), \widehat{\gamma}(b)\in\partial M$ (if applicable).
	\end{remark}
	
	We will later need the following consequences of our asymptotic assumptions which we formulate for general decay rate $\tau\geq0$ for convenience of the reader.
	\begin{lemma}[Asymptotics]\label{lem:asymptotics}
		Let $(M^{n},g,f)$, $n\geq3$, be an asymptotically flat static system of mass $m\in\R$ and decay rate $\tau\geq0$ with respect to a diffeomorphism $x \colon M \setminus K \longrightarrow \R^n \setminus \overline{B}$ and denote the induced coordinates by $(x^{i})$. Then, for $i,j=1,\dots,n$, we have
		\begin{align*}
			(\nabla f)^{i}&=(\nabla_{\!m}f_{m})^{i}+o_{1}(\vert x\vert^{-(n-1)})=\frac{(n-2)\,mx^{i}}{\vert x\vert^{n}}+o_{1}(\vert x\vert^{-(n-1)}),\\
			\Vert\nabla f\Vert^2&=\Vert\nabla_{\!m} f_{m}\Vert^2_{m}+o_{1}(\vert x\vert^{-2(n-1)})=\frac{(n-2)^{2}\,m^{2}}{\vert x\vert^{2(n-1)}}+o_{1}(\vert x\vert^{-2(n-1)}),\\
			\nabla^{2}_{ij}f&=(\nabla^{2}_{\!m})_{ij}f_{m}+o_{0}(\vert x\vert^{-n})=\frac{(n-2)m\delta_{ij}}{\vert x\vert^{n}}-\frac{n(n-2)m x_{i} x_{j}}{\vert x\vert^{n+2}}+o_{0}(\vert x\vert^{-n}),\\
			\Vert\nabla^{2}f\Vert^{2}&=\Vert\nabla_{\!m}^{2}f_{m}\Vert^{2}_{m}+o_{0}(\vert x\vert^{-2n})=\frac{n(n-1)(n-2)^{2}\,m^{2}}{\vert x\vert^{2n}}+o_{0}(\vert x\vert^{-2n}),\\
			\nabla^{i}\Vert\nabla f\Vert^{2}&=\nabla^{i}_{\!m}\Vert\nabla_{\!m} f_{m}\Vert^{2}_{m}+o_{0}(\vert x\vert^{-(2n-1)})=-\frac{2(n-1)(n-2)^{2}\,m^{2}x^{i}}{\vert x\vert^{2n}}+o_{0}(\vert x\vert^{-(2n-1)}),\\
			\Vert\nabla\Vert\nabla f\Vert^{2}\Vert^{2}&=\Vert \nabla_{\!m}\Vert\nabla_{m} f_{m}\Vert^{2}_{m}\Vert^{2}_{m}+o_{0}(\vert x\vert^{-2(2n-1)}),\\
			\Ric(\nabla f,\nabla f)&=\Ric_{m}(\nabla_{\!m} f_{m},\nabla_{\!m} f_{m})+o_{0}(\vert x\vert^{-(\tau+2n)})=o_{0}(\vert x\vert^{-(\tau+2n)})
		\end{align*}
		as $\vert x\vert\to\infty$. Here, $\nabla$ and $\Vert\cdot\Vert$ denote the connection and tensor norm with respect to $g$ and $\nabla_{\!m}$, $\Vert\cdot\Vert_{m}$, and $\Ric_{m}$ denote the connection, tensor norm, and Ricci tensor with respect to $g_{m}$, respectively. Furthermore, let $r>0$ be such that $B_r \definedas\left\{ x \in \R^n \,:\, |x|<r \right\}\supset \overline{B}$ and let $\nu$ denote the unit normal to $x^{-1}(\partial B_{r})$ pointing towards the asymptotically flat end and let $\HH$ denote the mean curvature of $x^{-1}(\partial B_{r})$ with respect to $\nu$. Then
		\begin{align}\label{asyeta}
			\nu^{i}&=\frac{x^{i}}{\vert x\vert}+o(\vert x\vert^{-\tau}),\\\label{asyH}
			\HH&=\frac{n-1}{\vert x\vert}+o(\vert x\vert^{-1-\tau})
		\end{align}
		as $\vert x\vert\to\infty$. Now let $u,v,u_{0},v_{0}\colon M\setminus K\to\R$ be continuous functions such that $u=u_{0}+o(\vert x\vert^{-(n-1)})$, $u_{0}=O(\vert x\vert^{-(n-1)})$, $v=v_{0}+o(\vert x\vert^{-n})$, and $v_{0}=O(\vert x\vert^{-n})$ as $\vert x\vert\to\infty$. Then
		\begin{align}\label{asyint}
			\int_{x^{-1}(\partial B_{r})}u\,dS&=\int_{\partial B_{r}} (u_{0}\circ x^{-1})\,dS_{\delta}+o(1),\\\label{asyvolint}
			\int_{x^{-1}(\R^{n}\setminus\overline{B_{r}})} v\,dV&=\int_{\R^{n}\setminus\overline{B_{r}}} (v_{0}\circ x^{-1})\,dV_{\delta}+o(1)
		\end{align}
		as $r\to\infty$, where $dS$ and $dS_{\delta}$ denote the area elements induced on $x^{-1}(\partial B_{r})$ and $\partial B_{r}$ and $dV$ and $dV_{\delta}$ denote the volume elements induced on $x^{-1}(\R^{n}\setminus\overline{B_{r}})$ and $\R^{n}\setminus\overline{B}_{r}$ by $g$ and $\delta$, respectively. In particular, $v$ is integrable on $x^{-1}(\R^{n}\setminus\overline{B}_{r})$ with respect to $dV$.
	\end{lemma}
	
	\begin{proof}
		The claims in~\Cref{lem:asymptotics} follow from straightforward computations. For addressing~\eqref{asyeta}, \eqref{asyint}, and~\eqref{asyvolint}, let $(r,\theta^{J})_{J=1}^{n-1}$ be standard polar coordinates for $\R^{n}$ so that $(\partial_{\theta^{K}})^{i}=O_{\infty}(\vert x\vert)$ as $\vert x\vert\to\infty$ and $\delta_{IJ}=r^{2}(g_{\mathbb{S}^{n-1}})_{IJ}$. Here and in what follows, we use the convention that capital latin indices $I,J,K,\dots=1,\dots,n-1$ label the polar coordinates $(\theta^{K})$, while small latin indices $i,j,k,\dots=1,\dots,n$ label the Cartesian coordinates $(x^{i})$ as before. 
		
For $\nu$, we make the ansatz 
		\begin{align*}
			\nu^{i}&=(1+\lambda)\frac{x^{i}}{\vert x\vert}-\delta^{ij}(g_{jk}-\delta_{jk})\frac{x^{k}}{\vert x\vert}+\mu^{L}(\partial_{\theta^{L}})^{i}
		\end{align*}
		for $\lambda, \mu^{L}\in C^{\infty}(\R^{n}\setminus B_{r})$, $L=1,\dots,n-1$. Then for $K,L=1,\dots,n-1$, we compute
		\begin{align*}
			0=g(\nu,\partial_{\theta^{K}})&=-(g_{ij}-\delta_{ij})\frac{x^{i}}{\vert x\vert}(\partial_{\theta^{K}})^{j}+\mu^{L}\vert x\vert^{2}\,(g_{\mathbb{S}^{n-1}})_{KL}+(g_{ij}-\delta_{ij})\nu^{i}(\partial_{\theta^{K}})^{j}\\
			&=\mu^{L}\vert x\vert^{2}\left((g_{\mathbb{S}^{n-1}})_{KL}+o(\vert x\vert^{-\tau})\right)+\lambda\cdot o(\vert x\vert^{-\tau+1})+o(\vert x\vert^{-2\tau+1}),\\
			1=\;\,g(\nu,\nu)\;\,&=(1+\lambda)^{2}(1+o(\vert x\vert^{-\tau}))+\mu^{L}\cdot o(\vert x\vert^{-\tau+1})+(1+\lambda)\cdot o(\vert x\vert^{-\tau})\\
			&\quad+\mu^{K}\mu^{L}\vert x\vert^{2}\left((g_{\mathbb{S}^{n-1}})_{KL}+o(\vert x\vert^{-\tau})\right)+(1+\lambda)\mu^{L}\cdot o(\vert x\vert^{-\tau+1})+o(\vert x\vert^{-2\tau})
		\end{align*}
		as $\vert x\vert\to\infty$. We rewrite the first equation as 
		\begin{align}\label{temp}
			\mu^{L}&=\lambda\cdot o(\vert x\vert^{-(\tau+1)})+o(\vert x\vert^{-(2\tau+1)})
		\end{align}
		and plug this into the second equation, obtaining $1=(1+\lambda)^{2}+o(r^{-\tau})+\lambda\cdot o(r^{-\tau})+\lambda^{2}\cdot o(r^{-\tau})$ and hence by Taylor's formula, this quadratic equation has the two solutions $\lambda_{1}=o(\vert x\vert^{-\tau})$ and $\lambda_{2}=-2+o(\vert x\vert^{-\tau})$ as $\vert x\vert\to\infty$. As we are interested in finding the normal pointing towards $\vert x\vert\to\infty$, we can exclude $\lambda_{2}$ and obtain $\lambda=o(\vert x\vert^{-\tau})$ as desired. Combining this with~\eqref{temp}, we find $\mu^{L}=o(\vert x\vert^{-(2\tau+1)})$ for $L=1,\dots,n-1$ as $\vert x\vert\to\infty$. This proves~\eqref{asyeta}.
		
		For \eqref{asyH}, we compute as above that the components of the inverse induced metric $(\sigma^{IJ})$ on $x^{-1}(\partial B_{r})$ satisfy $\sigma^{IJ}=\frac{1}{\vert x\vert^{2}}(g_{\mathbb{S}^{n-1}})^{IJ}+o(\vert x\vert^{-\tau-2})$ as $\vert x\vert\to\infty$, while the components of the inverse metric satisfy $g^{rr}=1+o_{2}(\vert x\vert^{-\tau})$, $g^{rI}=o_{2}(\vert x\vert^{-\tau-1})$, $g^{IJ}=\frac{1}{\vert x\vert^{2}}(g_{\mathbb{S}^{n-1}})^{IJ}+o_{2}(\vert x\vert^{-\tau-2})$ as $\vert x\vert\to\infty$. From this, one finds that the Christoffel symbols of $g$ behave as
		\begin{align*}
			\Gamma_{IJ}^{r}&=-\vert x\vert(g_{\mathbb{S}^{n-1}})_{IJ}+o(\vert x\vert^{1-\tau}),\\
			\Gamma_{IJ}^{K}&=o(\vert x\vert^{-\tau})
		\end{align*}
		as $\vert x\vert\to\infty$ and thus, using \eqref{asyeta}, we obtain
		\begin{align*}
			\HH&=-\sigma^{IJ}g(\nabla_{I}\partial_{J},\nu)=\frac{n-1}{\vert x\vert}+o(\vert x\vert ^{-1-\tau})
		\end{align*}
		as $\vert x\vert\to\infty$ as claimed. Next, for~\eqref{asyvolint}, we note that
		\begin{align*}
			\sqrt{\deter\left(g_{ij}\right)}&=\sqrt{\deter\left(\delta_{ij}+o(\vert x\vert^{-\tau})\right)}=1+o(\vert x\vert^{-\tau})
		\end{align*}
		as $\vert x\vert\to\infty$ by Taylor's formula. 
		
Hence
		\begin{align*}
			\int_{x^{-1}(\R^{n}\setminus\overline{B_{r}})}v\,dV&=\int_{\R^{n}\setminus\overline{B_{r}}}(v\circ x^{-1})\sqrt{\deter\left(g_{ij}\right)}\,dx^{1}\cdot\dots\cdot dx^n\\
			&=\int_{\R^{n}\setminus\overline{B_{r}}}(v\circ x^{-1})\left(1+o(\vert x\vert^{-\tau})\right)dx^{1}\cdot\dots\cdot dx^n\\
			&=\int_{\R^{n}\setminus\overline{B_{r}}}(v_{0}\circ x^{-1})\left(1+o(\vert x\vert^{-\tau})\right)dV_{\delta}+\int_{\R^{n}\setminus\overline{B_{r}}} o(\vert x\vert^{-n})\,dV_{\delta}\\
			\phantom{\int_{x^{-1}(\R^{n}\setminus\overline{B_{r}})}v\,dV}&=\int_{\R^{n}\setminus\overline{B_{r}}}(v_{0}\circ x^{-1})\,dV_{\delta}+\int_{\R^{n}\setminus\overline{B_{r}}} o(\vert x\vert^{-n})\,dV_{\delta}=\int_{\R^{n}\setminus\overline{B_{r}}}(v_{0}\circ x^{-1})dV_{\delta}+o(1)
		\end{align*}
		as $\vert x\vert\to\infty$, where we have used the decay assumption on $v$ and $v_{0}$ in the third and second to last, and the $L^{\infty}$-$L^{1}$-H\"older inequality in the last step.
		
		Finally, for~\eqref{asyint}, we argue as before and compute 
		\begin{align*}
			\sqrt{\deter\left(g_{IJ}\right)}&=\sqrt{\deter\left(r^{2}(g_{\mathbb{S}^{n-1}})_{IJ}+o(\vert x\vert^{-\tau+2})\right)}\\
			&=r^{n-1}\sqrt{\deter\left((g_{\mathbb{S}^{n-1}})_{IJ}+o(\vert x\vert^{-\tau})\right)}\\
			&=r^{n-1}\sqrt{\deter\left((g_{\mathbb{S}^{n-1}})_{IJ}\right)}\sqrt{\deter\left(\delta_{KL}+((g_{\mathbb{S}^{n-1}})^{-1})_{KL}\cdot o(\vert x\vert^{-\tau})\right)}\\
			&=r^{n-1}\sqrt{\deter\left((g_{\mathbb{S}^{n-1}})_{IJ}\right)}\left(1+o(\vert x\vert^{-\tau})\right)
		\end{align*}
		as $\vert x\vert\to\infty$ by the algebraic properties of the determinant and by Taylor's formula. Arguing as before and using the decay assumption on $u$, this implies
		\begin{align*}
			\int_{x^{-1}(\partial B_{r})}u\,dS&=\int_{\partial B_{r}}(u\circ x^{-1})\sqrt{\deter\left(g_{IJ}\right)}\,d\theta^{1}\cdot\dots\cdot d\theta^{n-1}\\
			&=\int_{\partial B_{r}}(u\circ x^{-1})\,r^{n-1}\sqrt{\deter\left((g_{\mathbb{S}^{n-1}})_{IJ}\right)}\left(1+o(\vert x\vert^{-\tau})\right)d\theta^{1}\cdot\dots\cdot d\theta^{n-1}\\
			&=\int_{\partial B_{r}}(u\circ x^{-1})\left(1+o(\vert x\vert^{-\tau})\right)dS_{\delta}\\
			&=\int_{\partial B_{r}}(u_{0}\circ x^{-1})\left(1+o(\vert x\vert^{-\tau})\right)dS_{\delta}+\int_{\partial B_{r}} o(\vert x\vert^{-(n-1)})\,dS_{\delta}\\
			&=\int_{\partial B_{r}}(u_{0}\circ x^{-1})\,dS_{\delta}+o(1)
		\end{align*}
		as $\vert x\vert\to\infty$. This completes the proof.
	\end{proof}

	\begin{remark}[Choice of normal, regular boundary, tensor norm]\label{rmk:lapse-asymp}
		Let $(M^{n},g,f)$, $n\geq3$, be an asymptotically flat static vacuum system of mass $m$ and decay rate $\tau\geq0$ with connected boundary $\partial M$. Let $\nu$ denote the unit normal to $\partial M$ pointing towards the asymptotically flat end. Now assume first that $f\vert_{\partial M}=f_{0}$ for some $f_0 \in [0, 1)$. Then since $f$ is harmonic by \eqref{vacuum equation2}, the maximum principle\footnote{Indeed, the maximum principle applies under our weak asymptotic flatness conditions from \Cref{def2} which can be seen as follows: Suppose that $\{f\geq1\}\neq\emptyset$. Since $f=f_{0}$ on $\partial M$, $f\to 1$ at infinity, $f$ is continuous, and $M$ is metrically complete up $\partial M$ by \Cref{rem:complete}, $f$ must have a positive maximum at a point $q_0\in M\setminus\partial M$, with $f(q_{0})\geq1$. Now let $U\subset M\setminus\partial M$ be an open neighborhood of $q_{0}$ with smooth boundary $\partial U$, large enough to contain some $q\in U$ with $f(q)<f(q_{0})$; such a neighborhood exists because $f=f_{0}<1$ on $\partial M$. Applying the strong maximum principle to $f\vert_{U}$ gives a contradiction. The possibility that $\{f\leq f_{0}\}\neq\emptyset$ can be handled analogously.} ensures that 
		\begin{align}\label{range_f}
			0 \leq f_0 < f < 1
		\end{align}
		holds on $M$. Moreover, by the Hopf lemma\footnote{Similarly modified as the maximum principle argument to allow for non-compact $M$.}, we can deduce that $\nu(f)=\|\nabla f\| > 0$ on $\partial M$, implying that $\partial M$ is a regular level set of $f$. Thus
		\begin{align}\label{normal}
			\nu &= \frac{\nabla f }{\|\nabla f\|},
		\end{align}
		where here and in what follows, $\|\cdot\|$ denotes the tensor norm induced by $g$ and we slightly abuse notation and denote the gradient of $f$ by $\nabla f$. Next assume  that $f\vert_{\partial M}=f_{0}$ for some $f_0>1$. The same arguments imply that
		\begin{align}\label{range_f1}
			f_0 > f > 1
		\end{align}
		holds on $M$ and 
		\begin{align}\label{normal1}
			\nu &= -\frac{\nabla f }{\|\nabla f\|}.
		\end{align}
		When studying (regular) level sets $\{f=f_{0}\}$ of $f$, we will also use the unit normal $\nu$ pointing towards infinity, so that \eqref{normal} respectively \eqref{normal1} hold when $f_{0}\in(0,1)$ respectively $f_{0}\in(1,\infty)$. Finally, assume that $f\vert_{\partial M}=1$. Then  by the maximum principle, $f\equiv1$ holds on $M$. 
	\end{remark}
	
	\subsection{Static horizons and equipotential photon surfaces}\label{sec:photo}
	Static (black hole) horizons and their surface gravity are defined as follows. For simplicity, we will restrict our attention to connected static horizons already here.
	
	\begin{definition}[Static horizons]
		Let $(M^{n},g,f)$, $n\geq3$, be a static system with connected boundary $\partial M$. We say that $\partial M$ is a \emph{static (black hole) horizon} if $f\vert_{\partial M}=0$.
	\end{definition}
	
	In fact, static horizons as defined above can be seen to be Killing horizons in the sense that the static Killing vector field $\partial_{t}$ smoothly extends to the (extension to the) boundary of the static spacetime $(\R\times M,\overline{g}=-f^{2}dt^{2}+g)$ but at the same time degenerates along this boundary, namely $-f^{2}=\overline{g}(\partial_{t},\partial_{t})\to0$. The standard example of a static system with a static horizon is the Schwarz\-schild system $(M^{n}_{m},g_{m},f_{m})$ of mass $m>0$.
	
	Let us now collect some important properties of static horizons in static vacuum systems.
	
	\begin{remark}[Surface gravity, horizons are totally geodesic]\label{rem:geodesic}
		It is a well-known and straightforward consequence of~\eqref{vacuum equation1} that static horizons in static vacuum systems are totally geodesic and in particular minimal surfaces. Moreover, using again~\eqref{vacuum equation1}, one computes that
		\begin{align}\label{eq002}
			\nabla\|\nabla f\|^{2}&=2f\Ric(\nabla f,\cdot)
		\end{align}
		which manifestly vanishes on a static horizon $\partial M$. This implies that the \emph{surface gravity} $\kappa$ defined by
		\begin{align}\label{kappa}
			\kappa\definedas \nu(f)\vert_{\partial M}
		\end{align}
		for some unit normal along $\partial M$ is constant on the static horizon $\partial M$. Combined with \Cref{rmk:lapse-asymp}, this shows that the surface gravity of a (connected) static horizon in an asymptotically flat static vacuum system is necessarily non-vanishing, $\kappa\neq0$ and positive when one chooses $\nu$ to point to infinity. This fact is sometimes expressed as saying that such static horizons are ``non-degenerate''.
	\end{remark}
	
	Next, let us recall the definition and properties of equipotential photon surfaces and of photon spheres, the central objects studied in \Cref{teoclassifica2}. We will be very brief as we will only need specific properties and refer the interested reader to~\cite{CedGalSurface,CJV} for more information and references. In particular, we will assume that all photon surfaces are necessarily connected for simplicity of the exposition and as we will only study connected photon surfaces in this paper, anyway. It will temporarily be more convenient to think about static spacetimes rather than static systems.
	\begin{definition}[(Equipotential) photon surface, photon sphere]
		A smooth, timelike, embedded, and connected hypersurface in a smooth Lorentzian manifold is called a \emph{photon surface} if it is totally umbilic. A photon surface $P^{n}$ in a static spacetime $(\R\times M^{n},\overline{g}=-f^{2}dt^{2}+g)$ is called \emph{equipotential} if the lapse function $f$ of the spacetime is constant along each connected component of each \emph{time-slice} $\Sigma^{n-1}(t)\definedas P^{n}\,\cap \left(\{t\}\times M^{n}\right)$ of the photon surface. An equipotential photon surface is called a \emph{photon sphere} if the lapse function $f$ is constant (in space and time) on $P^{n}$.
	\end{definition}
	It is well-known that the (exterior) Schwarz\-schild spacetime of mass $m>0$ (i.e., the spacetime associated to the Schwarzschild system $(M_{m}^{n},g_{m},f_{m})$ of mass $m>0$) possesses a photon sphere at $r=(nm)^{\frac{1}{n-2}}$. Moreover, it follows from a combination of results by Cederbaum and Galloway~\cite[Theorem 3.5, Proposition 3.18]{CedGalSurface} and by Cederbaum, Jahns, and Vi\v{c}\'anek Mart\'inez~\cite[Theorems 3.7, 3.9, and 3.10]{CJV} that all Schwarz\-schild spacetimes possess very many equipotential photon surfaces. In particular, every sphere $\mathbb{S}^{n-1}(r)\subset M^{n}_{m}$ arises as a time-slice of an equipotential photon surface. On the other hand, no other closed hypersurfaces of $(M^{n}_{m},g_{m}, f_{m})$ arise as time-slices of equipotential photon surfaces by \cite[Corollary 3.9]{CedGalSurface}.
	
	Let us now move on to study the intrinsic and extrinsic geometry of time-slices of equipotential photon spheres. Time-slices of equipotential photon surfaces and in particular of photon spheres have the following useful properties.
	\begin{proposition}[{\cite[Proposition 5.5]{CJV}}]\label{prop:photo}
		Let $(M^{n},g,f)$, $n\geq3$, be an asymptotically flat static vacuum system and let $\partial M$ be a time-slice of an equipotential photon surface with $f=f_{0}$ on $\partial M$ for some constant $f_{0}>0$, $f_{0}\neq1$. Then $\partial M$ is totally umbilic in $(M,g)$, has constant scalar curvature $\scal_{\partial M}$, constant mean curvature $\HH$, and constant $\kappa\definedas\nu(f)\vert_{\partial M}$, related by the \emph{equipotential photon surface constraint}
		\begin{align}\label{eq:photo}
			\scal_{\partial M}&=\frac{2\kappa\HH }{f_{0}}+\frac{n-2}{n-1}\HH^{2}.
		\end{align}
	\end{proposition}
	Here, we are using that $\kappa=\nu(f)\vert_{\partial M}\neq0$ by \Cref{rmk:lapse-asymp}.
	
	\begin{proposition}[{\cite[Lemma 2.6]{carlagregpmt}, \cite[Theorem 5.22]{CJV}}]\label{prop:signH}
		In the setting of \Cref{prop:photo}, we have $\HH>0$.
	\end{proposition}
	In fact, both \cite[Lemma 2.6]{carlagregpmt} and \cite[Theorem 5.22]{CJV} assume stronger asymptotic decay than we do, and in addition assume $\nu(N)>0$ resp. $\HH\nu(N)>0$ on $\partial M$. As $\partial M$ is connected here, neither of the second assumptions are needed to conclude as can be seen in the corresponding proofs, as these assumptions are only needed to handle potential other boundary componentss. Concerning the asymptotic decay, it suffices to note that our decay assumptions imply that large coordinate spheres have positive mean curvature by \Cref{lem:asymptotics}.
	
	\begin{remark}\label{unification}
		Formally taking the limit of the equipotential photon surface constraint~\eqref{eq:photo} as $f_{0}\searrow0$, one recovers the twice contracted Gau{\ss} equation
		\begin{align*}
			\scal_{\partial M}&=-\frac{2\Ric(\nu,\nu)\kappa}{\Vert\nabla f\Vert}=-2\Ric(\nu,\nu),
		\end{align*}
		with $\kappa$ denoting the surface gravity of the static horizon $\{f_{0}=0\}$. To see this, one uses the well-known fact that 
		$\HH=-\frac{\nabla^{2}f(\nu,\nu)}{\Vert\nabla f\Vert}$ on regular level sets of $f$ (for $0<f<1$), \eqref{vacuum equation1}, and \eqref{scal0}. In particular, the first term $\frac{2\kappa H}{f_{0}}$ of \eqref{eq:photo} remains well-defined in the case $f_{0}=0$.
	\end{remark}
	
	\begin{lemma}[Smarr formula]\label{Smarr}
		Let $(M^{n},g,f)$, $n\geq3$, be an asymptotically flat static vacuum system with mass $m\in\R$. Then the \emph{Smarr formula}
		\begin{align}\label{eq:Smarr}
			\int_{\{f=z\}} \nu(f)\,dS &= (n-2)\vert\mathbb{S}^{n-1}\vert\, m
		\end{align}
		holds for every regular, connected level set $\{f=z\}$ of $f$, where $z \geq 0$ is a constant. Here, $\vert\mathbb{S}^{n-1}\vert$ denotes the area of $(\mathbb{S}^{n-1},g_{\mathbb{S}^{n-1}})$ and $\nu$ denotes the unit normal to $\{f=z\}$. Moreover,  if $(M,g,f)$ has a connected boundary $\partial M$ then
		\begin{align}\label{eq:Smarr2}
			\int_{\partial M} \nu(f)\,dS &= (n-2)\vert\mathbb{S}^{n-1}\vert\, m.
		\end{align}
		Furthermore, if in addition $f\vert_{\partial M}=f_{0}$ for some $f_0\geq0$ then $m>0$ when $f_{0}\in[0,1)$, $m=0$ when $f_{0}=1$, and $m<0$ when $f_{0}>1$. In particular, if $\partial M$ is a static horizon or a time-slice of an equipotential photon surface with $f_{0}<1$ resp. $f_{0}>1$ then $m>0$ resp. $m<0$.
	\end{lemma}
	
	\begin{remark}[Quasi-local mass, outward directed equipotential photon surfaces, and why we avoid the zero mass case]\label{rem:Smarrmass}
		The Smarr formula \eqref{eq:Smarr2} allows one to define a quasi-local mass for $\partial M$ by expressing $m$ in terms of the other quantities in \eqref{eq:Smarr2} (see e.g.~\cite{CDiss}). \Cref{Smarr} hence states that said quasi-local mass of a connected boundary $\partial M$ coincides with the asymptotic mass parameter $m$ of the static system. Furthermore, it informs us that if $f=f_{0}$ on $\partial M$ for some constant $f_{0}\geq0$,  the sign/vanishing of the mass $m$ is fixed by the value of $f_{0}$. This allows to refer to the case $f_{0}\in[0,1)$ as the \emph{positive mass case}, to $f_{0}=0$ as the \emph{zero mass case}, and to the case $f_{0}\in(1,\infty)$ as the \emph{negative mass case}, respectively. It also explains why we avoid the zero mass case in this paper altogether: If $f_{0}=1$, \Cref{rmk:lapse-asymp} informs us that $f\equiv1$ on $M$ and thus $(M,g)$ is necessarily Ricci-flat by \eqref{vacuum equation1}. In dimension $n=3$, this implies that $(M,g)$ is indeed flat; one can conclude that it isometric to Euclidean space without a ball using the asymptotically flatness with decay rate $\tau \geq 0$, without assuming any additional properties (see \cite{CCF}). In higher dimensions, proving a similar statement is a problem of a different nature, which is going to be addressed elsewhere.
		
		As briefly touched upon in \Cref{sec:intro}, the existing static vacuum uniqueness results for equipotential photon surfaces all\footnote{With the exception of \cite{CCF} for $n=3$ and connected $\partial M$.} assume that those are \emph{outward directed}, meaning that $\kappa=\nu(f)\vert_{\partial M}>0$. In view of \Cref{Smarr}, this corresponds to a restriction to the positive (quasi-local) mass case.
	\end{remark}

	\begin{proof}[Proof of \Cref{Smarr}]
		The fact that the left-hand side of \eqref{eq:Smarr} is independent of the value of $z$ is a direct consequence of~\eqref{vacuum equation2} and the divergence theorem. To see that the constants on the right-hand sides of \eqref{eq:Smarr}, \eqref{eq:Smarr2} equal $(n-2)\vert\mathbb{S}^{n-1}\vert\, m$, one argues as follows, using the notation from \Cref{lem:asymptotics}. First, $\nu(f)=\frac{(n-2) m}{\vert x\vert^{n-1}}+o(\vert x\vert^{-(n-1)})$ as $\vert x\vert\to\infty$ by \Cref{lem:asymptotics} and \eqref{asyeta}. Hence $u\definedas\nu(f)$, $u_{0}\definedas \frac{(n-2) m}{\vert x\vert^{n-1}}$ are suitable functions for the application of \eqref{asyint}. Then, by \eqref{vacuum equation2} and the divergence theorem, we get
		\begin{align*}
			\int_{\partial M} \nu(f)\,dS &= -\int_{\{p\in M\,:\,\vert x\vert(p)<r\}}\Delta f\,dV+\int_{x^{-1}(\partial B_{r})}\nu(f)\,dS=\int_{x^{-1}(\partial B_{r})}\nu(f)\,dS\\
			&=\int_{\partial B_{r}} (u_{0}\circ x^{-1})\,dS_{\delta}+o(1)=(n-2)m\int_{\partial B_{r}}\frac{1}{\vert x\vert^{n-1}}\,dS_{\delta}+o(1)\\
			&=(n-2)\vert\mathbb{S}^{n-1}\vert\, m+o(1)
		\end{align*}
		as $r=\vert x\vert\to\infty$, where $dV$ denotes the volume element on $M$. This proves \eqref{eq:Smarr2}. In particular, if $f_{0}=1$, \Cref{rmk:lapse-asymp} tells us that $f=1$ on $M$ and hence there are no regular level sets of $f$ and no claim about \eqref{eq:Smarr}. The asymptotic formula for $\nu(f)=0$ directly shows that $m=0$. If $f_{0}\neq1$, regular level sets can exist and \eqref{eq:Smarr} then follows precisely as \eqref{eq:Smarr2}, up to a sign in front of the volume integral over $\triangle f$ if $z>1$, and with the domain of said volume integral taking the form $\{p\in M\,:\,f(p)>z,\, \vert x\vert(p)<r\}$ if $0\leq z<1$ and the form $\{p\in M\,:\,f(p)<z,\, \vert x\vert(p)<r\}$ if $z>1$ in view of \Cref{rmk:lapse-asymp}. The remaining claims are direct consequences of the Smarr formula and of \Cref{rmk:lapse-asymp}, via \Cref{prop:photo}.
	\end{proof}
	
	\begin{remark}[Admissible decay rates]\label{rem:admissibledecay}
		In \Cref{def2}, we have allowed the decay rate $\tau\geq0$ to be arbitrary. In the static vacuum setting, $\tau\geq n-2$ implies that $m=0$ via \Cref{Smarr}, arguing as in the proof of \Cref{lem:asymptotics}, hence our assumption \eqref{eq:f-asymp} effectively restrict the range of the decay rate to $\tau<n-2$. 
		\end{remark}

	\section{The $T$-tensor and its properties}\label{sec:Ttensor}
	In this section, we will discuss properties of the $T$-tensor introduced in \eqref{tensorT} which will be essential for establishing our results. We will also give a proof of the rigidity result \Cref{thm:rigidity}. Remember that, for a Riemannian manifold $(M^{n},\,g),$ $n\geq 3,$ the Weyl tensor $W$ is defined as
	\begin{align}\label{weyl}
		W&\definedas\Rm-\frac{1}{n-2}\left(\Ric-\frac{\scal}{2}g\right)\mathbin{\bigcirc\mspace{-15mu}\wedge\mspace{3mu}} g-\frac{\scal}{2n(n-1)}g\mathbin{\bigcirc\mspace{-15mu}\wedge\mspace{3mu}} g,
	\end{align}
	where $\Rm$ stands for the Riemann curvature operator of $(M,g)$, and $\mathbin{\bigcirc\mspace{-15mu}\wedge\mspace{3mu}}$ denotes the Kulkarni--Nomizu product. Moreover, the Cotton tensor $C$ of $(M,g)$ is given by
	\begin{align}
		\begin{split}\label{cotton}
			C(X,Y,Z)&\definedas(\nabla_{X}\Ric)(Y,Z)-(\nabla_{Y}\Ric)(X,Z)\\
			&\quad-\frac{1}{2(n-1)}\left((\nabla_{X}\scal) g(Y,Z)-(\nabla_{Y}\scal) g(X,Z)\right)
		\end{split}
	\end{align}
	for $X,Y,Z\in\Gamma(TM)$. It is well-known that $W$ vanishes for $n=3$, while for $n\geq4$, $W$ and $C$ are related via
	\begin{align}\label{WeylCotton}
		C&=-\frac{(n-2)}{(n-3)}(\nabla_{E_{i}}\!W)(\cdot,\cdot,\cdot,E_{j})\delta^{ij}
	\end{align}
	for any local orthonormal frame $\left\{E_{i}\right\}_{i=1}^{n}$ of $M$. 
	
	For $n=3$, it is well-known that the Cotton tensor detects (local) conformal flatness in the sense that $C=0$ if and only if $(M^{3},g)$ is locally conformally flat. The same holds true for the Weyl tensor when $n\geq4$. 
	
	The \emph{$T$-tensor} of a Riemannian manifold $(M^{n},g)$, $n\geq3$, carrying a smooth function $f\colon M\to\R$ is given by \eqref{tensorT}. Due to the symmetry of the Ricci tensor, $T$ is antisymmetric in its first two entries. By a straightforward algebraic computation, its squared norm can be computed to be
	\begin{align}\label{prop007}
		\|T\|^2&=\frac{2(n-1)}{(n-2)^{2}}\left[(n-1)\|\!\Ric\!\|^{2}\,\|\nabla f\|^{2}-n\|\Ric(\nabla f,\cdot)\|^{2}+2 \scal\Ric(\nabla f,\nabla f)\right].
	\end{align}
	In particular, if $(M^{n},g,f)$, $n\geq3$, is a static vacuum system, the last term in \eqref{prop007} vanishes by \eqref{scal0}. It is interesting to note the following relation between the Weyl, the Cotton, and the $T$-tensor.
	
	\begin{lemma}[Relation between $W$, $C$, and $T$]\label{lem20}
		Let $(M^n,g,f)$, $n\geq3$, be a static vacuum system. Then
		\begin{align}\label{eq:CTW}
			fC&=W(\cdot,\cdot,\cdot,\nabla f) + T
		\end{align}
		holds on $M$.
	\end{lemma}
	\begin{proof}
		For simplicity, we will use abstract index notation in this proof. First, taking the covariant derivative of \eqref{vacuum equation1}, we have
		\begin{align*}
			\nabla_{i}\nabla_{j}\nabla_{k}f&=\nabla_{i}f\Ric_{jk}+f\nabla_{i}\Ric_{jk}.
		\end{align*} 
		Next, from the Ricci equation we get that
		\begin{align*}
			\Ric_{jk}\nabla_{i}f-\Ric_{ik}\nabla_{j}f+f\left(\nabla_{i}\Ric_{jk}-\nabla_{j}\Ric_{ik}\right)&=\nabla_{i}\nabla_{j}\nabla_{k}f-\nabla_{j}\nabla_{i}\nabla_{k}f=\Rm_{ijkl}\nabla^{l}f.
		\end{align*} 
		By~\eqref{scal0}, we obtain from the definition of $C$ in \eqref{cotton} that
		\begin{align*}
			\Ric_{jk}\nabla_{i}f-\Ric_{ik}\nabla_{j}f+fC_{ijk}=\Rm_{ijkl} \nabla^{l}f.
		\end{align*} 
		Similarly, from the definition of the Weyl tensor in \eqref{weyl}, we obtain
		\begin{align*}
			\Rm_{ijkl}\nabla^{l}f=W_{ijkl}\nabla^{l}f+\frac{1}{n-2}\left(\Ric_{ik}\nabla_{j}f-\Ric_{jk}\nabla_{i}f+\Ric_{jl}\nabla^{l}fg_{ik}-\Ric_{il}\nabla^{l}fg_{jk}\right).
		\end{align*}
		Combining the last two equations gives the desired result.
	\end{proof}

	It is well-known that the Schwarz\-schild system $(M^{n}_{m},g_{m},f_{m})$ of mass $m$ can be rewritten in a manifestly conformally flat way by using the above-mentioned isotropic coordinates (this also applies in the negative mass case although not in a global isotropic coordinate chart). Hence its Weyl tensor $W_{m}$ vanishes for all $n\geq3$ and its Cotton tensor $C_{m}$ vanishes for $n=3$. From \eqref{WeylCotton}, we deduce that in fact its Cotton tensor $C_{m}$ and hence by \Cref{lem20} its $T$-tensor $T_{m}$ vanish for all $n\geq3$, that is $C_{m}=T_{m}=0$.
	
	We will later make use of the following lemma which relies on the idea of rewriting $T$ only in terms of $f$.
	\begin{lemma}[An identity for $\Vert T\Vert^{2}$]\label{lem21}
		Let $(M^n,g,f)$, $n\geq3$, be a static vacuum system. Then
		\begin{align*}
			\frac{(n-2)^{2}}{(n-1)}f^{2}\|T\|^{2} &= (n-1)\|\nabla f\|^{2}\left(\Delta\|\nabla f\|^{2} - \frac{\langle\nabla\|\nabla f\|^{2},\nabla f\rangle}{f}\right) - \frac{n}{2}\|\nabla\|\nabla f\|^{2}\|^{2}
		\end{align*}
		holds on $M$, where $\langle\cdot,\cdot\rangle$ denotes the metric $g$.
	\end{lemma}
	\begin{proof}
		Let us rewrite the norm of $T$ only in terms of the function $f$, not explicitly involving any curvature terms. To that end, taking the divergence of \eqref{eq002} divided by $2f$ and using \eqref{vacuum equation1}, \eqref{vacuum equation2} and the Bianchi identity, we get
		\begin{align*}
			2\|\!\Ric\!\|^{2}=\frac{\Delta\|\nabla f\|^{2}}{f^{2}}-\frac{\langle\nabla\|\nabla f\|^{2},\nabla f\rangle}{f^{3}}.
		\end{align*}
		Combining this identity with \eqref{eq002} and \eqref{prop007} gives the result.
	\end{proof}
	
	Let us also state the following interesting fact which is useful for understanding when $T$ vanishes and will be used to prove the rigidity result \Cref{thm:rigidity}.
	\begin{lemma}\label{lemmaT}
		 Let $(M^n,g)$, $n\geq3$, be a smooth Riemannian manifold carrying a smooth function $f\colon M\to\R$. Then $T=0$ on $M\setminus\crit$ if and only if 
		\begin{align}\label{gEinstein}
			\|\nabla f \|^2 \Ric &=- \frac{\lambda \|\nabla f\|^2}{n-1} g+\frac{n\lambda}{n-1} df\otimes df
		\end{align}
		on $M\setminus\crit$ for some smooth function $\lambda\colon M\setminus\crit\to\R$. Note that \eqref{gEinstein} implies in particular that		
\begin{align}\label{eigenvector}
			\Ric(\nabla f,\cdot)^{\#}&=\lambda\, \nabla f
	\end{align}
	so that $\nabla f$ is an eigenvector field of $\Ric$ on $M\setminus\crit$ with eigenvalue $\lambda$.
		\end{lemma}
		
	\begin{proof}
		If $T=0$, 
		one has
		\begin{align*}
			0&= T(\cdot, \cdot, \nabla f) = \frac{1}{n-2}\left(\Ric(\nabla f,\cdot) \otimes df - df \otimes \Ric(\nabla f,\cdot) \right)
		\end{align*}
		on $M\setminus\crit$ which implies \eqref{eigenvector} for a smooth function $\lambda$. To see that \eqref{gEinstein} holds, we use \eqref{eigenvector} to compute
		\begin{align*}
			0 &= \frac{n-2}{n-1}\, T(\cdot , \nabla f, \cdot) = \|\nabla f \|^2 \Ric + \frac{\lambda \|\nabla f\|^2}{n-1} g-\frac{n\lambda}{n-1} df\otimes df
		\end{align*}
		on $M\setminus\crit$ as claimed. Conversely, using \eqref{gEinstein}, we find by straightforward computations using linear and multilinear arguments that
	\begin{align*}
	T(X,Y,Z)&= \frac{\lambda}{n-2} \left(-g(X,Z)\nabla_{\!Y}f + g(Y,Z) \nabla_{\!X}f\right)\\
				& \quad- \frac{\lambda}{n-2} \left( g(Y, Z)\nabla_{X}f -g(X, Z)\nabla_{Y}f\right)=0
	\end{align*}
	on $M\setminus\crit$ for all $X,Y,Z\in\Gamma(TM)$.
	\end{proof}

Next, we prove the following local characterization of static vacuum systems $(M^{n},g,f)$ satisfying $T=0$.

\begin{theorem}[Local characterization of $T=0$]\label{thm:locally}
Let $(M^n,g,f)$, $n\geq3$, be a static vacuum system. Then $T=0$ on $M$ if and only if each regular point of $f$ has an open neighborhood $V\subseteq M\setminus\crit$ such that $(V,g\vert_{V},f\vert_{V})$ belongs to precisely one of the following types of systems, with $\lambda\vert_{V}\colon V\to\R$ denoting the eigenvalue of $\Ric$ from \eqref{eigenvector}. Either, in \textbf{Type 1}, there is a constant $a>0$, an interval $I\subseteq\R^{+}$, and a Ricci flat manifold $(\Sigma^{n-1},\sigma)$ such that
\begin{align*}
V&=I\times \Sigma\ni(h,\cdot),\\
g\vert_{V}&=dh^{2}+\sigma,\\
f\vert_{V}(h,\cdot)&=a h,\\
\lambda&\equiv0.
\end{align*}
Or, in \textbf{Types 2--4}, there are constants $a>0$, $b\in\R$, an interval $I\subseteq\R^{+}$, a Riemannian manifold $(\Sigma^{n-1},\sigma)$, and a smooth function $u\colon I\to\R^{+}$ such that
\begin{align*}
V&=I\times \Sigma\ni (r,\cdot),\\
g\vert_{V}&=\frac{1}{u(r)^{2}}dr^{2}+r^{2}\sigma,\\
f\vert_{V}(r,\cdot)&=a u(r),\\
\lambda\vert_{V}(r,\cdot)&=\lambda(r)
\end{align*}
with
\vspace{1ex}
\begin{center}
\begin{tabular}{l|c}
& \\
 \textbf{Type 2} \quad&\quad $b=0$, $(\Sigma,\sigma)$ is Ricci flat, $u(r)=\frac{1}{r^\frac{n-2}{2}}$, and $\lambda(r)=\frac{(n-1)(n-2)}{2r^{n}}$ \quad\\[3ex]\hline\\
 \textbf{Type 3} \quad& \quad $b>0$, $(\Sigma,\sigma)$ is Einstein with $\scal_{\sigma}=-(n-1)(n-2)$,\\
 &\quad $u(r)=\sqrt{\frac{b}{r^{n-2}}-1}$, $\lambda(r)=\frac{(n-1)(n-2)b}{2r^n}$, and $I\subseteq(0,b^{\frac{1}{n-2}})$ \quad\\[3ex]\hline\\
\textbf{Type 4} \quad&\quad $b\neq0$, $(\Sigma,\sigma)$ is Einstein with $\scal_{\sigma}=(n-1)(n-2)$,\\
 &\quad $u(r)=\sqrt{\frac{b}{r^{n-2}}+1}$, $\lambda(r)=\frac{(n-1)(n-2)b}{2r^n}$, and $I\subseteq(-b^{\frac{1}{n-2}},\infty)$ when $b<0$ \quad\\[3ex]
 \end{tabular}
\end{center}
\vspace{1ex}
up to a change of local coordinates.
\end{theorem}

\begin{remark}[Quasi-Schwarzschild systems]\label{rem:quasi}
All systems of Type 4 in \Cref{thm:locally} are quasi-Schwarzschild systems (cf.\ \cite{ndimunique}) of negative $(b>0)$ or positive $(b<0)$ mass $m=-\frac{b}{2}$, respectively. They are Schwarzschild systems of negative respectively positive mass precisely when $(\Sigma,\sigma)$ is a unit radius round sphere.
\end{remark}

\begin{proof}
By continuity, each regular point $p$ of $f$  has a neighborhood $\widetilde{U}\subseteq M$ on which $\nabla f\neq0$, so that $\widetilde{U}\subseteq M\setminus\crit$. Let $\Sigma\definedas \widetilde{U}\cap \{f=f(p)\}$ and choose local coordinates $(\varphi^{J})_{J=1}^{n-1}$ on $\Sigma$  (making $\Sigma$ smaller if necessary). Now flow the coordinates $(\varphi^{J})$ to a neighborhood of $\Sigma$ along $\frac{\nabla f}{\Vert\nabla f\Vert^{2}}$, staying inside $\widetilde{U}$. Making $\Sigma$ even smaller if necessary, this construction gives local coordinates $((f,\varphi^{J}))_{J=1}^{n-1}$ on a neighborhood $U\subset M\setminus\crit$ of $p$ with $U\approx F(f)\times\Sigma$, with $F(f)$ some open interval. 

 In the coordinates $((f,\varphi^{J}))_{J=1}^{n-1}$, one finds $\partial_{f}=\frac{\nabla f}{\Vert\nabla f\Vert^{2}}$ and obtains the usual level set flow formulas  
\begin{align}\label{gf}
g&= \frac{df^{2}}{\|\nabla f\|^{2}} + g_f,\\\label{hf}
h_{f}&=\frac{\Vert\nabla f\Vert}{2}\,\partial_{f}g_{f}
\end{align}
on $U$ and $\{f\}\times \Sigma\asdefined \Sigma_{f}$, respectively. Here, $g_{f}$ denotes the induced metric on the regular level set $\{f\}\times \Sigma\asdefined \Sigma_{f}$ of $f$ in $U$ and $h_{f}$ denotes the second fundamental form of $\Sigma_{f}$ in $U$ with respect to  the unit normal $\|\nabla f\|\,\partial_{f}$. Using harmonicity of $f$ from \eqref{vacuum equation2}, we obtain the usual formula for the mean curvature $H_{f}$ of $\Sigma_{f}$ with respect to the unit normal $\|\nabla f\|\,\partial_{f}$, that is
\begin{align}\label{Hf}
\HH_{f}&=-\partial_{f}\Vert\nabla f\Vert
\end{align}
on $\Sigma_{f}$. 

Now by \Cref{lemmaT}, we know that $T=0$ on $U$ is equivalent to the existence of a smooth function $\lambda\colon U\to\R$ such that \eqref{gEinstein} holds on $U$. Rewriting this in our adapted coordinates $(f,\varphi^{J})$ and using the static vacuum equation \eqref{vacuum equation1}, \eqref{gEinstein} implies
\begin{align}\label{A}
\partial_{f}g_{f}&=-\frac{2\lambda f}{(n-1)\|\nabla f\|^{2}}\,g_{f},\\\label{B}
\partial_{f}\|\nabla f\|&=\frac{\lambda f}{\|\nabla f\|}
\end{align}
on all $\Sigma_{f}$. Rewriting \eqref{A} using  \eqref{hf}, we obtain
\begin{align}\label{eq:gfumbilic}
h_{f}&=-\frac{\lambda f}{(n-1)\|\nabla f\|}\,g_{f}
\end{align}
so that in particular each $\Sigma_{f}$ is umbilic when $T=0$. Moreover, rewriting \eqref{B} as a vector field expression gives
\begin{align}\label{proplambda}
\nabla\|\nabla f\|^{2}&=2\lambda f \|\nabla f\|^{2}\,\nabla f
\end{align}
which shows that $\|\nabla f\|$ is constant on each level set $\Sigma_{f}$ of $f$ as can be seen by inserting all vector fields $X\in\Gamma(U)$ with $X(f)=0$ on $U$ into \eqref{proplambda}. This allows us to set
\begin{align}\label{def:psi}
\psi(f)&\definedas\|\nabla f\|\vert_{\Sigma_{f}}>0
\end{align}
for $f\in F(f)$. Inserting this into \eqref{B} gives 
\begin{align}\label{eq:lambda}
\psi'(f)&=\frac{\lambda f}{ \psi(f)}
\end{align}
on $U$, where $'=\frac{d}{df}$. In particular, $\lambda$ is constant on each level set $\Sigma_{f}$ of $f$ so that we can suggestively write $\lambda=\lambda(f)=\frac{\psi'(f)\psi(f)}{f}$ on $F(f)$. Moreover, each level set $\Sigma_{f}$ must have constant mean curvature
\begin{equation}\label{eq:Hflambda}
H_{f}=-\psi'(f).
\end{equation}

In summary, we have established that $T=0$ on $U$ if and only if $g_{f}$ satisfies 
\begin{equation}\label{eq:T0}
\partial_{f}g_{f}=-\frac{2\psi'(f)}{(n-1)\psi(f)}\,g_{f}
\end{equation}
on all $\Sigma_{f}$ for some smooth, positive function $\psi\colon U\to\R^{+}$ (which implies \eqref{def:psi}) and
\begin{equation}\label{eq:lambdaspecific}
\Ric=\frac{\psi'(f)}{(n-1)f\psi(f)}\left((n-1)df^{2}-\psi(f)^{2}g_{f}\right)
\end{equation}
holds on $U$ by \eqref{gEinstein} and \eqref{gf}. In particular, this implies that all $\Sigma_{f}$ are totally umbilic with constant mean curvature given by \eqref{eq:Hflambda}. Also, note that the static vacuum equations \eqref{vacuum equation1}, \eqref{vacuum equation2} are automatically satisfied by metrics of this type via \eqref{proplambda} and \eqref{eq:lambdaspecific}. 

Using \eqref{gf} and the definition of $\psi$ from \eqref{def:psi}, \eqref{eq:lambdaspecific} can be seen to be equivalent to 
\begin{align}\label{ODE}
0&=\frac{\psi''(f)}{\psi(f)}-\frac{\psi'(f)^{2}}{(n-1)\psi(f)^{2}}-\frac{\psi'(f)}{f\psi(f)},\\\label{Einstein}
\Ric_{g_{f}}&=\frac{1}{n-1}\left(-\psi(f)\psi''(f)+\psi'(f)^{2}-\frac{\psi(f)\psi'(f)}{f}\right)g_{f}
\end{align}
on $U$ by standard computations, where $\Ric_{g_{f}}$ denotes the Ricci tensor of $g_{f}$ on $\Sigma_{f}$. Standard ODE tricks show that the general solution to \eqref{ODE} is given by
\begin{align}\label{solutionpsi}
\psi(f)&=\left(\alpha f^{2}+\beta\right)^{\frac{n-1}{n-2}}
\end{align}
for constants $\alpha,\beta\in\R$ satisfying 
\begin{equation}\label{alphabeta}
\alpha f^{2}+\beta>0
\end{equation} on $F(f)$. Inserting \eqref{solutionpsi} into \eqref{Einstein} gives
\begin{align}\label{Einsteinsimple}
\Ric_{g_{f}}&=-\frac{4\alpha\beta\psi(f)^{\frac{2}{n-1}}}{n-2}\,g_{f}
\end{align}
on $U$. In particular, this shows that each manifold $(\Sigma_{f},g_{f})$ is Einstein. Moreover, \eqref{eq:T0} and \eqref{solutionpsi} give
\begin{align}\label{eq:T0concrete}
\partial_{f}g_{f}&=-\frac{4\alpha f}{(n-2)(\alpha f^{2}+\beta)}\,g_{f}
\end{align}
on $U$ and $F(f)$, respectively. Summarizing, we have shown that $T=0$ on $U$ is equivalent to the combination of \eqref{gf}, \eqref{def:psi}, \eqref{eq:T0concrete}, and \eqref{solutionpsi} and \eqref{Einsteinsimple} holding on $U$ for constants $\alpha,\beta\in\R$ satisfying \eqref{alphabeta}. Let us now discuss the different cases arising from picking specific cases for the signs of $\alpha,\beta$.\\[-1ex]

\underline{First of all, for $\alpha=0$,} we have by \eqref{Einsteinsimple} that $g_{f}$ is Ricci flat, and by \eqref{eq:T0concrete} that $\partial_{f}g_{f}=0$. Now set $a\definedas\beta^{\,\frac{n-1}{n-2}}$ which is well-defined as $\beta>0$ by \eqref{alphabeta}  and note that $a\in\R^{+}$ is unrestricted by \eqref{alphabeta}. Then we can rewrite the static vacuum system $(U,g,f)$ as $U\approx I\times\Sigma\asdefined V$ for the open interval  $I\definedas a^{-1}F(f)\subseteq\R^{+}$, $g=dh^{2}+\sigma$ on $V$ for $h\definedas a^{-1}f$ and with $\sigma\definedas g_{f}$ being a fixed Ricci flat metric on $\Sigma$, and $f(h,\cdot)=a h$ on $V$ satisfying $f(V)=F(f)$. Moreover, $\lambda=0$ in this case by \eqref{eq:lambda}. This shows that for $\alpha=0$, the system $(U,g,f)$ is of Type 1 and that systems of Type 1 satisfy $T=0$ on $V$ as well as the static vacuum equations \eqref{vacuum equation1}, \eqref{vacuum equation2}. The latter statement exploits that $\sigma$ is unrestricted other than being Ricci flat.\\
  
\underline{Second, for $\beta=0$,} we find that $\alpha>0$ is unrestricted by \eqref{alphabeta}. By \eqref{Einsteinsimple}, we learn that $g_{f}$ is Ricci flat, while \eqref{eq:T0concrete} gives 
\begin{equation}\label{eq:evogf}
\partial_{f}g_{f}=-\frac{4}{(n-2)f}\,g_{f}.
\end{equation}
Picking $\widetilde{\sigma}\definedas g_{f_{0}}$ for any fixed $f_{0}\in F(f)$, this gives $g_{f}=\left(\frac{f_{0}}{f}\right)^{\!\frac{4}{n-2}}\widetilde{\sigma}$. Now we set 
\begin{align*}
\kappa&\definedas\frac{n-2}{2\alpha^{\frac{n-1}{n-2}}f_{0}^{\frac{n}{n-2}}},\\
r(f)&\definedas\kappa^{\frac{2}{n}} \left(\frac{f_{0}}{f}\right)^{\!\frac{2}{n-2}}
\end{align*}
 on $F(f)$ and find that $r=r(f)$ has the inverse function $f=f(r)$ given by
 \begin{equation*}
 f(r)=\frac{\kappa^{\frac{n-2}{n}}f_{0}}{r^{\frac{n-2}{2}}}\asdefined\frac{a}{r^{\frac{n-2}{2}}}
 \end{equation*}
 on $I\definedas r(F(f))\subseteq\R^{+}$ with unrestricted $a>0$ by construction (noticing that $a=\frac{(n-2)^{\frac{n-2}{n}}}{2^{\frac{n-2}{n}}\alpha^{\frac{n-1}{n}}}$ with unrestricted $\alpha\in\R^{+}$). This gives
 \begin{equation*}
 f'(r)=-\frac{(n-2)f(r)}{2r}
 \end{equation*}
 on $I$. Setting $ \sigma\definedas \kappa^{-\frac{4}{n}}\,\widetilde{\sigma}$ and recalling \eqref{gf}, \eqref{def:psi}, and \eqref{solutionpsi},  we obtain
\begin{equation*}
g=\frac{df^{2}}{\psi(f)^{2}}+g_{f}=r^{n-2}\,dr^{2}+r^{2}\sigma
\end{equation*}
on $V\definedas I\times\Sigma$, with $(\Sigma,\sigma)$ being Ricci flat but otherwise unrestricted by \eqref{eq:evogf}. Moreover, we find 
\begin{equation*}
\lambda(r)=\frac{(n-1)(n-2)}{2r^{n}}
\end{equation*}
for $r\in I$ by \eqref{eq:lambda}. Consequently, for $\beta=0$, the system $(U,g,f)$ is of Type 2 and systems of Type 2 satisfy $T=0$ on $V$ as well as the static vacuum equations \eqref{vacuum equation1}, \eqref{vacuum equation2}.\\

\underline{Third, for $\alpha,\beta>0$,} we find that both $\alpha,\beta>0$ are unrestricted by \eqref{alphabeta}. We now observe that $\scal_{g_{f}}<0$ by \eqref{Einsteinsimple} which allows us to pick $f_0\in F(f)$ and set 
\begin{align}\label{def:r03}
r_0&\definedas \sqrt{-\frac{(n-1)(n-2)}{\scal_{g_{f_0}}}},\\\label{def:sigma3}
\sigma&\definedas \frac{1}{r_0^2}\,g_{f_0}.
\end{align}
By \eqref{eq:T0concrete}, we find that $g_f=r(f)^2\,\sigma$ for $r\colon F(f)\to\R^+$ given by
\begin{align}\label{eq:r(f)}
r(f)\definedas r_0\left(\frac{\alpha f_0^2+\beta}{\alpha f^2+\beta}\right)^\frac{1}{n-2}.
\end{align}
Plugging this into the trace of \eqref{Einsteinsimple} and exploiting that $\scal_{r^{2}\sigma}=-\frac{(n-1)(n-2)}{r^{2}}$ gives
\begin{equation*}
r_{0}(\alpha f_{0}^{2}+\beta)^{\frac{1}{n-2}}=\frac{n-2}{2\sqrt{\alpha\beta}}
\end{equation*}
 which removes our choice of $f_{0}$ and our definition of $r_{0}$ from the picture, giving 
\begin{align*}
r(f)= \frac{n-2}{2\sqrt{\alpha\beta}\,(\alpha f^2+\beta)^\frac{1}{n-2}}.
\end{align*}
with inverse function $f\colon r(F(f))\to\R^+$ given by
\begin{align*}
f(r)=\sqrt{\frac{1}{\alpha}\left(\left(\frac{n-2}{2\sqrt{\alpha\beta}}\right)^{n-2}\frac{1}{r^{n-2}}-\beta\right)},
\end{align*}
where we note that $f$ is well-defined on the interval $I\definedas r(F(f))\subseteq\R^+$. Using this, we obtain
\begin{align*}
f(r)&=a\sqrt{\frac{b}{r^{n-2}}-1},\\
g&=\frac{dr^{2}}{\frac{b}{r^{n-2}}-1}+r^{2}\sigma
\end{align*}
on $I$ and $V\definedas I\times\Sigma$, respectively, for constants $a,b>0$ given by
\begin{align*}
a&\definedas \sqrt{\frac{\beta}{\alpha}},\\
b&\definedas \frac{1}{\beta}\left(\frac{n-2}{2\sqrt{\alpha\beta}}\right)^{n-2}.
\end{align*}
We note that, other than $a,b>0$, $a,b$ are unrestricted by \eqref{alphabeta} and that $\sigma$ is an arbitrary Einstein metric on $\Sigma$ satisfying $\scal_{\sigma}=-(n-1)(n-2)$. However, it must hold that $I\subseteq(0,b^{\frac{1}{n-2}})$ which is implied by $I=r(F(f))$. Moreover, we find 
\begin{equation}\label{eq:bequation}
\lambda(r)=\frac{(n-1)(n-2)b}{2r^n}
\end{equation}
for $r\in I$ by \eqref{eq:lambda}. Consequently for $\alpha,\beta>0$, the system $(U,g,f)$ is of Type 3 and systems of Type 3 satisfy $T=0$ on $V$ as well as the static vacuum equations \eqref{vacuum equation1}, \eqref{vacuum equation2}.\\

\underline{Last but not least, for $\alpha\beta<0$,} we find that $\alpha,\beta$ are restricted by \eqref{alphabeta} such that
\begin{equation}\label{><restriction}
-\frac{\beta}{\alpha}<f^{2} \text{ when }  \alpha>0 \quad \text{ and }\quad -\frac{\beta}{\alpha}>f^{2} \text{ when }  \alpha<0
\end{equation}
on $U$. We now observe that $\scal_{g_{f}}>0$ by \eqref{Einsteinsimple} which allows us to pick $f_0\in F(f)$ and set 
\begin{align}\label{def:r03-}
r_0&\definedas \sqrt{\frac{(n-1)(n-2)}{\scal_{g_{f_0}}}},\\\label{def:sigma3-}
\sigma&\definedas \frac{1}{r_0^2}\,g_{f_0}.
\end{align}
By \eqref{eq:T0concrete}, we again have that $g_f=r(f)^2\,\sigma$ for $r\colon F(f)\to\R^+$ given by \eqref{eq:r(f)}. Plugging this into the trace of \eqref{Einsteinsimple} and exploiting that $\scal_{r^{2}\sigma}=\frac{(n-1)(n-2)}{r^{2}}$ gives
\begin{equation*}
r_{0}(\alpha f_{0}^{2}+\beta)^{\frac{1}{n-2}}=\frac{n-2}{2\sqrt{-\alpha\beta}},
\end{equation*}
giving 
\begin{align*}
r(f)= \frac{n-2}{2\sqrt{-\alpha\beta}\,(\alpha f^2+\beta)^\frac{1}{n-2}}.
\end{align*}
We note that $I\definedas r(F(f))\subseteq\R^{+}$ is unrestricted when $\alpha>0$ while 
\begin{equation}\label{eq:restricted}
I\subseteq(\frac{n-2}{2\sqrt{-\alpha\beta}\beta^{\frac{1}{n-2}}},\infty)
\end{equation}
when $\alpha<0$ by monotonicity of $r\colon F(f)\to\R^{+}$ and as $F(f)\subseteq\R^{+}$ is restricted only by~\eqref{><restriction}. The inverse function $f\colon r(F(f))\to\R^+$ of $r=r(f)$  given by
\begin{align*}
f(r)=\sqrt{\frac{1}{\alpha}\left(\left(\frac{n-2}{2\sqrt{-\alpha\beta}}\right)^{n-2}\frac{1}{r^{n-2}}-\beta\right)}
\end{align*}
and well-defined on $I$ by the above. Using this, we obtain
\begin{align*}
f(r)&=a\sqrt{\frac{b}{r^{n-2}}+1},\\
g&=\frac{dr^{2}}{\frac{b}{r^{n-2}}+1}+r^{2}\sigma
\end{align*}
on $I$ and $V\definedas I\times\Sigma$, respectively, for constants $a,b>0$ given by
\begin{align*}
a&\definedas \sqrt{-\frac{\beta}{\alpha}},\\
b&\definedas -\frac{1}{\beta}\left(\frac{n-2}{2\sqrt{-\alpha\beta}}\right)^{n-2}.
\end{align*}
We note that, other than $a>0$, $b\neq0$, $a,b$ are not further restricted and that $\sigma$ is an arbitrary Einstein metric on $\Sigma$ satisfying $\scal_{\sigma}=(n-1)(n-2)$. There is no restriction on $I$ when $b>0$ and the only restriction $I\subseteq(-b^{\frac{1}{n-2}},\infty)$ when $b<0$ by \eqref{eq:restricted} and the definition of $b$. This is consistent with $\frac{b}{r^{n-2}}+1>0$ on $\R^{+}$ when $b>0$ and $\frac{b}{r^{n-2}}+1>0$ precisely on $(-b^\frac{1}{n-2},\infty)$ when $b<0$. Moreover, we recover \eqref{eq:bequation} for $r\in I$ by \eqref{eq:lambda}. Consequently, for $\alpha\beta<0$, the system $(U,g,f)$ is of Type 4 and systems of Type 4 satisfy $T=0$ on $V$ as well as the static vacuum equations \eqref{vacuum equation1}, \eqref{vacuum equation2}.
\end{proof}
\newpage
\begin{corollary}[Options for $\lambda$ and ODEs for $\Vert\nabla f\Vert^2$]\label{coro:locally}
It will be useful later to observe that the proof of \Cref{thm:locally} shows that $\lambda$ and $\Vert\nabla f\Vert^2$ satisfy
\vspace{1ex}
\begin{center}
\begin{tabular}{c|c|c}
& \\[-1ex]
\quad Type 1 \quad & \quad $\lambda\equiv0$ \quad  & \quad $\nabla\Vert\nabla f\Vert^2=0$ \quad\\[2ex]\hline && \\[-1ex]
\quad Type 2 \quad&\quad $\lambda=\frac{2(n-1)\|\nabla f\|^{2}}{(n-2)f^{2}}$ \quad & \quad$\nabla\left(\frac{\Vert\nabla f\Vert^2}{f^\frac{4(n-1)}{(n-2)}}\right)=0$ \quad\\[4ex]\hline && \\[-1ex]
\quad Type 3 \quad&\quad $\lambda=\frac{2(n-1)\|\nabla f\|^{2}}{(n-2)(f^{2}+a^{2})}$ \quad\ & \quad$\nabla\left(\frac{\Vert\nabla f\Vert^2}{(f^2+a^2)^\frac{2(n-1)}{(n-2)}}\right)=0$ \quad\\[4ex]\hline && \\[-1ex]
\quad Type 4 \quad & \quad $\lambda=\frac{2(n-1)\|\nabla f\|^{2}}{(n-2)(f^{2}-a^{2})}$\quad\quad & \quad $\nabla\left(\frac{\Vert\nabla f\Vert^2}{\vert f^2-a^2\vert^\frac{2(n-1)}{(n-2)}}\right)=0$ \quad\\[3ex]
\end{tabular}
\end{center}
\vspace{1ex}
on $V$. Note that $a\notin f(V)$ so the ODE in Type 4 is also well-defined.
\end{corollary}

\begin{remark}[The corresponding ODE in the (quasi-)Schwarzschild case]\label{rem:ODE}
Let $(M^{n},g,f)$ be a static system with $f\neq1$. Then the identity
\begin{equation}\label{mainformulap3}
\nabla\|\nabla f\|^{2}+\frac{4(n-1)}{(n-2)}\frac{f\|\nabla f\|^{2}\,\nabla f}{1-f^{2}}=0
\end{equation}
with left hand side appearing in the divergence identity \eqref{mainformulap} is equivalent to the ODE
		\begin{align}\label{RegFoli}
			\nabla\left[\frac{\|\nabla f\|^{2}}{\vert {f^{2}-1}\vert^{\frac{2(n-1)}{n-2}}}\right] &= 0
		\end{align}
		on $M$, a special case of the ODE in Type 4, namely for $a=1$, see \Cref{coro:locally}. It holds in all quasi-Schwarzschild systems (with nonzero mass) as can be seen by direct computations.
		\end{remark}

The following global characterization of static vacuum systems with $T=0$ can be proved by appealing to real analyticity. We choose to prove it \lq by hand\rq\ as this adds some useful insights.

\begin{corollary}[Global characterization of $T=0$]\label{coro:globally}
Let $(M^n,g,f)$, $n\geq3$, be a static vacuum system. Then $T=0$ on $M$ if and only if either $f\equiv{const}$ on $M$ or $(M,g)$ is (globally) isometric to a suitable piece of one of the Riemannian manifolds of Types 1, 2, 3, or 4 in \Cref{thm:locally} and $f$ is regular on $M$ and corresponds to the corresponding restriction of the lapse function of the same system of Type 1, 2, 3, or 4 under this isometry.
\end{corollary}
\begin{proof}
First, if $(M,g,f)$ is a piece of a static vacuum system of Type 1, 2, 3, or 4, we know from \Cref{coro:locally} that $f$ is regular on $M$. Next, from \Cref{thm:locally}, we know that $T=0$ on $M$. Also, if $(M,g,f)$ is a static vacuum system with $f\equiv\text{const}$ on $M$, we trivially know that $T=0$ by definition of $T$ in \eqref{tensorT}.

On the other hand, if $(M,g,f)$ is a static vacuum system with $T=0$ and $f\neq\text{const}$ on $M$, we know from \Cref{thm:locally} that $(M,g,f)$ looks like one of the systems of Types 1, 2, 3, or 4 locally near each regular point. From the ODEs in \Cref{coro:locally}, we can deduce that there exists a constant $\rho>0$ such that
\vspace{1ex}
\begin{center}
\begin{tabular}{c|c}
& \\[-1ex]
\quad Type 1 \quad &  \quad $\Vert\nabla f\Vert^2=\rho$ \quad\\[2ex]\hline & \\[-1ex]
\quad Type 2 \quad&\quad $\Vert\nabla f\Vert^2=\rho f^\frac{4(n-1)}{(n-2)}$ \quad\\[3ex]\hline & \\[-1ex]
\quad Type 3 \quad&\quad $\Vert\nabla f\Vert^2=\rho(f^2+a^2)^\frac{2(n-1)}{(n-2)}$ \quad\\[3ex]\hline & \\[-1ex]
\quad Type 4 \quad & \quad  $\Vert\nabla f\Vert^2=\rho\vert f^2-a^2\vert^\frac{2(n-1)}{(n-2)}$ \quad\\[3ex]
\end{tabular}
\end{center}

on this neighborhood. Here, we know that $\rho>0$ because the neighborhood contains no critical points of $f$ and the factors multiplied by $\rho$ are strictly positive on said neighborhood by the assumptions on the interval $I$. Now suppose towards a contradiction that $f$ has $\crit\neq\emptyset$. First note that  by smoothness of all involved quantities and because of the geometric nature of the radial coordinate $r$, if two such (open) neighborhoods overlap, they must be of the same type, have the same constants $a$, $b$, and $\rho>0$ (where applicable), and the same Einstein metric $\sigma$ on the intersections of the level sets of $r$ in those neighborhoods. Consequently, each connected component of $M\setminus\crit$ has a uniform type, fixed constants $a$, $b$, and $\rho>0$ (where applicable), and a global coordinate $r$. 

As $\crit$ is closed by continuity and $\crit\neq M$ by assumption, we know that every connected component $U$ of $M\setminus\crit\neq\emptyset$ is open in $M$ and has $\partial U\subseteq\crit$. If $\partial U=\emptyset$ then $U$ is closed in $M$ and by connectedness of $M$ we have $U=M$ and hence $\crit=\emptyset$ as desired. Now suppose towards a contradiction that $\partial U\neq\emptyset$. By continuity of $f$ and $\Vert\nabla f\Vert^{2}$, this leads to the contradiction $\rho=0$ on $U$ in case $U$ is of Types 1, 2, or 3 because $f>0$ and $f^2+a^2>0$ by the definition of static systems. 

The same contradiction arises if $U$ is of Type 4 unless $f\vert_{\partial U}= a$. Combining this with the formula for $f$ in Type 4, we learn that $r\to\infty$ when approaching $\partial U$. Now pick $x\in U$ and $q\in\partial U$ and let $\gamma\colon[0,1]\to M$ denote the geodesic connecting $\gamma(0)=x$ to $\gamma(1)=q$ in $M$. If $\gamma\vert_{[0,1)}$ does not run entirely within $U$ or in other words if $\gamma([0,1))\cap\partial U\neq\emptyset$, there exists a smallest parameter $0<s_{\#}<1$ such that $\gamma(s_{\#})\in \partial U$ while $\gamma\vert_{[0,s_{\#})}\subset U$ because $\gamma([0,1])$ is compact and $\partial U$ is closed. Thus, without loss of generality, we can and will assume that indeed $\gamma([0,1))\subset U$, replacing the original endpoint $q$ by $\gamma(s_{\#})$ and reparametrizing $\gamma$ accordingly with an affine parameter transformation. {As we have assumed that $\gamma([0,1))\subset U$, we can now consider the function $S\colon[0,1)\to\R^{+}$ given by $S\definedas r\circ\gamma$. 
Then computing in local neighborhoods $I\times\Sigma$ on which we have coordinates $(r,\varphi^{K})_{K=1}^{n-1}$ as constructed above, we have $\dot{\gamma}=\dot{S}\,\partial_{r}\vert_{\gamma}+\dot{X}^{K}\partial_{\varphi^{K}}\vert_{\gamma}$ on all suitably small intervals $J\subseteq[0,1)$, where $X^{K}\definedas\varphi^{K}\circ\gamma$ on $J$. The radial part of the geodesic equation for $\gamma$ gives
\begin{align*}
\ddot{S}-\frac{f'(S)}{f(S)}\dot{S}^{2}=\frac{Sf(S)^{2}}{a^{2}}\,\sigma_{\gamma}(\dot{X},\dot{X})
\end{align*}
on $J$. In particular, if $\dot{S}(s_{*})=0$ for some $s_{*}\in [0,1)$, we learn from $\dot{\gamma}(s_{*})\neq0$ that $\ddot{S}(s_{*})>0$. Hence $S$ has a strict local minimum at $s_{*}$. In particular, $S$ can have at most one critical point in $[0,1)$. Without loss of generality, we will assume that there is no critical point of $S$ on $(0,1)$, replacing the original starting point $x$ by $\gamma(s_{*})$ and reparametrizing $\gamma$ accordingly with an affine parameter transformation. Moreover, as we have seen that $S(s)\to\infty$ for $s\to1$, it follows that $\dot{S}>0$ on $(0,1)$. For any $0<\varepsilon<\frac{1}{2}$, this allows us to estimate the length $L[\gamma]$ of $\gamma$ from below by
\begin{align*}
L[\gamma]&\geq L[\gamma\vert_{(\varepsilon,1-\varepsilon)}]=\int_{\varepsilon}^{1-\varepsilon}\Vert\dot{\gamma}(s)\Vert\,ds\geq\int_{\varepsilon}^{1-\varepsilon}\frac{a\dot{S}(s)}{f(S(s))}\,ds\\
&=\int_{\varepsilon}^{1-\varepsilon}\frac{\dot{S}(s)}{\sqrt{1+\frac{b}{S(s)^{n-2}}}}\,ds=\int_{c_{\varepsilon}}^{d_{\varepsilon}}\frac{1}{\sqrt{1+\frac{b}{r^{n-2}}}}\,dr\asdefined E_{\varepsilon}
\end{align*}
for suitable constants $S(0)<c_{\varepsilon}<d_{\varepsilon}<\infty$ with $c_{\varepsilon}\to S(0)$ and $d_{\varepsilon}\to\infty$ as $\varepsilon \to 0$. As is well-known, this shows that $E_{\varepsilon}\to\infty$ as $\varepsilon\to0$, giving the desired contradiction. Hence $\crit=\emptyset$. This also proves the remaining claims as $M$ is connected and smoothly partitioned by the regular level sets of $f$ and because $r$ is the scalar curvature radius and thus a geometric coordinate.}
\end{proof}

\section{Recovering the Schwarz\-schild geometry}\label{sec:rigidity}
In this section, we will prove \Cref{thm:rigidity}. Its proof heavily relies on \Cref{coro:globally}. Before we start, let us make the following remark.

\begin{remark}[Simplified rigidity argument]\label{rem:simplify}
Note that proving \Cref{thm:rigidity} with the additional assumption that \eqref{mainformulap3} holds on $M$ would be somewhat simpler, readily establishing the absence of critical points as in the proof of \Cref{coro:globally} as well as ruling out Types 1--3 in \Cref{thm:locally} and fixing $a=1$ in Type 4. While assuming \eqref{mainformulap3} would be fully sufficient for getting rigidity in all claimed geometric inequalities when $p>p_{n}$, it is important to prove \Cref{thm:rigidity} without this extra assumption to be able to include the threshold case $p=p_{n}$.
\end{remark}

	\begin{proof}[Proof of \Cref{thm:rigidity}]
		By \Cref{coro:globally}, we know that $f$ regularly foliates $M\cup\,\partial M$ as $f=\text{const}$ on $M$ contradicts \Cref{rmk:lapse-asymp}. Also we know that $(M,g,f)$ is isometric to a suitable piece of a system of either of the Types 1--4 with fixed constants $a,b$. As in our setup here $M$ is regularly foliated by closed level sets of $f$ up to the boundary e.g.\ via \Cref{lem:asymptotics} and \Cref{rem:complete}, we know that $(M,g,f)$ is in fact globally isometric to a system of either of the Types 1--4 with fixed constants $a,b$, a fixed closed Einstein manifold $(\Sigma,\sigma)$, and a fixed interval $I\subseteq\R^+$ {which satisfies} $I\subseteq(-b^\frac{1}{n-2},\infty)$ in the $b>0$ case of Type 4. In particular, $(M,g)$ is a warped product.
		
	Next, we know from the proof of \Cref{coro:globally} that $\Vert\nabla f\Vert$ is fully determined by $f$ via a precise formula, up to a multiplicative constant $\rho>0$. Now recall from \Cref{lem:asymptotics} that $\Vert\nabla f\Vert\to0$ as $\vert x\vert\to\infty$ for asymptotically flat static vacuum systems and note that this rules out Types 1--3 as $f\to1$ as $\vert x\vert\to\infty$. For Type 4, the same argument fixes $a=1$.

	Now let $x \colon M \setminus K \longrightarrow \R^n \setminus \overline{B}$ denote a diffeomorphism making $(M,g,f)$ asymptotically flat and denote the induced coordinates by $(x^{i})$. For convenience, let us switch to standard polar coordinates $(\vert x\vert,\theta^{J})_{J=1}^{n-1}$ for $\R^{n}$ associated to $(x^{i})$ so that $(\partial_{\theta^{K}})^{i}=O(\vert x\vert)$ as $\vert x\vert\to\infty$ and $\delta(\partial_{\theta^{I}},\partial_{\theta^{J}})=r^{2}(g_{\mathbb{S}^{n-1}})(\partial_{\theta^{I}},\partial_{\theta^{J}})$ for $I,J,K=1,\dots,n-1$. Thinking of the coordinate $r$ in the Type 4 representation of our system as a function $r\colon M\to\R^+$, the asymptotic assumptions we made on $f$ in \eqref{eq:f-asymp} give
\begin{align*}
r&=\left(-\frac{b}{2m}\right)^\frac{1}{n-2}\vert x\vert+o_2(\vert x\vert)
\end{align*}
as $\vert x\vert\to\infty$. On the other hand, recalling that $r$ denotes the scalar curvature radius of the level sets of $f$ and plugging in $\scal=0$ as well as the asymptotic decay assertions from \Cref{lem:asymptotics} and in particular \eqref{asyH}, the twice contracted Gau{\ss} equation gives
\begin{align*}
\frac{(n-1)(n-2)}{\left(-\frac{b}{2m}\right)^\frac{2}{n-2}\vert x\vert^2}-\frac{n-2}{n-1}\left(\frac{n-1}{\vert x\vert}\right)^2&=o(\vert x\vert^{-2})
\end{align*}
as $\vert x\vert\to\infty$, where we have used that we already know that all level sets of $f$ are totally umbilic. This gives $b=-2m$ (in line with \Cref{rem:quasi}) and thus $r=\vert x\vert+o_2(\vert x\vert)$ as $\vert x\vert\to\infty$. Taking a $\partial_{\theta^{K}}$-derivative of this identity, one sees that
		\begin{align*}
			\frac{\partial r}{\partial \theta^{K}}&=\partial_{\theta^{K}}o_{2}(\vert x\vert)=(\partial_{\theta^{K}})^{i}o_{1}(1)=o_{1}(\vert x\vert)
		\end{align*}
as $\vert x\vert\to\infty$ for $K=1,\dots,n-1$.		
		
	 To see that $(\Sigma,\sigma)$ is indeed isometric to the standard round $(n-1)$-sphere of radius $1$, we need to carefully consider the asymptotic decay of $g$. As implicitly done above and as usual, we will interpret $\sigma$ as an $r$-independent tensor field on $M$ which is applied to tensor fields on $M$ by first projecting them tangentially to the level sets $\Sigma_{f}\approx\{r(f)\}\times\partial M$ of $f$. Similarly, we interpret $g_{\mathbb{S}^{n-1}}$ as an $\vert x\vert$-independent tensor field on $\R^{+}\times\mathbb{S}^{n-1}$ by projection onto round spheres as usual. Exploiting this convention, our asymptotic flatness assumptions in \Cref{def2} translated to spherical polar coordinates and the above insights give
		\begin{align*}
			\vert x\vert^{2}(g_{\mathbb{S}^{n-1}})_{\theta^{I}\theta^{J}}+o(\vert x\vert^{2})=g_{\theta^{I}\theta^{J}}&=o(\vert x\vert^{2})+\vert x\vert^2(1+o(1))\sigma_{\theta^{I}\theta^{J}}
		\end{align*}
		as $\vert x\vert\to\infty$ for $I,J=1,\dots,n-1$. This can easily be rewritten as
		\begin{align*}
			\sigma_{\theta^{I}\theta^{J}}=(1+o(1))(g_{\mathbb{S}^{n-1}})_{\theta^{I}\theta^{J}}=(g_{\mathbb{S}^{n-1}})_{\theta^{I}\theta^{J}}+o(r)
		\end{align*}
		as $r\to\infty$ for $I,J=1,\dots,n-1$. As $\sigma$ is independent of $r$, this allows us to conclude $(\Sigma,\sigma)$ is isometric to the round $(n-1)$-sphere of radius $1$ as desired.
		
 Thus, as $b=-2m$, we deduce that $(M,g)$ must be isometric to the piece $[r_{0},\infty)\times\mathbb{S}^{n-1}$  of the Schwarzschild manifold $(M^{n}_{m},g_{m})$ of mass $m$  for $f_0>0$, with $f$ corresponding to $f_{m}$ while $(M\setminus\partial M,g)$ must be isometric to the piece $(r_0,\infty)\times\mathbb{S}^{n-1}$  of the Schwarzschild manifold $(M^{n}_{m},g_{m})$ of mass $m$ when $f_{0}=0$, with $f$ corresponding to $f_{m}$. Switching to isotropic coordinates then also allows us to conclude that the claims extend to $\partial M$ when $f_{0}=0$. Here, $r_0\definedas r(\partial M)$ denotes the scalar curvature radius of $\partial M$.
	\end{proof}

\newpage
	\section{The divergence identity}\label{sec:proofs}
	With the help of \Cref{lem21}, we are now in the position to prove \Cref{maintheoremp}.
	\begin{proof}[Proof of \Cref{maintheoremp}]
		First, note that $f\neq1$ on $M$ by assumption. Then clearly
		\begin{align*}
			\left\|\nabla\|\nabla f\|^{2}+\frac{4(n-1)}{n-2}\frac{f\|\nabla f\|^{2}\,\nabla f}{1-f^{2}}\right\|^{2}&=\|\nabla\|\nabla f\|^{2}\|^{2}+\frac{8(n-1)}{n-2}\frac{f\|\nabla f\|^{2}}{1-f^{2}}\langle\nabla\|\nabla f\|^{2},\nabla f\rangle\\ 
			&\quad+\frac{16(n-1)^{2}}{(n-2)^{2}}\frac{f^{2}\|\nabla f\|^{6}}{(1-f^{2})^{2}}
		\end{align*}
		holds on $M$. Combining this with \Cref{lem21}, we get that
		\begin{align*}
			&\frac{(n-2)^{2}}{(n-1)^{2}}fF(f)\|T\|^{2} +F(f)\left(\frac{n}{2(n-1)}+\frac{p-3}{2}\right)\frac{1}{f}\left\|\nabla\|\nabla f\|^{2}+\frac{4(n-1)}{n-2}\frac{f\|\nabla f\|^{2}}{1-f^{2}}\nabla f\right\|^{2}\\
			&\quad= \|\nabla f\|^{2}\,F(f)\left[\frac{1}{f}\Delta\|\nabla f\|^{2} - \frac{1}{f^{2}}\langle\nabla\|\nabla f\|^{2},\nabla f\rangle\right.\\
			&\quad\quad\phantom{\|\nabla f\|^{2}}\quad\left.+\left(\frac{4n}{n-2}+\frac{4(n-1)(p-3)}{n-2}\right)\frac{1}{1-f^{2}}\langle\nabla\|\nabla f\|^{2},\nabla f\rangle\right.\\
			&\quad\quad\phantom{\|\nabla f\|^{2}}\quad\left.+\left(\frac{8n(n-1)}{(n-2)^{2}}+\frac{8(n-1)^{2}(p-3)}{(n-2)^{2}}\right)\frac{f}{(1-f^{2})^{2}}\|\nabla f\|^{4}\right]\\
			&\quad\quad+\frac{p-3}{2}\frac{F(f)}{f}\Vert\nabla \Vert \nabla f\Vert^{2}\Vert^{2}
		\end{align*}
		holds on $M$ for any smooth function $F\colon[0,1)\cup(1,\infty)\to\R$ and any $0<f<1$ or $f>1$. Also, using~\eqref{vacuum equation1}, one computes
		\begin{align*}
			&\diver\left(\frac{F(f)}{f}\|\nabla f\|^{p-3}\,\nabla\|\nabla f\|^{2}+G(f)\|\nabla f\|^{p-1}\,\nabla f\right)\\
			&\quad=\frac{F(f)}{f}\|\nabla f\|^{p-3}\Delta\|\nabla f\|^{2}+\left(\frac{F'(f)}{f}-\frac{F(f)}{f^{2}}+\frac{p-1}{2}G(f)\right)\|\nabla f\|^{p-3}\langle\nabla\|\nabla f\|^{2},\nabla f\rangle\\
			&\quad\quad+\frac{p-3}{2}\frac{F(f)}{f}\|\nabla f\|^{p-5}\,\|\nabla \Vert\nabla f\Vert^{2}\|^{2}+G'(f)\|\nabla f\|^{p+1}
		\end{align*}
		on $M\setminus\crit$ for any $p\in\R$ and any smooth functions $F,G\colon[0,1)\cup(1,\infty)\to\R$, where $'=\frac{d}{df}$. Combining these two identities, we find
		\begin{align*}
			&\Vert\nabla f\Vert^{p-3}\,F(f)\left[\frac{(n-2)^{2}f}{(n-1)^{2}}\|T\|^{2}\right.\\
			&\phantom{\Vert\nabla f\Vert^{p-3}[}\left.+\left(\frac{n}{2(n-1)}+\frac{p-3}{2}\right)\frac{1}{f}\left\|\nabla\|\nabla f\|^{2}+\frac{4(n-1)}{(n-2)}\frac{f\|\nabla f\|^{2}}{(1-f^{2})}\nabla f\right\|^{2}\right]
			\end{align*} 	
			\begin{align*}
			&\quad= \|\nabla f\|^{2}\left[\diver\left(\frac{F(f)}{f}\|\nabla f\|^{p-3}\,\nabla\|\nabla f\|^{2}+G(f)\|\nabla f\|^{p-1}\,\nabla f\right)\right.\\
			&\quad\quad\phantom{\|\nabla f\|^{2}}\quad\left. +\left\{-\left(\frac{F'(f)}{f}-\frac{F(f)}{f^{2}}+\frac{p-1}{2}G(f)\right)\|\nabla f\|^{p-3}\langle\nabla\|\nabla f\|^{2},\nabla f\rangle-G'(f)\|\nabla f\|^{p+1}\right.\right.\\
			&\quad\quad\quad\quad\phantom{\|\nabla f\|^{2}}\quad\left.\left.-\frac{F(f)}{f^{2}}\|\nabla f\|^{p-3}\langle\nabla\|\nabla f\|^{2},\nabla f\rangle\right.\right.\\
			&\quad\quad\quad\quad\phantom{\|\nabla f\|^{2}}\quad\left.\left.+\left(\frac{4n}{n-2}+\frac{4(n-1)(p-3)}{n-2}\right)\frac{F(f)}{1-f^{2}}\|\nabla f\|^{p-3}\langle\nabla\|\nabla f\|^{2},\nabla f\rangle\right.\right.\\
			&\quad\quad\quad\quad\phantom{\|\nabla f\|^{2}}\quad \left.\left.+\frac{2(n-1)}{n-2}\left(\frac{4n}{n-2}+\frac{4(n-1)(p-3)}{n-2}\right)\frac{fF(f)}{(1-f^{2})^{2}}\|\nabla f\|^{p+1} \right\}\right]
		\end{align*}
		on $M\setminus\crit$ for any $p\in\R$ and any smooth functions $F,G\colon[0,1)\cup(1,\infty)\to\R$. Now, plugging in the precise forms of $F$ and $G$ given in \eqref{eq:F} and \eqref{eq:G} and observing that they solve the system of ODEs
		\begin{align*}
			F'(t)&=4\left(\frac{(n-1)(p-1)}{(n-2)}-1\right)\frac{tF(t)}{1-t^{2}}-\frac{p-1}{2}tG(t),\\
			G'(t)&=\frac{8(n-1)}{n-2}\left(\frac{(n-1)(p-1)}{(n-2)}-1\right)\frac{tF(t)}{(1-t^{2})^{2}}
		\end{align*}
		for $t\in[0,1)\cup(1,\infty)$, one detects that the term inside the braces vanishes and obtains \eqref{mainformulap} on $M\setminus\crit$ for all $p\in\R$. 
		
Now let us address the claimed continuity of the divergence on $M$ for $p\geq3$: First, we note that second term in the argument of the divergence is continuously differentiable on $M$ for $p\geq3$ by smoothness of $f$ and $\|\nabla f\|^{2}$, as one immediately sees upon rewriting $\|\nabla f\|^{p-1}=(\|\nabla f\|^{2})^{\frac{p-1}{2}}$. Considering the first term $\Vert\nabla f\Vert^{p-3}\,\nabla\Vert\nabla f\Vert^{2}$ in the argument of the divergence, we note that it vanishes at all critical points of $f$ as $\Vert\nabla f\Vert^{2}\geq0$ attains a minimum there, hence it is continuous on $M$. Next, its derivative (defined on $M\setminus\crit$) extends continuously across critical points because $\nabla\Vert\nabla f\Vert^{2}=2\nabla^{2}f(\nabla f,\cdot)$, so that the derivative of $\|\nabla f\|^{p-3}$ is bounded from above by $\|\nabla f\|^{p-4}\|\nabla^{2}f\|$ by the Cauchy--Schwarz inequality away from $\crit$; as it gets multiplied by another factor of $\nabla\|\nabla f\|^{2}$, we recover a continuous upper bound multiplied by $\|\nabla f\|^{p-3}$ which goes to zero for $p>3$. This bound has a numerical factor $p-3$, hence it identically vanishes for $p=3$. Altogether, noting that the other contributions to the derivative of $\Vert\nabla f\Vert^{p-3}\,\nabla\Vert\nabla f\Vert^{2}$ are smooth anyways, this establishes that the divergence and thus \eqref{mainformulap} continuously extends to $M$ for $p\geq3$ as claimed. Last but not least, the right hand side of \eqref{mainformulap} is manifestly non-negative if \eqref{eq:thresholdp} holds which gives \eqref{mainformulap0}.
	\end{proof}
		
	It may be worth noting that the free constants $c,d\in\R$ in the statement of \Cref{maintheoremp} correspond to the free constants of integration of the ODEs for $F$ and $G$ arising in its proof. Moreover, it may be useful to note that, for the Schwarz\-schild system $(M^{n}_{m},g_{m},f_{m})$ of mass $m\neq0$, the first term of the right-hand side of \eqref{mainformulap} vanishes by \Cref{lem20}, while the second term manifestly vanishes by an explicit computation, see also \Cref{rem:ODE}. Moreover, $\crit_{m}=\emptyset$. Hence by \Cref{maintheoremp}, the vector field inside the divergence of \eqref{mainformulap} is divergence-free in the Schwarz\-schild case and thus gives rise to a three-parameter family (parametrized by $c,d,p$) of conserved quantities
	\begin{align*}
		\int_{\{f_{m}=z\}}\left[\frac{F(f_{m})}{f_{m}}\Vert\nabla_{m}f_{m}\Vert_{m}^{p-3}g_{m}(\nabla_{m}\|\nabla_{m} f_{m}\|^{2}_{m},\nu_{m})+G(f_{m})\|\nabla_{m} f_{m}\|^{p-1}_{m}g_{m}(\nabla_{m}f_{m},\nu_{m})\right]\!dS_{m},
	\end{align*}
	by the divergence theorem. Here, $dS_{m}$, $\nabla_{m}$, $\Vert\cdot\Vert_{m}$, and $\nu_{m}$ denote the area element on $\{f_{m}=z\}$, the covariant derivative, and the tensor norm induced by $g_{m}$, and the $g_{m}$-unit normal to $\{f_{m}=z\}$ pointing towards infinity. Evaluating this conserved quantity for $z\to1$, one finds\footnote{See also the more general discussion and the computations in \Cref{sec:inequalities}.} $-\mathcal{F}^{\,c,d}_{\!p}(m)$ from \eqref{eq:Fmcd}.
	
	\section{Geometric inequalities}\label{sec:inequalities}
	In this section, we will prove the geometric inequality in \Cref{maintheoremp2} and its equivalent formulation \Cref{maintheoremp3}. To do so, we will exploit the divergence identity~\eqref{mainformulap} by estimating its right-hand side from below by zero and applying the divergence theorem in combination with the asymptotic flatness assumptions. To apply a suitably adapted version of the divergence theorem, we will need to first assert integrability of the left hand side in \eqref{mainformulap}; our proof of said integrability is inspired by \cite[Section 4]{Anabel}, see also \cite{Agostiniani.2020}. For the equality claim, we will rely on the rigidity assertion of \Cref{thm:rigidity}. We will then also show how to derive \Cref{teoclassifica1,teoclassifica2} from \Cref{maintheoremp2,maintheoremp3}.
	
	\begin{remark}\label{extremalF}
		In the setting of \Cref{maintheoremp2}, consider first the case when $f_{0}\in[0,1)$. We have $F\geq0$ on $[f_{0},1)$ if and only if both $c+d\geq0$ and $cf_{0}^{2}+d\geq0$, see also \Cref{fig}. In particular,  $F>0$ holds on $(f_{0},1)$ provided that in addition we do  not have $c=d=0$. Similarly if $f_{0}\in(1,\infty)$, we have $F\geq0$ on $(1,f_{0}]$ if and only if both $c+d\geq0$ and $cf_{0}^{2}+d\geq0$ and  $F>0$ on $(1,f_{0})$ if in addition $c=d=0$ does not hold, see also \Cref{fig2}. This explains the assumptions made on $c,d$ in \Cref{maintheoremp2}.
	\end{remark}
	
	\begin{figure}[h]
		\includegraphics[scale=0.75]{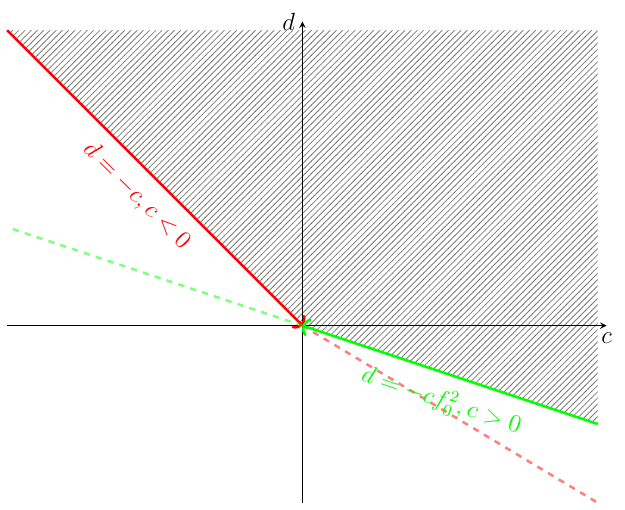}
		\caption{The shaded region together with the \textcolor{red}{red semi-axis} represents all $(c,d)\in\R^{2}$ such that $F>0$ on $[f_{0},1)$ for some $f_{0}\in[0,1)$. On the \textcolor{green}{green semi-axis}, $F(f_{0})=0$ and $F>0$ on $(f_{0},1)$ so in particular $F\geq0$.}\label{fig}
	\end{figure}
	
	\begin{figure}[h]
		\includegraphics[scale=0.75]{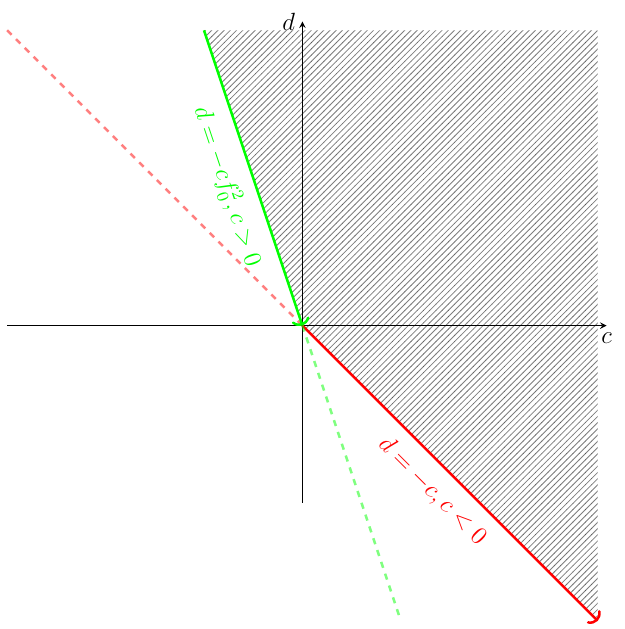}
		\caption{The shaded region together with the \textcolor{red}{red semi-axis} represents all $(c,d)\in\R^{2}$ such that $F>0$ on $(1,f_{0}]$ for some $f_{0}\in(1,\infty)$. On the \textcolor{green}{green semi-axis}, $F(f_{0})=0$ and $F>0$ on $(1,f_{0})$ so in particular $F\geq0$.}\label{fig2}
	\end{figure}
	
	\begin{proof}[Proof of \Cref{maintheoremp2}]
		First of all, by \Cref{rmk:lapse-asymp} and \Cref{Smarr}, we know that $m\neq0$ and $\kappa\neq0$, with $m,\kappa>0$ if $f_{0}\in[0,1)$ and $m,\kappa<0$ if $f_{0}\in(1,\infty)$. Next, note that the right-hand side of \eqref{mainformulap} is non-negative on $M\setminus\crit$ if $c,d\in\R$ satisfy $c+d\geq0$, $cf_{0}^{2}+d\geq0$ as this gives $F(f)\geq0$ by \Cref{extremalF}. Now set
		\begin{align*}
			\mathcal{D}\definedas\diver\left(\frac{F(f)}{f}\|\nabla f\|^{p-3}\,\nabla\|\nabla f\|^{2}+G(f)\|\nabla f\|^{p-1}\,\nabla f\right)
		\end{align*}
		on $M\setminus\crit$. Then by \Cref{maintheoremp}, $\mathcal{D}$ is non-negative and satisfies \eqref{mainformulap} on $M\setminus\crit$. Aiming for an application of the divergence theorem to $\mathcal{D}$, let us show that $\mathcal{D}$ can be extended to a $dV$-integrable function on $M\cup\,\partial M$, where the volume measure $dV$ naturally extends to $\partial M$ by smoothness of the metric $g$. Let us first extend $\mathcal{D}$ to $\partial M$. As $f$ is regular in a neighborhood of $\partial M$, $\mathcal{D}$ continuously extends to $\partial M$; this is immediate when $f_0>0$ and follows from \eqref{eq002} for $f_0=0$ via 
		\begin{align*}
		\frac{F(f)}{f}\,\nabla\Vert\nabla f\Vert^{2}&=2F(f)\Ric(\nabla f,\cdot).
		\end{align*}
		 Thus $\mathcal{D}$ is $dV$-integrable on a regular neighborhood of $\partial M$. To analyze the behavior of $\mathcal{D}$ towards the asymptotic end of $(M,g,f)$, recall that by the asymptotic decay established in \Cref{lem:asymptotics}, we know that $f$ has no critical points in a suitable neighborhood of infinity. This means we can choose a compact subset $K\subseteq M$ such that $(M\setminus K)\cap\crit=\emptyset$ and a diffeomorphism $x \colon M \setminus K \longrightarrow \R^n \setminus \overline{B}$ making $(M,g,f)$ asymptotically flat. The asymptotic assumption \eqref{eq:f-asymp} implies that
		\begin{align*}
			F(f)&=F(f_{m})+o(\vert x\vert^{(n-1)(p-1)-(n-2)}),\\
			F'(f)&=F'(f_{m})+o(\vert x\vert^{(n-1)(p-1)}),\\
			G(f)&=G(f_{m})+o(\vert x\vert^{(n-1)(p-1)}),\\
			G'(f)&=G'(f_{m})+o(\vert x\vert^{(n-1)(p-1)+2(n-2)})
		\end{align*}
		as $\vert x\vert\to\infty$, which can be verified most easily via the ODEs for $F$ and $G$ printed in the proof of \Cref{maintheoremp}. To study the asymptotics of the first term in $\mathcal{D}$, we apply the Bochner formula and the static vacuum equation \eqref{vacuum equation2} to see that
		\begin{align}\label{eq:Bochner}
			\frac{1}{2}\Delta \Vert\nabla f\Vert^{2}&=\Vert \nabla^{2}f\Vert^{2}+\Ric(\nabla f,\nabla f)
		\end{align}
		on $M$, as otherwise we would have to deal with the asymptotic behavior of third derivatives of $f$ about which we have not made any assumptions. Doing so, we find that
		\begin{align*}
			&\diver\left(\frac{F(f)}{f}\|\nabla f\|^{p-3}\,\nabla\Vert\nabla f\Vert^{2}\right)\\
			&=\left(\frac{F'(f)}{f}-\frac{F(f)}{f^{2}}\right)\|\nabla f\|^{p-3}\,\langle\nabla f,\nabla\Vert\nabla f\Vert^{2}\rangle+\frac{p-3}{2}\frac{F(f)}{f}\|\nabla f\|^{p-5}\Vert\nabla\Vert\nabla f\Vert^{2}\Vert^{2}\\
			&\quad+\frac{2F(f)}{f}\|\nabla f\|^{p-3}\,\left(\Vert\nabla^{2}f\Vert^{2}+\Ric(\nabla f,\nabla f)\right),
		\end{align*}
		on $M\setminus K$. Taken together with \Cref{lem:asymptotics}, we find
		\begin{align*}
			\mathcal{D}&=\mathcal{D}_{m}+o(\vert x\vert^{-n})=o(\vert x\vert^{-n})
		\end{align*}
		as $\vert x\vert\to\infty$. Here, $\mathcal{D}_{m}$ denotes the divergence $\mathcal{D}$ for $f=f_{m}$ and $g=g_{m}$ and we are using that we have seen that $\mathcal{D}_{m}=0$ at the end of \Cref{sec:proofs}. Hence by~\eqref{asyvolint}, $\mathcal{D}$ is $dV$-integrable on $M\setminus K$. It remains to study the $dV$-integrability of $\mathcal{D}$ near $\crit$. As the divergence theorem readily applies when $\crit=\emptyset$, we will assume without loss of generality that $\crit\neq\emptyset$.
		
To establish $dV$-integrability near $\crit$, let us first recall from the work of Cheeger--Naber--Valtorta~\cite{Cheeger} and Hardt--Hoffmann-Ostenhof--Hoffmann-Ostenhof--Nadirashvili \cite{Hardt} that $\crit$ is a set of Hausdorff dimension at most $n-2$ as $f$ is harmonic. Now note that $dV$ is absolutely continuous with respect to the $n$-dimensional Hausdorff measure $\mathcal{H}^{n}$ and vice versa, with bounded densities, respectively, by \eqref{asyvolint}. This gives $dV(\crit)=0$. Setting
\begin{align}
W_{\varepsilon}\definedas\{\Vert\nabla f\Vert^{2}<\varepsilon\}
\end{align}
for all $\varepsilon>0$, it readily follows that $\crit\subset W_{\varepsilon}$, $W_{\varepsilon}\subset M$ is open, and $W_{\varepsilon}\cap\partial M=\emptyset$ for suitably small $0<\varepsilon<\mu$. Moreover, it follows from purely topological arguments\footnote{See \cite[page 15]{Anabel} for details.} as well as from the already established fact that $\crit$ is compact that,  for suitably small $0<\varepsilon<\overline{\mu}$, $W_{\varepsilon}$ has finitely many connected components which are all bounded except precisely one which is a neighborhood of infinity. Next, we observe that $\partial W_{\varepsilon}=\{\Vert\nabla f\Vert^{2}=\varepsilon\}$ is closed and satisfies $\partial W_{\varepsilon}\cap\crit=\emptyset$ for all $\varepsilon>0$.

It is well-known that static vacuum systems are real analytic in suitable coordinate systems (see e.g.\ \cite{analytic}), hence both $f$ and $\Vert\nabla f\Vert^{2}$ are real analytic functions on $M$. From the Morse--Sard theorem \cite[Theorem 1]{Soucek1972}, we can then deduce that $f(\crit)$ is finite and that $\Vert\nabla f\Vert^{2}(\operatorname{Crit}\Vert\nabla f\Vert^{2})$ is discrete. Moreover, we know that $\crit\subseteq\operatorname{Crit}\Vert\nabla f\Vert^{2}$ because any critical point of $f$ is a local minimum of $\Vert\nabla f\Vert^{2}$. Hence $0\in \operatorname{Crit}\Vert\nabla f\Vert^{2}$ as we have assumed $\crit\neq\emptyset$ so that, by discreteness of $\operatorname{Crit}\Vert\nabla f\Vert^{2}$, there must be a threshold $\delta>0$ such that
\begin{align*}
\Vert\nabla f\Vert^{2}\geq\delta
\end{align*}
on $\operatorname{Crit}\Vert\nabla f\Vert^{2}\setminus\crit$. Then, for $0<\varepsilon<\delta$, the implicit function theorem applied to $\Vert\nabla f\Vert^{2}$ asserts that $\partial W_{\varepsilon}$ must be a smooth hypersurface with multiple but finitely many components.

Now let $U\subseteq M$ be an open domain with smooth boundary $\partial U$ such that $U\supset\crit$, $\overline{U}\cap\partial M=\emptyset$, and $\overline{U}\subset K$, with $K$ the compact subset of $M$ defined above. We can then extend $\mathcal{D}$ to $U$ (and thus to $M\cup\partial M$ in combination with the above) by setting $\mathcal{D}\definedas0$ on $\crit$ and obtain $dV$-measurability of $\mathcal{D}$ on $U$ (and thus on $M\cup\partial M$) from the fact that $\crit$ has $dV$-measure zero. To prove $dV$-integrability of $\mathcal{D}$ on $U$ (and thus on $M$), we introduce the abbreviation
\begin{align*}
Z\definedas\frac{F(f)}{f}\|\nabla f\|^{p-3}\,\nabla\|\nabla f\|^{2}+G(f)\|\nabla f\|^{p-1}\,\nabla f
\end{align*}
on $U\setminus\crit$ and extend $Z$ by $0$ on $\crit$ so that $Z$ is $dV$-measurable on $U$, again because $\crit$ has $dV$-measure zero. To study the $dV$-integrability of $\mathcal{D}$ on $U$, we want to apply the monotone convergence theorem. Let us consider a smooth cut-off function $\xi\colon [0,\infty)\to [0,1]$ satisfying
\begin{align*}
\begin{cases}
\xi(t)=0 &\text{if } t\leq \frac{1}{2},\\
\xi(t)=1&\text{if } t\geq \frac{3}{2},\\
0<\dot{\xi}(t)<2 &\text{if } \frac{1}{2}<t<\frac{3}{2}.
\end{cases}
\end{align*}
For $\varepsilon>0$, we define $\xi_{\varepsilon}\colon [0,\infty)\to [0,1]$ by setting $\xi_{\varepsilon}(t)\definedas \xi(\frac{t}{\varepsilon})$ and observe that
\begin{align*}
\begin{cases}
\xi_{\varepsilon}(t)=0 &\text{if } t\leq \frac{\varepsilon}{2},\\
\xi_{\varepsilon}(t)=1&\text{if } t\geq \frac{3\varepsilon}{2},\\
0<\dot{\xi}_{\varepsilon}(t)<\frac{2}{\varepsilon} &\text{if } \frac{\varepsilon}{2}<t<\frac{3\varepsilon}{2},\\
\xi_{\varepsilon_0}\leq \xi_{\varepsilon_1} &\text{if } 0<\varepsilon_{1}<\varepsilon_{0},\\
\xi_{\varepsilon}\to1&\text{as }\varepsilon\to0.
\end{cases}
\end{align*}
Next, we cut off $\Vert \nabla f\Vert^{2}$ near $\crit$ or in other words analyze the function $\Theta_{\varepsilon}\colon\R^n\setminus \Omega\to[0,1]$
\begin{align*}
\Theta_{\varepsilon}\definedas\xi_{\varepsilon}\circ \Vert\nabla f\Vert^2,
\end{align*}
with $\supp \Theta_{\varepsilon}\subseteq \overline{W}_{\frac{3\varepsilon}{2}}\setminus W_{\frac{\varepsilon}{2}}$. We consider a strictly decreasing sequence $\{\varepsilon_{k}\}_{k\in\N}$ with $\varepsilon_{k}>0$ satisfying
$\frac{3\varepsilon_{k}}{2}<\min\{\delta,\mu,\overline{\mu}\}$ for all $k\in\N$ and $\varepsilon_{k}\to0$ as $k\to\infty$. With this choice of $\{\varepsilon_{k}\}_{k\in\N}$, $\{\Theta_{\varepsilon_{k}}\}_{k\in\N}\subset L^{1}(U,dV)$ is an increasing sequence, and we have $\Theta_{\varepsilon_{k}}\to1$ pointwise on $U$ as $k\to\infty$. 

We compute
\begin{align*}
&\diver(\Theta_{\varepsilon_{k}}Z)\\
&=\underbrace{\left(\dot{\xi}_{\varepsilon_{k}}\circ\Vert\nabla f\Vert^2\right)\left[\frac{F(f)}{f}\Vert\nabla f\Vert^{p-3}\Vert\nabla\Vert\nabla f\Vert^2\Vert^2+G(f)\Vert\nabla f\Vert^{p-1} g(\nabla f,\nabla\Vert\nabla f\Vert^2)\right]}_{=:\, \mathcal{A}_{k}}+\underbrace{\vphantom{\left[\frac{\beta}{2}\right]}\Theta_{\varepsilon_{k}} \diver Z}_{=:\, \mathcal{B}_{k}}
\end{align*}
on $U$ for all $k\in\N$. For $\mathcal{B}_{k}$, we note that as $\Theta_{\varepsilon_{k}}$ vanishes near $\crit$, $\mathcal{B}_{k}\in L^{1}(U,dV)$ for all $k\in\N$. Hence, by the monotone convergence theorem and using that $\diver Z\geq0$ $dV$-almost everywhere on $U$ by \Cref{maintheoremp} because $dV(\crit)=0$, we find that
\begin{equation*}
\int_{U}\mathcal{B}_{k}\,dV=\int_{U}\Theta_{\varepsilon_{k}} \diver( Z)\, dV\to \int_{U} \diver( Z)\, dV\in\R^{+}_{0}\cup\{\infty\}
\end{equation*}
as $k\to\infty$. For $\mathcal{A}_{k}$, note that as $\xi_{\varepsilon_{k}}$ vanishes near $\crit$ we know that $\mathcal{A}_{k}\in L^{1}(U,dV)$ for all $k\in\N$. Now observe that $\supp\mathcal{A}_{k}\subseteq\overline{W}_{\frac{3\varepsilon_{k}}{2}}\setminus W_{\frac{\varepsilon_{k}}{2}}$ for all $k\in\N$. Also, all involved quantities are continuous on $U\cap \left(\overline{W}_{\frac{3\varepsilon_{k}}{2}}\setminus W_{\frac{\varepsilon_{k}}{2}}\right)$ which informs us that the map 
\begin{align*}
s\mapsto\int_{U \cap \partial W_s}\left(\dot{\xi}_{\varepsilon_{k}}\circ\Vert\nabla f\Vert^2\right)\frac{F(f)}{f}\Vert\nabla f\Vert^{p-3}\Vert\nabla\Vert\nabla f\Vert^2\Vert^{2}\,dS
\end{align*} is non-negative and $dV$-integrable on $[\tfrac{\varepsilon_{k}}{2},\tfrac{3\varepsilon_{k}}{2}]$ for all $k\in\N$ and all $p>p_n$. As $\Vert \nabla\Vert \nabla f\Vert^2\Vert^{2}$ {is smooth on the compact set $\overline{U}$}, the coarea formula applies (see e.g. \cite[Theorem 5]{Evans.2022}). Using the Cauchy--Schwarz inequality, the coarea formula, and the mean value theorem for integrals on intervals, we compute
\begin{align*}
&\int_{U}\left\vert\mathcal{A}_{k}\right\vert dV\\
&\leq \int_{U\cap \left(\overline{W}_{\frac{3\varepsilon_{k}}{2}}\setminus W_{\frac{\varepsilon_{k}}{2}}\right)}\left(\dot{\xi}_{\varepsilon_{k}}\circ\Vert\nabla f\Vert^2\right)\left[\frac{F(f)}{f}\Vert\nabla f\Vert^{p-3}\Vert\nabla\Vert\nabla f\Vert^2\Vert^2+2\vert G(f)\vert \Vert\nabla f\Vert^{p+1} \Vert\nabla^2f\Vert \right]\!dV\\
&=\int_{\frac{\varepsilon_{k}}{2}}^{\frac{3\varepsilon_{k}}{2}} \left(\,\int_{U \cap \partial W_s}\left(\dot{\xi}_{\varepsilon_{k}}\circ\Vert\nabla f\Vert^2\right)\frac{F(f)}{f}\Vert\nabla f\Vert^{p-3}\Vert\nabla\Vert\nabla f\Vert^2\Vert\, dS\!\right)\!ds\\
&\quad+2\int_{U\cap \left(\overline{W}_{\frac{3\varepsilon_{k}}{2}}\setminus W_{\frac{\varepsilon_{k}}{2}}\right)}\left(\dot{\xi}_{\varepsilon_{k}}\circ\Vert\nabla f\Vert^2\right)\vert G(f)\vert \Vert\nabla f\Vert^{p+1}\Vert\nabla^2f\Vert\,dV\\
&=\int_{\frac{\varepsilon_{k}}{2}}^{\frac{3\varepsilon_{k}}{2}} \left(\dot{\xi}_{\varepsilon_{k}}(s)s^{\frac{p-3}{2}}\int_{U \cap \partial W_s}\frac{F(f)\Vert\nabla\Vert\nabla f\Vert^2\Vert}{f}\,dS\!\right)\!ds\\
&\quad+2\int_{U\cap \left(\overline{W}_{\frac{3\varepsilon_{k}}{2}}\setminus W_{\frac{\varepsilon_{k}}{2}}\right)}\left(\dot{\xi}_{\varepsilon_{k}}\circ\Vert\nabla f\Vert^2\right)\vert G(f)\vert \Vert\nabla f\Vert^{p+1}\Vert\nabla^2f\Vert\,dV
\end{align*}
\begin{align*}
&\leq \frac{2}{\varepsilon_{k}}\max_{\overline{U}}(F(f))\int_{\frac{\varepsilon_{k}}{2}}^{\frac{3\varepsilon_{k}}{2}} s^{\frac{p-3}{2}}\left(\,\int_{U \cap \partial W_s}\!\frac{\Vert\nabla\Vert\nabla f\Vert^2\Vert}{f}\,dS\!\right)\! ds+\frac{4}{\varepsilon_{k}}\max_{\overline{U}}\left(|G(f)|\Vert\nabla^2f\Vert\right) \vert U\vert\left(\frac{3\varepsilon_k}{2}\right)^\frac{p+1}{2}\\
&= 2 r_{k}^{\frac{p-3}{2}}\max_{\overline{U}}(F(f))\int_{U \cap \partial W_{r_{k}}}\!\frac{\Vert\nabla\Vert\nabla f\Vert^2\Vert}{f}\,dS+ \underbrace{\left[\frac{3^\frac{p+1}{2}}{2^\frac{p-3}{2}}\max_{\overline{U}}\left(|G(f)|\Vert\nabla^2f\Vert\right) \vert U\vert\right]}_{=:\, D}\varepsilon_{k}^{\frac{p-1}{2}}\\
&\leq 2\max_{\overline{U}}(F(f)) \underbrace{r_{k}^{\frac{p-3}{2}}\int_{\overline{U}\cap\partial W_{r_{k}}}\!\frac{\Vert\nabla\Vert\nabla f\Vert^2\Vert}{f}\,dS}_{=:\,\mathcal{C}_{k}}+\, D\underbrace{\vphantom{\int_{\overline{U}\cap\partial W_{r_{k}}}\Vert\nabla\Vert\nabla f\Vert^2\Vert\,dS}\varepsilon_{k}^{\frac{p-1}{2}}}_{=:\,\mathcal{D}_{k}}
\end{align*}
for some $r_{k}\in(\frac{\varepsilon_{k}}{2},\frac{3\varepsilon_{k}}{2})$ and all $k\in\N$. Here, $\vert U\vert$ denotes the (finite) $dV$-volume of $U$. Clearly, $\mathcal{D}_{k}\to0$ as $k\to\infty$ as $p\geq p_n>1$ {recalling that $\varepsilon_k\to0$ as $k\to\infty$}. We will now show that $\mathcal{C}_{k}\to0$ as $k\to\infty$, asserting by the above that 
\begin{equation}\label{eq:divconvergence}
\int_{U}\diver\left(\Theta_{\varepsilon_{k}}Z\right)dV\to\int_{U}\diver Z\,dV\in\R^{+}\cup\{\infty\}
\end{equation}
as $k\to\infty$. For reasons that will become clear later, we first assume $p>p_{n}$ and will handle the case $p=p_{n}$ separately towards the end of the proof. To analyze $\mathcal{C}_{k}$, we set 
\begin{equation*}
\rho_{U}\definedas\min\left\{\min_{\partial U}{\Vert \nabla f\Vert^{2}},\mu,\overline{\mu},\delta\right\}>0
\end{equation*}
and choose $k_{0}=k_{0}(U,g,f)\in\N$ such that $\frac{3\varepsilon_{k}{2}}<\rho_{U}$ for all $k\geq k_{0}$. This in particular asserts that $r_{k}<\rho_{U}$ for all $k\geq k_{0}$. By definition of $\rho_{U}$, we find that $\partial U\cap \overline{W_{r}}=\emptyset$ and thus $\partial(U \cap W_{r})=U \cap \partial W_{r}$ for all $0<r<\rho_{U}$.  With this in mind, let us anaylze the auxiliary function $\zeta\colon(0,\rho_{U})\to\R$ defined by
\begin{equation*}
\zeta(r)\definedas\int_{U \cap\, {\partial W_{r}}}\frac{\Vert \nabla\Vert\nabla f\Vert^2\Vert}{f}\,dS.
\end{equation*}
Clearly, we have $\zeta\in L^{\infty}(0,\rho_{U})\subset L^{1}(0,\rho_{U})$ as $\Vert \nabla f\Vert^{2}$ is continuous on $M$ and $\overline{U}\subset M$ is compact. Recall that we have asserted above that $\partial(U \cap W_{r})$ is a smooth hypersurface with finitely many components. Thus, applying the divergence theorem, the static vacuum equation \eqref{vacuum equation1}, and the Bochner formula \eqref{eq:Bochner}, we get
\begin{align*}
\zeta(r)&=\int_{\partial(U \cap W_{r})}g\left(\frac{\nabla\Vert \nabla f\Vert^2}{f},\frac{\nabla\Vert \nabla f\Vert^2}{\Vert\nabla\Vert \nabla f\Vert^2\Vert}\right)dS=\int_{U \cap W_r}\diver\left(\frac{ \nabla\Vert \nabla f\Vert^2}{f}\right)dV\\
&=\int_{U \cap W_r}\left(\frac{\Delta \Vert\nabla f\Vert^2}{f}-\frac{g\left(\nabla\Vert \nabla f\Vert^2,\nabla f\right)}{f^{2}}\right) dV=\int_{U \cap W_r}\left(\frac{\Delta \Vert\nabla f\Vert^2-2\Ric(\nabla f,\nabla f)}{f}\right) dV\\
&=2\int_{U \cap W_r}\frac{\Vert \nabla^2f\Vert^2}{f}\, dV
\end{align*} 
for all $0<r<\rho_{U}$. Applying the coarea formula, we find
\begin{equation*}
\zeta(\overline{r})-\zeta(r)=2\int_{r}^{\overline{r}} \left(\,\int_{U\cap \partial W_{s}} \frac{\Vert\nabla^2f\Vert^2}{f\Vert\nabla\Vert \nabla f\Vert^2\Vert} \,dS\! \right)\!ds
\end{equation*}
for all $0<r\leq \overline{r}<\rho_{U}$ because $\Vert \nabla\Vert \nabla f\Vert^{2}\Vert$ is bounded from below by a positive constant on $\overline{U} \cap (\overline{W_{\overline{r}}}\setminus W_{r})$ and thus $\frac{\Vert \nabla^{2}f\Vert^{2}}{f\Vert \nabla\Vert \nabla f\Vert^{2}\Vert}\in L^{\infty}(U \cap (\overline{W_{\overline{r}}}\setminus W_{r}))\subset L^{1}(U \cap (\overline{W_{\overline{r}}}\setminus W_{r}))$. Similarly, appealing in addition to the fundamental theorem of calculus in the Sobolev space $W^{1,1}(\tau,\rho_{U})$, we have $\zeta\in W^{1,1}(\tau,\rho_{U})$ for any fixed $0<\tau<\rho_{U}$ with weak derivative  
\begin{align*}
\zeta'(r)=2\int_{U \cap \partial W_{r}}\frac{\Vert\nabla^2f\Vert^2}{f\Vert\nabla\Vert \nabla f\Vert^2\Vert}\,dS
\end{align*}
for almost all $\tau<r<\rho_{U}$. The $1$-dimensional Sobolev embedding theorem then gives us that $\zeta$ is continuous on $(\tau,\rho_{U})$ for all $0<\tau<\rho_{U}$ and hence continuous on $(0,\rho_{U})$. The refined Kato inequality (see e.g.\ \cite{SchoenA}) implies that
\begin{align}\label{kato}
\Vert\nabla ^2 f\Vert^2\geq\frac{n}{n-1}\Vert\nabla\Vert\nabla f\Vert\Vert^2
\end{align}
 on $U\setminus\crit$. Thus
\begin{align*}
\zeta'(r)\geq \frac{2n}{n-1} \int_{U \cap \partial W_{r}}\frac{\Vert\nabla\Vert \nabla f\Vert\Vert^2}{f\Vert\nabla\Vert \nabla f\Vert^2\Vert}\,dS=\frac{n}{2(n-1)}\frac{\zeta(r)}{r}
\end{align*}
for almost all $\tau<r<\rho_{U}$, using that $\Vert \nabla\Vert \nabla f\Vert\Vert=\frac{\Vert \nabla \Vert \nabla f\Vert^{2}\Vert}{2\Vert \nabla f\Vert }$  and $\left(U \cap \partial W_{r}\right)\cap\crit=\emptyset$. As $0<\tau<\rho_{U}$ is arbitrary, this is equivalent to
\begin{equation*}
(\ln\circ\,\zeta)'(r)\geq \frac{n}{2(n-1)}\ln'(r)
\end{equation*}
for almost all $0<r<\rho_{U}$. Picking a fixed $0<R<\rho_{U}$ for which this inequality holds, this integrates to 
\begin{align*}
\zeta(r)\leq \frac{\zeta(R)}{R^{\frac{n}{2(n-1)}}}\,r^{\frac{n}{2(n-1)}}
\end{align*}
for all $0<r<R$ by continuity of $\zeta$. Hence
\begin{equation}\label{inequ:zeta}
0<r^{\frac{p-3}{2}}\zeta(r)\leq \frac{\zeta(R)}{R^{\frac{n}{2(n-1)}}}\,r^{\frac{p-p_{n}}{2}}
\end{equation}
holds for all $0<r<R$. For $p>p_{n}$, the exponent of $r$ on the right hand side of \eqref{inequ:zeta} is strictly positive so that $\mathcal{C}_{k}=r_{k}^{\frac{p-3}{2}}\zeta(r_{k})\to0$ as $k\to\infty$. This proves \eqref{eq:divconvergence} for $p>p_{n}$. 

Consider now the surface integral term
\begin{equation*}
\int_{\partial U}g(\Theta_{\varepsilon_{k}}Z,\eta)\,dS,
\end{equation*}
where $\eta$ denotes the unit normal to $\partial U$ pointing to the outside of $U$. Recalling that $\partial U\cap\crit=\emptyset$, we know that $Z$ is continuous on $\partial U$ and hence by compactness of $\partial U$ and Lebesgue's dominated convergence theorem, we have
\begin{equation*}
\int_{\partial U}g(\Theta_{\varepsilon_{k}}Z,\eta)\,dS\to \int_{\partial U}g(Z,\eta)\,dS\in\R
\end{equation*}
as $k\to\infty$. Together with \eqref{eq:divconvergence} and applying the divergence theorem to $\Theta_{\varepsilon_{k}}Z$ on $U$, this establishes $dV$-integrability of $\mathcal{D}=\diver Z$ on $U$ and hence on $M$ when $p>p_{n}$. Moreover, denoting $Z$ by $Z_{p}$ to be able to distinguish the above results for different $p>p_{n}$, we have asserted that
\begin{align}\label{eq:divTHM}
\int_{U}\diver Z_{p}\,dV=\int_{\partial U} g(Z_{p},\eta)\,dS
\end{align}
for all $p>p_{n}$. To conclude that $\mathcal{D}=\diver Z$ is $dV$-integrable for $p=p_{n}$, let us consider a strictly decreasing sequence $\{p_{l}\}_{l\in\N}$ with $p_{l}>p_{n}$ and $p_l\to p_{n}$ as $l\to\infty$. Using again that $\partial U\cap\crit=\emptyset$, we find that $Z_{p_{l}}\to Z_{p_{n}}$ on $\partial U$ as $l\to\infty$. As $\{Z_{p_{l}}\}_{l\in\N}$ is uniformly bounded on the compact set $\partial U$ by continuity, Lebesgue's dominated convergence theorem informs us that
\begin{equation*}
\int_{\partial U} g(Z_{p_{l}},\eta)\,dS\to\int_{\partial U} g(Z_{p_{n}},\eta)\,dS
\end{equation*}
as $l\to\infty$. Now recall that $dV(\crit)=0$ and note that this gives $Z_{p_{l}}\to Z_{p_{n}}$ pointwise $dV$-almost everywhere as $l\to\infty$. Splitting $U$ into $U\cap W_{1}$ and $U\setminus W_{1}$, Lebesgue's dominated convergence theorem tells us that 
\begin{align*}
\int_{U\setminus W_{1}}\diver Z_{p_{l}}\,dV \to \int_{U\setminus W_{1}}\diver Z_{p_{n}}\,dV\in\R^{+}_{0}
\end{align*}
as $l\to\infty$. On $U\cap W_{1}$, we rewrite \eqref{mainformulap} as
\begin{align*}
\diver Z_{p_{l}}&=\underbrace{\Vert\nabla f\Vert^{p_{l}-5}\,\frac{(n-2)^{2}F(f)f}{(n-1)^{2}}\|T\|^{2}}_{=:\mathcal{E}_{l}}\\
&\quad+(p_{l}-p_{n})\underbrace{\Vert\nabla f\Vert^{p_{l}-5}\,\frac{F(f)}{2f}\left\|\nabla\|\nabla f\|^{2}+\frac{4(n-1)}{(n-2)}\frac{f\|\nabla f\|^{2}\,\nabla f}{1-f^{2}}\right\|^{2}}_{=:\mathcal{F}_{l}}
\end{align*}
and note that $\{\mathcal{E}_{l}\}_{l\in\N}$, $\{\mathcal{F}_{l}\}_{l\in\N}$ are non-negative sequences of $dV$-measurable functions on $U\cap W_{1}$ by \Cref{maintheoremp} and because $p_{l}> p_{n}$ and $F(f)\geq0$. Moreover, both $\{\mathcal{E}_{l}\}_{l\in\N},\{\mathcal{F}_{l}\}_{l\in\N}$ are monotonically increasing sequences on $U\cap W_{1}$ as
\begin{align*}
\frac{\partial \Vert \nabla f\Vert^{p-5}}{\partial p}=\ln(\Vert \nabla f\Vert)\Vert \nabla f\Vert^{p-5}<0
\end{align*}
holds for all $p\in\R$ $dV$-almost everywhere on $U\cap W_{1}$ {as $\Vert\nabla f\Vert<1$ on $W_1$}. By the monotone convergence theorem, we obtain\footnote{We would like to remark that we cannot conclude that the term involving $(p_{l}-p_{n})$ vanishes in the limit $l\to\infty$ as $\lim_{l\to\infty}\int_{U\cap W_{1}}\mathcal{F}_{l}\,dV$ may be infinite. This causes no issues as $p_{l}>p_{n}$ and $\mathcal{E}_{l},\mathcal{F}_{l}\geq0$.}
\begin{align*}
\int_{U\cap W_{1}}\diver Z_{p_{l}}\,dV=\int_{U\cap W_{1}}\mathcal{E}_{l}\,dV+(p_{l}-p_{n})\int_{U\cap W_{1}}\mathcal{F}_{l}\,dV\to \int_{U\cap W_{1}}\diver Z_{p_{n}}\,dV\in\R^{+}_{0}\cup\{\infty\}
\end{align*}
as $l\to\infty$. By \eqref{eq:divTHM}, we can thus deduce that $\mathcal{D}=\diver Z$ is $dV$-integrable on $U$ and thus on $M$ also for $p=p_{n}$ as claimed.

We now would like to apply the divergence theorem to the vector field $Z$ on $M$ for all $p\geq p_{n}$. As we have asserted the $dV$-integrability of $\mathcal{D}=\diver Z$ and know that $Z$ is smooth near $\partial M$ and in a neighborhood of infinity and because $(M,g)$ is geodesically complete up to $\partial M$ by \Cref{rem:complete}, the divergence theorem applies. It only remains to study the ``boundary integral at infinity'' and to evaluate the surface integral at the inner boundary. To this end, let $r>0$ be such that $B_r \definedas\left\{ x \in \R^n \,:\, |x|<r \right\}\supset \overline{B}$ and thus $x^{-1}(\R^{n}\setminus B_{r})\subseteq M$, where $x$ denotes the asymptotically flat chart and $\overline{B}$ the complement of the image of $x$ in $\R^{n}$. From the divergence theorem, our choice of unit normal $\nu$ pointing towards infinity, and \eqref{eq002}, we obtain
		\begin{align*}
			&\int_{M\setminus x^{-1}(\R^{n}\setminus B_{r})}\mathcal{D}\,dV\\
			&\quad=\int_{x^{-1} (\partial B_r)}\left(\frac{F(f)}{f}\Vert\nabla f\Vert^{p-3}\langle\nabla\|\nabla f\|^{2},\nu\rangle+G(f)\|\nabla f\|^{p-1}\langle\nabla f,\,\nu\rangle\right) dS\\
			&\quad\quad-\int_{\partial M}\left(\frac{F(f)}{f}\Vert\nabla f\Vert^{p-3}\langle\nabla\|\nabla f\|^{2},\nu\rangle+G(f)\|\nabla f\|^{p-1}\langle\nabla f,\,\nu\rangle\right)dS.
		\end{align*} 
		Exploiting \Cref{lem:asymptotics} and the above asymptotics for $F(f)$ and $G(f)$, we find
		\begin{align*}
			&\int_{x^{-1} (\partial B_r)}\left(\frac{F(f)}{f}\Vert\nabla f\Vert^{p-3}\langle\nabla\|\nabla f\|^{2},\nu\rangle+G(f)\|\nabla f\|^{p-1}\langle\nabla f,\,\nu\rangle\right) dS\\
			&=\int_{x^{-1} (\partial B_r)}\left(\frac{F(f_{m})}{f_{m}}\Vert\nabla_{m} f_{m}\Vert_{m}^{p-3}\langle\nabla_{m}\|\nabla_{m} f\|_{m}^{2},\nu_{m}\rangle_{m}+G(f_{m})\|\nabla_{m} f_{m}\|_{m}^{p-1}\langle\nabla_{m} f_{m},\,\nu_{m}\rangle_{m}\right) dS_{\delta}\\
			&\quad+o(1)\\
			&=-\mathcal{F}^{\,c,d}_{\!p}(m)+o(1)
		\end{align*}
		as $r\to\infty$, where $\langle\cdot,\cdot\rangle_{m}=g_{m}$ and $\nu_{m}$ denotes the unit normal to $x^{-1}(\partial B_{r})$ with respect to $g_{m}$ and pointing to infinity. For the inner boundary integral, we recall from \Cref{rmk:lapse-asymp} that $\nu=\frac{\nabla f}{\Vert\nabla f\Vert}$, $\kappa=\Vert\nabla f\Vert$ if $f_{0}\in[0,1)$ and $\nu=-\frac{\nabla f}{\Vert\nabla f\Vert}$, $\kappa=-\Vert\nabla f\Vert$ if $f_{0}\in(1,\infty)$. Exploiting that $f=f_{0}$ and $\kappa$ are constant on $\partial M$ by assumption, we compute
		\begin{align*}
			&\int_{\partial M}\left(\frac{F(f)}{f}\Vert\nabla f\Vert^{p-3}\langle\nabla\|\nabla f\|^{2},\nu\rangle+G(f)\|\nabla f\|^{p-1}\langle\nabla f,\,\nu\rangle\right)dS\\
			&=\pm\frac{F_{0}}{f_{0}}\vert\kappa\vert^{p-4}\int_{\partial M}\langle\nabla\Vert\nabla f\Vert^{2},\nabla f\rangle dS\pm G_{0}\vert\kappa\vert^{p}\vert\partial M\vert,
		\end{align*}
		with $\pm=+$ if $f_{0}\in[0,1)$ and $\pm=-$ if $f_{0}\in(1,\infty)$, respectively. Using \eqref{vacuum equation1} and \eqref{scal0} as well as the Gau{\ss} equation, we compute
		\begin{align*}
			&\langle\nabla\Vert\nabla f\Vert^{2},\nabla f\rangle=2\nabla^{2}f(\nabla f,\nabla f)=2f_{0}\Ric(\nabla f,\nabla f)\\
			&\quad=2f_{0}\kappa^{2}\Ric(\nu,\nu)=-f_{0}\kappa^{2}\left(\scal_{\partial M}-\tfrac{n-2}{n-1} \HH^2+\Vert\mathring{h}\Vert^{2}\right).
		\end{align*}
		Combining this with the above, we find
		\begin{align*}
			&\int_{\partial M}\left(\frac{F(f)}{f}\Vert\nabla f\Vert^{p-3}\langle\nabla\|\nabla f\|^{2},\nu\rangle+G(f)\|\nabla f\|^{p-1}\langle\nabla f,\,\nu\rangle\right)dS\\
			&=\mp F_{0}\vert\kappa\vert^{p-2}\int_{\partial M}\left(\scal_{\partial M}-\tfrac{n-2}{n-1} \HH^2+\Vert\mathring{h}\Vert^{2}\right) dS\pm G_{0}\vert\kappa\vert^{p}\vert\partial M\vert
		\end{align*}
		and thus
		\begin{align*}
			{0\leq}\int_{M} \mathcal{D}\,dV&=-\mathcal{F}^{\,c,d}_{\!p}(m)\pm F_{0}\vert\kappa\vert^{p-2}\int_{\partial M}\left(\scal_{\partial M}-\tfrac{n-2}{n-1} \HH^2+\Vert\mathring{h}\Vert^{2}\right)dS\mp G_{0}\vert\kappa\vert^{p}\,\vert\partial M\vert.
		\end{align*}
		Consequently, for $c,d\in\R$ satisfying $c+d\geq0$, $cf_{0}^{2}+d\geq0$, and for all $p\geq p_{n}$, we find~\eqref{mainformulap2} and \eqref{mainformulap2b}. Equality holds in~\eqref{mainformulap2} or in \eqref{mainformulap2b} if and only if equality holds in \eqref{mainformulap0} and thus $\mathcal{D}=0$ on $M$. Hence if $c+d\geq0$, $cf_{0}^{2}+d\geq0$ (but not $c=d=0$), vanishing of both sides in \eqref{mainformulap} gives $T=0$ {and, if $p>p_{n}$, also \eqref{mainformulap3}. By \Cref{thm:rigidity}, this implies the equality assertion of \Cref{maintheoremp2}}. 
	\end{proof}
	
	Let us now discuss the geometric implications of \eqref{mainformulap2} and \eqref{mainformulap2b} or in other words prove \Cref{maintheoremp3} and, in passing, its equivalence to \Cref{maintheoremp2}, {see also \Cref{Coro:p>1}}. 
	\begin{proof}[Proof of \Cref{maintheoremp3}]
		We begin by choosing $f_{0}\in[0,1)$, recalling that $\kappa>0$ and $m>0$ in this case. Choosing the admissible constants $c=1$, $d=-f_{0}^{2}$ in \eqref{mainformulap2} {and any $p\geq p_{n}$}, we find from \Cref{Smarr} that $\kappa=(n-2)\frac{\vert\mathbb{S}^{n-1}\vert}{\vert\partial M\vert} \,m$ and thus
		\begin{align}\label{mgeq}
			m \geq \frac{(1-f_0^2)\left(s_{\partial M}\right)^{n-2}}{2}{>0},
		\end{align}
		{asserting the right hand side inequality in \eqref{maininequality2h}.} Choosing instead the admissible constants $c=-1$, $d=1$ {and any $p\geq p_{n}$}, \eqref{mainformulap2} reduces to
		\begin{align}\label{m^2leq}
			(1-f_{0}^{2})\vert\partial M\vert\int_{\partial M} \left(\scal_{\partial M} -  \tfrac{n-2}{n-1} \HH^2+\Vert\mathring{h}\Vert^{2}\right) dS\geq 4(n-1)(n-2)\vert\mathbb{S}^{n-1}\vert^{2}\, m^{2}.
		\end{align}
		{This asserts the left-hand side inequality in \eqref{maininequality2h}} via an algebraic re-arrangement. {Via \Cref{thm:rigidity}, we have hence proved \Cref{maintheoremp3} for $f_{0}\in[0,1)$.}
 On the other hand, as \eqref{mainformulap2} is linear in $c,d$ and the constraints $c+d\geq0$, $cf_{0}^{2}+d\geq0$ are linear as well, the combination of {\eqref{mgeq}, \eqref{m^2leq} asserts \eqref{mainformulap2} for any $p>1$ via the Smarr formula \eqref{eq:Smarr}}. 
				
		Next, let us consider $f_{0}\in(1,\infty)$, recalling that $\kappa,m<0$ in this case. Choosing the admissible constants $c=-1$, $d=f_{0}^{2}$ in \eqref{mainformulap2b} {and any $p\geq p_{n}$}, we again find from \Cref{Smarr} that $\kappa=(n-2)\frac{\vert\mathbb{S}^{n-1}\vert}{\vert\partial M\vert}\,m$ and thus
		\begin{align*}
			m \leq \frac{(1-f_0^2)\left(s_{\partial M}\right)^{n-2}}{2}{<0},
		\end{align*}
		{asserting the right hand side inequality in \eqref{maininequality2bh}.} Choosing instead the admissible constants $c=1$, $d=-1$ {and any $p\geq p_{n}$}, \eqref{mainformulap2} reduces to
		\begin{align*}
			(f_{0}^{2}-1)\vert\partial M\vert\int_{\partial M} \left(\scal_{\partial M} -  \tfrac{n-2}{n-1} \HH^2+\Vert\mathring{h}\Vert^{2}\right) dS\leq -4(n-1)(n-2)\vert\mathbb{S}^{n-1}\vert^{2}\, m^{2}.
		\end{align*}
		{This asserts the left-hand side inequality in \eqref{maininequality2bh}} via an algebraic re-arrangement.  {Via \Cref{thm:rigidity}, we have hence proved \Cref{maintheoremp3} for $f_{0}\in(1,\infty)$}. Again, as \eqref{mainformulap2b} is linear in $c,d$ and the constraints $c+d\geq0$, $cf_{0}^{2}+d\geq0$ are linear as well, the combination of the two above inequalities asserts \eqref{mainformulap2b} {for any $p>1$ via the Smarr formula \eqref{eq:Smarr}.}
	\end{proof}
	
\begin{corollary}[{\Cref{maintheoremp2}} holds for $p>1$]\label{Coro:p>1}
{It follows from the proof of \Cref{maintheoremp3} that \Cref{maintheoremp2} remains valid for $p>1$, with $F_0$ and $G_0$ formally extended to $p\in(1,p_n)$.}
\end{corollary}
	
	We now proceed to proving  \Cref{teoclassifica1} and  \Cref{teoclassifica2}, where we will use \Cref{thm:rigidity} {which was proven} in \Cref{sec:rigidity}.
	
	\begin{proof}[Proof of \Cref{teoclassifica1}]
		To prove \Cref{teoclassifica1}, we consider the implications of \Cref{maintheoremp3} in the setting of \Cref{teoclassifica1}, i.e., if $f_{0}=0$ and $\partial M$ is a connected static horizon. Then we know that $\HH=0$, $\mathring{h}=0$ on $\partial M$ by \Cref{rem:geodesic} and hence \eqref{maininequality2h} reduces to \eqref{maininequality}. Moreover, \eqref{maininequality2h} implies \eqref{ineqA} upon dropping the middle term and squaring. If, in addition, assumption \eqref{CondR} holds then we have equality in \eqref{ineqA} and thus in both inequalities in \eqref{maininequality2h}. {By the equality case assertion in \Cref{maintheoremp3}, this proves \Cref{teoclassifica1}}.
	\end{proof}
	
	\begin{proof}[Proof of \Cref{teoclassifica2}]
		To see that \Cref{teoclassifica2} holds, we consider the implications of \Cref{maintheoremp3} in the setting of \Cref{teoclassifica2}, i.e., if $f_{0}\in(0,1)\cup(1,\infty)$ and $\partial M$ is a connected time-slice of a photon surface and thus in particular has constant scalar curvature $\scal_{\partial M}$, constant mean curvature $\HH$, is totally umbilic ($\mathring{h}=0$) and obeys the photon surface constraint~\eqref{eq:photo}. When $f_{0}\in(0,1)$, \eqref{maininequality2h} gives \eqref{maininequality2}. Moreover, dropping the middle term in \eqref{maininequality2h} and squaring it gives \eqref{implicationphoto}. Rewriting \eqref{implicationphoto} via the photon surface constraint \eqref{eq:photo} gives
		\begin{align*}
			\frac{2\kappa\HH}{f_{0}}&\geq \frac{(n-1)(n-2)(1-f_{0}^{2})}{\left(s_{\partial M}\right)^{2}}.
		\end{align*}
		Assuming in addition \eqref{AgoMazzCond} and rewriting it via the photon surface constraint \eqref{eq:photo}  gives
		\begin{align*}
			\frac{2\kappa \HH}{f_{0}}+\frac{n-2}{n-1}\HH^{2}&\leq\frac{(n-1)(n-2)}{\left(s_{\partial M}\right)^{2}}.
		\end{align*}
		Taken together, this gives
		\begin{align*}
			\frac{2\kappa\HH}{f_{0}}+\frac{n-2}{n-1}\HH^{2}&\leq\frac{(n-1)(n-2)}{\left(s_{\partial M}\right)^{2}}\leq\frac{2\kappa\HH f_{0}}{1-f_{0}^{2}}
		\end{align*}
		Recalling that $\HH>0$ from the above or by \Cref{prop:signH} implies
		\begin{align*}
			1-f_0^2 & \leq \frac{2(n-1)\kappa f_{0}}{(n-2)\HH}.
		\end{align*}
		On the other hand, the squared left-hand side inequality in \eqref{maininequality2} together with the Smarr formula~\eqref{eq:Smarr} leads to
		\begin{align*}
			1-f_0^2 & \geq \frac{2(n-1)\kappa f_{0}}{(n-2)\HH}.
		\end{align*}
		Thus, equality holds in all the above inequalities and hence in \eqref{maininequality2}, too. {By the equality assertion in \Cref{maintheoremp3}, this proves \Cref{teoclassifica2}} when $f_{0}\in(0,1)$. For $f_{0}\in(1,\infty)$, the argument is the same with reversed signs.
	\end{proof}

\section{Discussion and monotone functions}\label{sec:discussion}
\subsection{Monotone functions along level sets}\label{subsec:monotone}
In \Cref{maintheoremp} and the proof of \Cref{maintheoremp2}, we have seen that the divergence of the vector field 
	\begin{align}\label{def:Z}
	Z\definedas \frac{F(f)}{f}\Vert\nabla f\Vert^{p-3}\,\nabla\Vert\nabla f\Vert^{2}+G(f)\Vert\nabla f\Vert^{p-1}\,\nabla f,
	\end{align}
	$\mathcal{D}=\diver Z$, is $dV$-integrable and non-negative $dV$-almost everywhere for all $n\geq3$, $p\geq p_{n}$, $c,d\in\R$ with $c+d\geq0$, $cf_{0}^{2}+d\geq0$, and $F$ and $G$ as defined in \eqref{eq:F}, \eqref{eq:G}. We exploited this to prove the parametric geometric inequalities in \Cref{maintheoremp2} and, equivalently, the geometric inequalities in \Cref{maintheoremp3}, by applying a suitably adapted divergence theorem to $Z$ on $M$ and evaluating the corresponding surface integrals at $\partial M$ with $f=f_{0}$ on $\partial M$, $f_{0}\in[0,1)\cup(1,\infty)$, and at infinity. Of course, one can also apply the adapted divergence theorem to $Z$ on suitable open domains $N\subset M$ and exploit the non-negativity of $\diver Z$ to obtain estimates between the different components of $\partial N$. In view of the fact that we are using an approach based on a potential, and in order to compare our technique of proof to the monotone function approach by Agostiniani and Mazzieri~\cite{agostiniani}, it will be most interesting to study such $N$ for which $\partial N$ consists of level sets of the lapse function $f$. Our arguments are inspired by \cite[Proposition 4.2]{Anabel}, see also \cite{Agostiniani.2020}.
	
Given a static vacuum system $(M^n,g,f)$, $n\geq3$, with boundary $\partial M$ such that $f=f_{0}$ on $\partial M$ for some $f_{0}\in[0,1)$ or $f_0\in(1,\infty)$, and given $p\geq p_{n}$, $c,d\in\R$ satisfying $c+d\geq0$, $cf_{0}^{2}+d\geq0$, and $F$ and $G$ as defined in \eqref{eq:F}, \eqref{eq:G}, we define the functions $\mathcal{H}_{p}^{c,d}\colon\left([f_{0},1)\cup(1,f_0]\right)\setminus f(\crit)\to\R$ by
	\begin{align}\label{def:curlyH}
		\mathcal{H}_{p}^{c,d}(f)&\definedas \int_{\Sigma_{f}}\langle Z,\nu\rangle\, dS =\pm\int_{\Sigma_{f}}\left[\frac{F(f)}{f}\Vert\nabla f\Vert^{p-4}\langle\nabla\Vert\nabla f\Vert^{2},\nabla f\rangle+G(f)\Vert\nabla f\Vert^{p}\right]dS,
\end{align}
where $\Sigma_{f}$ denotes the $f$-level set of the lapse function $f$ and the sign $\pm$ is $+$ for $f_0\in[0,1)$ and $-$ for $f_0\in(1,\infty)$. $\mathcal{H}_{p}^{c,d}(f)$ is clearly well-defined as we restricted its definition to regular values of $f$ and as we have already asserted that the integral under consideration is well-defined for $f=f_0=0$. {Using the decomposition of $\Delta$ along a level set of $f$ and the static vacuum equation \eqref{vacuum equation2}, one obtains that the mean curvature $\HH$ of any regular level set $\Sigma_f$ is given by
\begin{align}\label{eq:Hlevel}
\HH&=\mp\frac{\nabla^2f(\nabla f,\nabla f)}{\Vert\nabla f\Vert^2},
\end{align}
where the sign $\mp$ is $-$ for $f\in[f_0,1)$ and $+$ if $f\in(1,f_0]$. This} shows that 
	\begin{align}\label{eq:identitycurlyH}
		\mathcal{H}_{p}^{c,d}(f)=\int_{\Sigma_{f}}\Vert\nabla f\Vert^{p-1}\left[-\frac{2F(f)\HH}{f}\pm G(f)\Vert\nabla f\Vert\right]dS
	\end{align}
	holds for all regular values $f\in[f_{0},1)\cup{(1,f_0]}$, understood at $f_0$ as the limit $f\to f_0$ in case $F(f_0)=0$ or $f_0=0$. Recalling from the proof of \Cref{maintheoremp2} that $f(\crit)$ is finite and $f_{0}$ is a regular value of $f$, we will now show that $\mathcal{H}^{c,d}_{p}$ (in both of its representations \eqref{def:curlyH}, \eqref{eq:identitycurlyH}) can be continuously extended to the at most finitely many singular values of~$f$ and is monotone. To see this, let $f_{*}\in(f_{0},1)\cup{(1,f_0]}$ be a critical value of $f$ which then necessarily has an open neighborhood $(f_{*}-2\varepsilon,f_{*}+2\varepsilon)$ for some suitably small $\varepsilon>0$ such that $(f_{*}-2\varepsilon,f_{*}+2\varepsilon)\setminus\{f_{0}\}$ contains only regular points. We set
\begin{equation}
\Psi \definedas \pm\mathcal{H}^{c,d}_{p}\vert_{(f_{*}-2\varepsilon,f_{*}+2\varepsilon)\setminus\{f_{0}\}}
\end{equation}
which is clearly well-defined. Applying the adapted divergence theorem from the proof of \Cref{maintheoremp2} to $Z$ on the domains $(\eta,f_{*}+\varepsilon)$ and $(f_{*}-\varepsilon,\eta)$ for a fixed $\eta\in(f_{*}-\varepsilon,f_{*}+\varepsilon)\setminus\{f_{0}\}$, we learn that 
\begin{equation}\label{eq:comparison}
\Psi (\eta)=\Psi (f_{*}+\varepsilon)-\int_{\{\eta<f<f_{*}+\varepsilon\}}\diver Z\,dV=\Psi (f_{*}-\varepsilon)+\int_{\{\eta>f>f_{*}-\varepsilon\}}\diver Z\,dV.
\end{equation}
As $\diver Z\geq0$ holds $dV$-almost everywhere on $M$ by \Cref{maintheoremp}, $\Psi $ is monotonically increasing on $(f_{*}-\varepsilon,f_{*}+\varepsilon)\setminus\{f_{*}\}$ and we have
\begin{equation}\label{eq:cts}
\Psi (f_{*}+\varepsilon)\geq\Psi (\eta)\geq\Psi (f_{*}-\varepsilon).
\end{equation}
for all $\eta\in(f_{*}-\varepsilon,f_{*}+\varepsilon)\setminus\{f_{0}\}$. This establishes that $\limsup_{\eta\to f_{*}}\Psi (\eta)$ and $\liminf_{\eta\to f_{*}}\Psi (\eta)$ are finite. Moreover, we learn from \eqref{eq:comparison} that
\begin{equation}
\Psi (f_{*}+\varepsilon)-\Psi (f_{*}-\varepsilon)=\int_{\{f_{*}-\varepsilon<f<f_{*}+\varepsilon\}}\diver Z\,dV
\end{equation}
holds for all suitably small $\varepsilon>0$. Thus, recalling from the proof of \Cref{maintheoremp2} that $dV$ {and the $n$-dimensional Hausdorff measure $\mathcal{H}^{n}$} are absolutely continuous with respect to each other with bounded densities and that $\diver Z$ is $dV$-integrable, it follows that $\limsup_{\eta\to f_{*}}\Psi (\eta)=\liminf_{\eta\to f_{*}}\Psi (\eta)$ so that $\Psi $ can continuously extended to $f_{*}$. Extending $\mathcal{H}^{c,d}_{p}$ to $\crit$ continuously in this way, we obtain from \eqref{eq:cts} that $\mathcal{H}^{c,d}_{p}$ is well-defined, continuous, and monotonically increasing on $[f_{0},1)$ for all $f_{0}\in[0,1)$ / monotonically decreasing on $(1,f_0]$ for all $f_0\in(1,\infty)$. Moreover, by \Cref{maintheoremp2} and its proof, we know 
	\begin{align}
		\lim_{{f\to1}}\mathcal{H}_{p}^{c,d}(f)=-\mathcal{F}_{p}^{c,d}(m)
	\end{align}
	with $\mathcal{F}_{p}^{c,d}$ as in \eqref{eq:Fmcd} and $m$ the mass parameter of $(M,g,f)$, recalling that $f$ has no critical points near infinity. Moreover, recall from the end of \Cref{sec:proofs} that $\mathcal{H}_{p}^{c,d}=-\mathcal{F}_{p}^{c,d}(m)$ on $[0,1)$ for the Schwarzschild systems $(M^{n}_{m},g_{m},f_{m})$ of mass $m$. From \Cref{maintheoremp,maintheoremp2}, we also know that (suitable subsets of) the Schwarzschild systems are the only static vacuum systems satisfying this identity.

\begin{theorem}[Monotone functions]\label{thm:monotone}
Let $(M^n,g,f)$, $n\geq3$, be an asymptotically flat static vacuum system of mass $m\in\R$ with connected boundary $\partial M$. Assume that $f\vert_{\partial M} = f_0$ for a constant $f_{0}\in[0,1)\cup(1,\infty)$ and choose the unit normal $\nu$ to $\partial M$ pointing towards the asymptotic end. Let $F$ and $G$ be as in \Cref{maintheoremp} for some $p\geq p_{n}$ and  constants $c,d\in\R$ satisfying $c+d\geq0$ and $cf_{0}^{2}+d\geq0$. Then the function $\mathcal{H}^{c,d}_{p}\colon[f_{0},1)\cup(1,f_{0}]\to\R$ given by \eqref{eq:identitycurlyH} is well-defined, continuous, and monotonically increasing when $f_{0}\in[0,1)$ / monotonically decreasing when $f_{0}\in(1,\infty)$, with $\lim_{f\to1}\mathcal{H}_{p}^{c,d}(f)=-\mathcal{F}_{p}^{c,d}(m)$. Unless $c=d=0$, $\mathcal{H}^{c,d}_{p}\equiv-\mathcal{F}^{c,d}_{p}(m)$ on $[f_{0},1)$ or $(1,f_{0}]$ holds if and only if $(M,g)$ is isometric to a suitable piece of the Schwarz\-schild manifold $(M^{n}_{m},g_{m})$ of mass~$m$ and $f$ corresponds to the corresponding restriction of $f_{m}$ under this isometry. Finally, $m>0$ if $f_0\in[0,1)$ while $m<0$ for $f_0\in(1,\infty)$.
\end{theorem}

\begin{remark}[Geometric monotonicity of $\mathcal{H}^{c,d}_{p}$]\label{rem:monotonegeometric}
Before we move on, we would like to draw the readers' attention to the fact that no matter whether $f_{0}\in[0,1)$ or $f_{0}\in(1,\infty)$, the function $\mathcal{H}^{c,d}_{p}$ is monotonically increasing along the level sets of $f$ from $\partial M$ towards the asymptotic end. This is because for $f_{0}\in(1,\infty)$, $f$ decreases along its level sets from $\partial M$ towards the asymptotic end.
\end{remark}
	
\subsection{Comparison with the proof by Agostiniani and Mazzieri}\label{subsec:AM}
Let us now relate the functions $\mathcal{H}_{p}^{c,d}$ from \eqref{def:curlyH} to the functions $U_{p}$ and their derivatives introduced in~\cite{agostiniani}. In our notation, these functions are given by
	\begin{align}\label{def:Up}
		U_{p}(f)&\definedas\left(\frac{2m}{1-f^{2}}\right)^{\frac{(n-1)(p-1)}{n-2}}\int_{\Sigma_{f}}\Vert\nabla f\Vert^{p}\,dS,\\\label{def:Up'}
		U_p'(f)&=-(p-1)\left(\frac{2m}{1-f^{2}}\right)^{\frac{(n-1)(p-1)}{n-2}}\int_{\Sigma_{f}}\Vert\nabla f\Vert^{p-1}\left[\HH-\frac{2(n-1)f\Vert\nabla f\Vert}{(n-2)(1-f^2)}\right]dS
	\end{align}
for $f\in[f_0,1)\setminus\crit$ for $f_0\in[0,1)$ and $p\geq1$. From this {and \eqref{eq:identitycurlyH}}, one readily computes
	\begin{align}\label{def:UpH}
		U_p(f)&=\frac{\mu_p}{4(1-f_0^2)}\left[\frac{f^2-f_0^2}{1-f^2}\,\mathcal{H}_{p}^{-1,1}(f)-\mathcal{H}_{p}^{1,-f_0^2}(f)\right],\\\label{def:Up'H}
		U_p'(f)&=\frac{\mu_p f}{2(1-f^2)^2}\,\mathcal{H}_{p}^{-1,1}(f)
	\end{align}
on $[f_0,1)\setminus\crit$ for $\mu_p\definedas(p-1)(2m)^{\frac{(n-1)(p-1)}{n-2}}>0$ when $p\geq p_n$ and $f_0\in[0,1)$. 

It is shown in \cite[Theorem 1.1]{agostiniani} that $U_p$ is differentiable on $[f_0,1)$ for $p\geq3$ with derivative $U_p'$ and continuous on $[f_0,1)$ for $p\geq1$. Moreover, it is shown in \cite[Theorem 1.2]{agostiniani} that $U_p$ is differentiable with derivative $U_p'$ on $[f_0,1)\setminus f(\crit)$ for $p\geq p_n$. A fortiori, it follows from \Cref{thm:monotone} that the functions $U_{p}$ and $U_p'$ given by \eqref{def:UpH}, \eqref{def:Up'H} are well-defined as continuous functions on $[f_0,1)$ not only for $p\geq3$ but in fact for $p\geq p_{n}$. As there are only finitely many critical values of $f$ and as $U_{p}$, $U_{p}'$ are continuous also at critical values of $f$, the fundamental theorem of calculus implies that $U_{p}$ is continuously differentiable with derivative $U_{p}'$ on $[f_0,1)$ for $f_0\in[0,1)$ for all $p\geq p_n$. Moreover, as $\mathcal{H}^{-1,1}_p$ is monotonically increasing with limit $-\mathcal{F}^{-1,1}_p(m)=0$ as $f\to1$ by \Cref{thm:monotone}, we have $U_p'\leq0$ so that $U_p$ is montonically decreasing on $[f_0,1)$ for all $p\geq p_n$. Moreover, if $U_p'(f)=0$ for some $f\in[f_0,1)$, $f\neq0$, then $\mathcal{H}^{-1,1}_p(f)=0$ and hence by \Cref{thm:monotone} $(M,g,f)$ is isometric to a suitable piece of the Schwarz\-schild manifold $(M^{n}_{m},g_{m})$ of mass~$m$ and $f$ corresponds to the corresponding restriction of $f_{m}$ under this isometry. Finally, if $f_0=0$, $U_p'(0)=0$ is automatic from \eqref{def:Up'H} and one finds
\begin{align}
\begin{split}\label{eq:Up''}
U_p''(0)&\definedas \lim_{f\to 0+}\frac{U_p'(f)}{f}=\frac{\mu_p}{2}\lim_{f\to0+}\mathcal{H}^{-1,1}_p(f)=\frac{\mu_p}{2}\,\mathcal{H}^{-1,1}_p(0)\\
&=-\frac{(p-1)}{2}(2m)^\frac{(n-1)(p-1)}{n-2}\int_{\partial M}\Vert\nabla f\Vert^{p-2}\left[\scal_{\partial M}-\frac{4(n-1)\Vert\nabla f\Vert^2}{n-2}\right]dS,
\end{split}
\end{align}
where we have used continuity of $\mathcal{H}^{-1,1}_p$, the expression for the mean curvature in terms of the Hessian of the harmonic function $f$, the static vacuum equation \eqref{vacuum equation1}, the Gau{\ss} equation, and \Cref{rem:geodesic}. Also, it follows that $U_p''(0)\leq0$ because we have already established above that $\mathcal{H}^{-1,1}_p\leq0$. Finally, $U_p''(0)=0$ if and only if $\mathcal{H}^{-1,1}_{p}(0)=0$ if and only if $(M,g,f)$ is isometric to a suitable piece of the Schwarz\-schild manifold $(M^{n}_{m},g_{m})$ of mass~$m$ and $f$ corresponds to the corresponding restriction of $f_{m}$ under this isometry by \Cref{thm:monotone}. This can be re-expressed as follows.

\begin{corollary}[Monotonicity-Rigidity \`a la Agostiniani--Mazzieri]\label{coro:AM}
Items (ii) and (iii) of the Monotonicity-Rigidity Theorem \cite[Theorem 1.1]{agostiniani} hold for $p\geq p_n$.
\end{corollary}

Let us now discuss the case $f_0\in(1,\infty)$. To do so, let us extend \eqref{def:Up} as the definition of $U_p$ also to $(1,f_0]\setminus\crit$ for any $f_0\in(1,\infty)$ and any $p\geq p_n$, noting that the pre-factor of the integral is well-defined as $m<0$ and $f>1$ in this case. One directly computes that \eqref{def:UpH} gets replaced by 
\begin{align}\label{def:UpH1}
U_p(f)&=\frac{\mu_p}{4(f_0^2-1)}\left[\frac{f_{0}^2-f^2}{f^2-1}\,\mathcal{H}_{p}^{1,-1}(f)-\mathcal{H}_{p}^{-1,f_0^2}(f)\right]
\end{align}
with $\mu_p\definedas (p-1)(2\vert m\vert)^{\frac{(n-1)(p-1)}{n-2}}>0$. Next, using our above results, we can show continuous differentiability of $U_p$  for $f_0\in(1,\infty)$ and $p\geq p_n$ and compute its derivative $U_p'$ as follows: Recall from \Cref{sec:proofs} that
\begin{equation*}
\diver\left(\Vert\nabla f\Vert^{p-1} \nabla f\right)=(p-1)\Vert\nabla f\Vert^{p-3}\,\nabla^2f(\nabla f,\nabla f)
\end{equation*}
on $M\setminus\crit${. Exploiting \eqref{eq:Hlevel} and arguing as  in \Cref{sec:proofs,sec:inequalities}, in particular using the coarea formula and the adapted divergence theorem, we get
\begin{equation*}
(p-1)\int_{f_1}^{f_2}\int_{\Sigma_\tau}\Vert\nabla f\Vert^{p-1} \HH\, dS\,d\tau = \int_{\Sigma_{f_2}}\Vert\nabla f\Vert^p\,dS-\int_{\Sigma_{f_1}}\Vert\nabla f\Vert^p\,dS
\end{equation*}
for any $f_0\geq f_2>f_1>1$}. Taking the limit $f_1\to1$, this reduces to
\begin{equation*}
(p-1)\int_{1}^{f}\int_{\Sigma_\tau}\Vert\nabla f\Vert^{p-1} \HH\, dS\,d\tau = \int_{\Sigma_{f}}\Vert\nabla f\Vert^p\,dS=\left(\frac{1-f^2}{2m}\right)^{\frac{(n-1)(p-1)}{(n-2)}} U_p(f)
\end{equation*}
for every $f\in(1,f_0]$ by continuity of $U_p$. This can be re-expressed as 
\begin{equation*}
U_p(f)=(p-1)\left(\frac{2m}{1-f^2}\right)^{\frac{(n-1)(p-1)}{(n-2)}}\int_{1}^{f}\int_{\Sigma_\tau}\Vert\nabla f\Vert^{p-1} \HH\, dS\,d\tau
\end{equation*}
for all $f\in(1,f_0]$. Now, arguing as in \Cref{subsec:monotone}, the function $f\mapsto \int_{\Sigma_f}\Vert\nabla f\Vert^{p-1} \HH\, dS$ is continuous on $(1,f_0]$ which gives continuous differentiability of $U_p$ and 
\begin{equation}\label{eq:Up'neg}
U_p'(f)=(p-1)\left(\frac{2m}{1-f^{2}}\right)^{\frac{(n-1)(p-1)}{n-2}}\int_{\Sigma_{f}}\Vert\nabla f\Vert^{p-1}\left[\HH+\frac{2(n-1)f\Vert\nabla f\Vert}{(n-2)(1-f^2)}\right]dS
\end{equation}
for all $f\in(1,f_0]$ by the fundamental theorem of calculus. In particular, $U_p$ is continuously differentiable on $(1,f_0]$ for all $f_0\in(1,\infty)$ and all $p\geq p_n$ and \eqref{def:Up'H} gets replaced by
\begin{align}\label{def:Up'H1}
		U_p'(f)&=-\frac{\mu_p f}{2(f^2-1)^2}\,\mathcal{H}_{p}^{1,-1}(f)
	\end{align}
 as can be seen by a direct computation. Moreover, by \eqref{def:Up'H1} and $\mathcal{H}^{-1,1}_{p}(f)\to-\mathcal{F}^{1,-1}_{p}(m)=$ as $f\to1$ by \eqref{eq:Fmcd} and in view of \Cref{rem:monotonegeometric}, we have $U_{p}'\geq0$ on $(1,f_{0}]$. Again, just as in \Cref{rem:monotonegeometric}, this means that $U_{p}$ is monotonically decreasing from $\partial M$ to infinity. In analogy with \Cref{coro:AM}, we can summarize our findings as follows.

\begin{corollary}[Monotonicity-Rigidity \`a la Agostiniani--Mazzieri for $f_{0}\in(1,\infty)$]
Item (ii) of the Mono\-to\-ni\-city-Rigidity Theorem \cite[Theorem 1.1]{agostiniani} holds for $U_{p}$ given by \eqref{def:Up} on $(1,f_{0}]$ with derivative $U_{p}'$ given by \eqref{eq:Up'neg} for all $f_{0}\in(1,\infty)$ and all $p\geq p_n$, with the opposite inequality $U_{p}'\geq0$ on $(1,f_{0}]$.
\end{corollary}

Last but not least, we would like to point out that we have computed $U_{p}$ and $U_{p}'$ only from $\mathcal{H}_{p}^{c,d}$ using only the extremal values of $(c,d)$ (normalized to $\vert c\vert=1$) as in the proofs of \Cref{teoclassifica1,teoclassifica2}, see also \Cref{fig,fig2}. All other functions $\mathcal{H}^{c,d}_{p}$ are related to the extremal ones by
\begin{align}
\mathcal{H}^{c,d}_{p}&=\frac{cf_{0}^{2}+d}{1-f_{0}^{2}}\mathcal{H}^{-1,1}_{p}+\frac{c+d}{1-f_{0}^{2}}\mathcal{H}^{1,-f_{0}^{2}}_{p},\\
\mathcal{H}^{c,d}_{p}&=\frac{cf_{0}^{2}+d}{f_{0}^{2}-1}\mathcal{H}^{1,-1}_{p}+\frac{c+d}{f_{0}^{2}-1}\mathcal{H}^{-1,f_{0}^{2}}_{p}
\end{align}
for $f_{0}\in[0,1)$ and $f_{0}\in(1,\infty)$, respectively. These relations allow us to express all functions $\mathcal{H}^{c,d}_p$ by $U_p$ and $U_p'$, obtaining
	\begin{align}\label{eq:final relation}
		\mu_p\mathcal{H}_{p}^{c,d}(f)&=\frac{2(cf^2+d)(1-f^{2})}{f}U_{p}'(f)-4(c+d)U_p(f)
	\end{align}
	for all $f\in[f_0,1)\cup(1,\infty)$, in consistency with \eqref{def:UpH}, \eqref{def:Up'H} and \eqref{def:UpH1}, \eqref{def:Up'H1}, respectively. As a consequence, for $f_{0}=0$, one has
	\begin{align}\label{eq:final relation''}
		\mu_p\mathcal{H}_{p}^{c,d}(0)&=2dU_{p}''(0)-4(c+d)U_{p}(0),
	\end{align}
	in consistency with \eqref{eq:Up''}.
	
To summarize, our comparison of the monotone functions $\mathcal{H}^{c,d}_{p}$ and $U_{p}$ shows that our divergence theorem approach leads to (an extension of) the results of the monotone function approach by Agostiniani and Mazzieri~\cite{agostiniani}, in some sense lifting the monotonicity from a derivative to the function itself. This circumvents the conformal change to an asymptotically cylindrical picture as introduced in \cite{agostiniani}. In particular, working directly with the divergence theorem in the static system makes the analysis of the equality case simpler, avoiding the need to appeal to a splitting theorem.
	
\begin{corollary}[Relation between governing functionals]
Let $(M^{n},g,f)$, $n\geq3$, be an asymptotically flat static vacuum system of mass $m$ with connected boundary $\partial M$. Let $p\geq p_{n}$ and $c,d\in\R$, and suppose that $f\vert_{\partial M}=f_{0}$ for some $f_{0}\in[0,1)\cup(1,\infty)$. Then the functions $\mathcal{H}^{c,d}_{p}$ given by \eqref{def:curlyH} and $U_{p}$ given by \eqref{def:Up} are related by \eqref{eq:final relation} as well as by \eqref{def:UpH}, \eqref{def:Up'H} when $f_{0}\in[0,1)$ and by \eqref{def:UpH1}, \eqref{def:Up'H1} when $f_{0}\in(1,\infty)$. Here, $\mu_{p}=(p-1)(2\vert m\vert)^{\frac{(n-1)(p-1)}{n-2}}$. Moreover, if $f_{0}=0$, they satisfy the relation \eqref{eq:final relation''}.
\end{corollary}
	
	Consequently, our approach gives new proofs for all geometric inequalities for black hole horizons described in \cite[Section 2.2]{agostiniani} and for the Willmore-type inequalities for black hole horizons and other level sets of the lapse function $f$ described in \cite[Section 2.3]{agostiniani}. It also extends their results to $p\geq p_{n}$, to weaker asymptotic assumptions, and to $f_{0}\in(1,\infty)$. In particular, we would like to point out that the right-hand side inequality in \eqref{ineqA} is the Riemannian Penrose inequality. We hence in particular reprove the Riemannian Penrose inequality for asymptotically flat static vacuum systems with connected black hole inner boundary (but with the extra assumption~\eqref{CondR} necessary to conclude rigidity) under very weak asymptotic assumptions.	
	
	Last but not least, we would like to mention that it should be possible to extend the methods used in \cite{agostiniani} from $p\geq3$ to $p\geq p_{n}$ for $f_{0}\in[0,1)$ like it is successfully done by Agostiniani and Mazzieri in a different context in \cite{Agostiniani.2020}. Furthermore, it may be possible to modify the methods used in \cite{agostiniani} to handle the case $f_{0}\in(1,\infty)$, performing a different, but likely similar conformal transformation. Finally, it is conceivable that the methods used in \cite{agostiniani} may be extended to weaker asymptotic assumptions.

\subsection{Comparison with the proof by Nozawa, Shiromizu, Izumi, and Yamada}\label{subsec:Nozawa}
In \cite[Section 5]{Nozawa}, Nozawa, Shiromizu, Izumi, and Yamada devise a strategy to proving a black hole uniqueness theorem which turns out to coincide with \Cref{teoclassifica1}. To do so, they devise a Robinson style strategy of proof including a parameter $c$ arising as a power which should be understood as $c\asdefined p-1$, see below. In this section, we will first identify that their strategy of proof almost coincides with our strategy of proving \Cref{teoclassifica1}. After that, we will discuss which new insights our proof gives and which hurdles we overcame to make this strategy a rigorous proof. We will restrict to $f_{0}=0$ as \cite{Nozawa} only addresses black holes.

The strategy followed by Nozawa, Shiromizu, Izumi, and Yamada is to introduce a divergence identity just like \eqref{mainformulap} based on a vector field $J$. They then show that the divergence of $J$ is non-negative and can be related to the pointwise tensor norm of certain expressions related to a $(0,2)$-tensor $H$, in analogy with the proof of \Cref{maintheoremp}, see below. They then discuss why the vanishing of the $H$-tensor should imply isometry to a suitable Schwarzschild system, see below; their approach is somewhat reminiscent of parts of our proof of \Cref{thm:rigidity} but bears some issues, see below. To obtain the black hole uniqueness result \Cref{teoclassifica1}, they then suggest to apply the divergence theorem to $J$ on the static manifold $M$ and use the properties of the static black hole horizon and the asymptotic decay assumptions ($g$ and $f$ are assumed to be asymptotic to $g_m$ and $f_m$ for some $m\in\R$, including at least two derivatives) and then conclude as we do.

First, let us note that the vector field $Z$ from \eqref{def:Z} used in our approach in fact coincides with the vector field $J$ that was used in \cite{Nozawa} up to a constant factor --- despite its seemingly different definition. To see this, recall that $m>0$ and $0<f<1$ on $M$ in the black hole case. Then, adjusting to our notation, in particular choosing their exponent $c\asdefined p-1$, their vector field $J$ can be computed to be the $\frac{(n-2)(p-1)}{2}$-multiple of
\begin{equation}\label{def:J}
\overline{J}\definedas \frac{F_J(f)}{f}\|\nabla f\|^{p-1}\,\nabla\|\nabla f\|^{2}+G_J(f)\|\nabla f\|^{p-1}\,\nabla f,
\end{equation}
on $M$, where 
\begin{align*}
F_{J}(t) &\definedas\frac{a+b(1-t^2)}{(1-t^2)^{\frac{(n-1)(p-1)}{(n-2)}-1}},\\
G_{J}(t) &\definedas\frac{\frac{4(n-1)}{n-2}(a+b(1-t^{2})) - \frac{4}{p-1}a}{(1-t^2)^{\frac{(n-1)(p-1)}{(n-2)}}}
\end{align*}
for parameters $a,b,p\in\mathbb{R}$ with $p\geq p_{n}$ and variables $0<t<1$. Choosing our parameters $c\definedas-b$, $d\definedas a+b$, we find that $F_{J}=F$ and $G_{J}=G$ with $F$ and $G$ from \eqref{eq:F}, \eqref{eq:G} and hence $\overline{J}=Z$. Moreover, the conditions for positivity of $F_J$ identified in \cite{Nozawa}, $a\geq0$ and $a+b\geq0$, exactly coincide with our (black hole case) conditions $c+d\geq0$ and $d\geq0$. 

This of course informs us that the divergences of $\overline{J}$ and $Z$ must also coincide. To relate the divergence identity \cite[(5.12)]{Nozawa} to \eqref{mainformulap}, we spell out \cite[(5.12)]{Nozawa} in our notation, obtaining
\begin{align}\label{eq:divIdJ}
\diver \overline{J} &=\|\nabla f\|^{p-1}\,\frac{F_J(f)}{f}\left[\,\left\|S\right\|^2+ \frac{2((n-1)(p-1)-(n-2))}{n-1}\|\overline{H}\|^2\right]
\end{align}
for parameters $a,b\in\mathbb{R}$, where $S$ is given by
\begin{align}
\begin{split}
S(X,Y,Z)&\definedas \frac{1}{\|\nabla f\|^2}(X(f) H(Y,Z) - Y(f) H(X,Z))\\
&\quad\quad -  \frac{1}{n-1}(g(\overline{H},X) g(Y,Z)-g(\overline{H},Y) g(X,Z))
\end{split}
\end{align}
away from $\crit$ for $X,Y,Z\in\Gamma(TM)$ in terms of the $H$-tensor
\begin{align}
H &\definedas \nabla^2f - \frac{2}{n-2}\frac{f\|\nabla f\|^2}{1-f^2}g + \frac{2n}{n-2}\frac{f}{1-f^2}(df\otimes df)
\end{align}
and the vector field $\overline{H}$ given by
\begin{align}\label{def:Hbar}
\overline{H} &\definedas \frac{\nabla\|\nabla f\|^{-1}}{\|\nabla f\|^{-1}} - \frac{2(n-1)}{(n-2)}\frac{f\,\nabla f}{(1-f^{2})}=-\frac{H(\nabla f,\cdot)^{\#}}{\|\nabla f\|^{2}}
\end{align}
away from $\crit$. Knowing already that $F_J=F$, let us now relate $\overline{H}$ to \eqref{mainformulap3} and compare the tensor $S$ with the $T$-tensor, obtaining 
\begin{align}
\overline{H}&=-\frac{1}{2\|\nabla f\|^2} \left(\nabla\|\nabla f\|^2+\frac{4(n-1)}{(n-2)}\frac{f\|\nabla f\|^2\,\nabla f}{1-f^2}\right),\\\label{eq:ST}
S&=-\frac{(n-2)f}{(n-1)\|\nabla f\|^2}\, T
\end{align}
away from $\crit$. To find \eqref{eq:ST}, we have used \eqref{eq:CTW} and the corresponding identity \cite[(5.16)]{Nozawa}. This confirms that, away from $\crit$, the two divergence identities are in fact identical  and only expressed differently, with \cite{Nozawa} building upon the $H$-tensor and our approach working directly with the $T$-tensor. However, we would like to point out that $T$ and \eqref{mainformulap3} are well-defined on $\crit$ while $S$ and $\overline{H}$ are not.

Before we suggest geometric interpretations of the two different viewpoints of the identical divergence identity expressed as \eqref{eq:divIdJ} and \eqref{mainformulap}, respectively, let us quickly delve into the strategies of asserting rigidity. Our proof of \Cref{thm:rigidity} heavily relies on the local analysis of solutions of $T=0$ in \Cref{sec:Ttensor} in combination with our asymptotic assumptions. It does \emph{not} exploit \eqref{mainformulap3} which is not available when $p=p_n$, see also \Cref{rem:simplify}. In contrast, Nozawa, Shiromizu, Izumi, and Yamada~\cite{Nozawa} first use $\overline{H}=0$ to obtain the functional relationship \cite[(5.11)]{Nozawa} which is equivalent to \eqref{mainformulap3} away from $\crit$. Using this functional relationship and $H=0$, they then suggest to proceed very similarly in spirit as we do in the proof of \Cref{thm:locally} with strong simplifications as most cases we study are excluded by the functional relationship \cite[(5.11)]{Nozawa}, see again \Cref{rem:simplify}. However, as we pointed out before, $\overline{H}=0$ is not readily explicitly deducible when $p=p_n$, a case also included in \cite{Nozawa}. Also, it is not discussed explicitly in \cite{Nozawa} how $H=0$ follows from $S=0$ and $\overline{H}=0$. It is thus worthwhile to investigate the relationship between $S$, $H$, and $\overline{H}$ more closely. To see that the $H$-tensor vanishes when $S=0$ and $\overline{H}=0$, we compute
\begin{align*}
0&=S(X,\nabla f,Y)=\frac{1}{\|\nabla f\|^2}\left(X(f)\underbrace{H(\nabla f,Y)}_{-\|\nabla f\|^2\langle\overline{H},Y\rangle=0}-\|\nabla f\|^2 H(X,Y)\right)=-H(X,Y)
\end{align*}
for all vector fields $X,Y\in\Gamma(TM)$ to conclude that $H=0$ when $S=0$, away from $\crit$. We would like to point out that this argument is very similar to the one given in the proof of \Cref{lemmaT}. In fact, it turns out that  
\begin{align*}
H&=\frac{f}{\|\nabla f\|^2}\left(\|\nabla f \|^2 \Ric + \frac{\lambda \|\nabla f\|^2}{n-1} g-\frac{n\lambda}{n-1} df\otimes df\right)
\end{align*}
follows from $T=0$ via \eqref{mainformulap3}, with $\lambda$ denoting the eigenvalue for the eigenvector $\nabla f$ of $\Ric$ with the right hand side coming from \eqref{gEinstein} as can be seen by computing $\lambda$ via \eqref{mainformulap3}. This is another way of seeing that $S=0$ implies $H=0$ when assuming $\overline{H}=0$ (or equivalently \eqref{mainformulap3}).  This is a purely local result. However, it does \emph{not} locally follow from $S=0$ that $H=0$ \emph{without} assuming $\overline{H}=0$ --- not even for $n=3$ --- which can be seen as follows. From \Cref{thm:locally}, we know that $S=T=0$ implies that each regular point $p\in M$ has an open neighborhood $p\in V\subseteq M$ such that $(V,g\vert_{V},f\vert_{V})$ has Type 1, 2, 3, or 4. In particular, \Cref{coro:locally} gives a full characterization of $\lambda\vert_{V}$. Thus, supposing that $\overline{H}=0$ and thus $\lambda=-\frac{2(n-1)\vert\nabla f\vert^{2}}{(n-2)(1-f^{2})}$ on $M$, excludes Types 1-3 and enforces Type 4 with $a=1$. Note that this can be restated as saying that the vanishing of $\overline{H}$ implies that a system $(V,g\vert_{V},f\vert_{V})$ as in \Cref{thm:locally} has to be a suitable piece of a quasi-Schwarzschild manifold, see \Cref{rem:quasi,rem:simplify}. Note furthermore that even the systems of Type 4 with $a=1$ are not spherically symmetric unless the Einstein manifold $(\Sigma,\sigma)$ is the standard round sphere, see \Cref{rem:quasi}. In other words, any system of Types 1-3 or of Type 4 with $a\neq1$ is a counter-example to concluding vanishing of $\overline{H}$ and thus of $H$ from vanishing of $S$. 

As discussed above, this insight becomes relevant when choosing the threshold parameter value $p-1=c=1-\frac{1}{n-1}=p_{n}-1$, included in the analysis in \cite{Nozawa}, when the divergence \eqref{eq:divIdJ} vanishes as one can then \emph{not} conclude that $\overline{H}=0$. The threshold case thus needs some more careful treatment as we have given it in \Cref{sec:Ttensor,sec:rigidity}. Note that for $n=3$, the rigidity for the treshold value $c=p-1=\frac{1}{2}$ was already handled in \cite{robinsonetal}.

Let us now turn to some geometric considerations regarding the various tensor and vector fields discussed in this section. First, from \eqref{eq:ST} and the second part of \eqref{def:Hbar}, we learn that $T$ is fully determined by $H$, as was already implicitly observed and exploited in \cite{Nozawa}. However, as argued above, in order to recover (vanishing of) $H$ from (vanishing of) $T$, we need the functional relationship \eqref{mainformulap3} which takes the form $\overline{H}=0$ in \cite{Nozawa}. The geometric interpretation of $T$ and thus of $S$ is given by the definition \eqref{tensorT} of $T$. The geometric interpretation of $H$ can be understood in multiple ways: First (see also \cite[(5.9)]{Nozawa}), $H$ can be derived from knowing that, in spherical symmetry with radial direction $\nabla f$, one can express
\begin{equation}\label{eq:ansatz}
\nabla^2f=\alpha g+\beta df\otimes df
\end{equation}
for suitable $\alpha, \beta\in C^\infty(M)$ because $\nabla^2f$ will vanish in tangential-normal directions along level sets of $f$. Next, assuming \eqref{mainformulap3} or $\overline{H}=0$ and plugging $\nabla f$ into \eqref{eq:ansatz}, we obtain  
\begin{equation}
\alpha+\beta=\frac{2(n-1)f}{(n-2)(1-f^2)}.
\end{equation}
On the other hand, using the static equation \eqref{vacuum equation2}, we find 
\begin{equation}
n\alpha+\|\nabla f\|^2\beta=0.
\end{equation}
Taken together, this gives
\begin{equation}
\nabla^2f=\frac{2f\|\nabla f\|^2}{(n-2)(1-f^2)} g+\beta df\otimes df-\frac{2nf}{(n-2)(1-f^2)}df\otimes df
\end{equation}
which explains the definition of $H$ as measuring the deviation from spherical symmetry, subject to \eqref{mainformulap3} or $\overline{H}=0$. The second interpretation of $H$ (see also \cite[(2.50)]{Nozawa}) is exploiting the fact that $rf_m(r)\partial_r$ is a conformal Killing vector field in the Schwarzschild system (or in other words in a spherically symmetric system satisfying \eqref{mainformulap3} or $\overline{H}=0$. Thus,
\begin{align}\label{def:zeta}
\zeta \definedas \dfrac{\nabla f}{(1-f^2)^{n/(n-2)}}
\end{align}
is a conformal Killing vector field in the Schwarzschild case, with
\begin{align}
\mathcal{L}_\zeta g-\frac{2}{n}(\diver\zeta) = \dfrac{2}{(1-f^2)^{\frac{n}{n-2}}}H
\end{align}
as can be seen from a straightforward computation. Hence, $H=0$ if and only if $\zeta$ as defined in \eqref{def:zeta} is a conformal Killing vector field in $(M,g)$. Yet another geometric interpretation of $H$ (see \cite[page 22]{Nozawa}) is that
\begin{equation}
\overline{\Ric}=\frac{1-f}{f(1+f)}H,
\end{equation}
where $\overline{\Ric}$ denotes the Ricci curvature tensor of the conformally transformed metric 
\begin{equation}
\overline{g}=\left(\frac{1+ f}{2}\right)^{\frac{4}{n-2}}g
\end{equation}
appearing in the Bunting--Masood-ul-Alam proof \cite{bunting} and its higher dimensional versions \cite{gibbons,Hwang1}. In summary, no matter which strategy of proof we follow (i.e., via $T$ as in this paper or via $H$ as in \cite{Nozawa}), we are heavily exploiting conformal flatness -- with the conformal factor expressed as a function of $f$ -- of Schwarzschild as well as its spherical symmetry made explicit in the functional relationship \eqref{mainformulap} or in the vanishing of~$\overline{H}$.

{Let us close by highlighting} that despite the similarities in the proofs of static vacuum black hole uniqueness subject to the scalar curvature bound \eqref{CondR} presented here and in \cite{Nozawa}, this paper adds the necessary and subtle analytic details needed to handle critical points of $f$, notably when asserting integrability of the divergence $\mathcal{D}$ and applicability of the divergence theorem in this weak regularity scenario, see \Cref{sec:inequalities}. This particularly applies to the case $p\in[p_n,3)$ when $\mathcal{D}$ is not continuous across $\crit$. We also demonstrate that much lower decay assumptions (namely asymptotic flatness with decay rate $\tau=0$) suffice to conclude. And of course, we add the equipotential photon surface uniqueness results of \Cref{teoclassifica2} with $f_0\in(0,1)\cup(1,\infty)$.

\bibliographystyle{amsalpha}
\bibliography{Robinson}
\end{document}